\documentclass[reqno]{amsart}
\usepackage{amssymb,amscd,verbatim,graphics,color,longtable}

\begin{document}

\setcounter{secnumdepth}{2}


\newcommand \und[2]{\underset{#1} {#2}}
\newcommand \fk[1]{{{\mathfrak #1}}}
\newcommand \C[1]{{\mathcal #1}}
\newcommand \ovl[1]{{\overline {#1}}}
\newcommand \ch[1]{{\check{#1}}}
\newcommand \bb[1]{{\mathbb #1}}
\newcommand \sL[1]{{^L {#1}}}
\newcommand \sT[1]{{^t {#1}}}
\newcommand \sV[1]{{^\vee{#1}}}

\newcommand \ti[1]{{\tilde #1}}
\newcommand \wti[1]{{\widetilde {#1}}}
\newcommand \un[1]{\underline{#1}}
\newcommand \unb[2]{\underset{#1}{{\underbrace{#2}}}}
\newcommand \ovb[2]{\overset{#1}{\overbrace{#2}}}

\newcommand\fg{\mathfrak g}
\newcommand\fh{\mathfrak h}
\newcommand\fa{\mathfrak a}
\newcommand\fz{\mathfrak z}
\newcommand\fm{\mathfrak m}

\newcommand \bA{{\mathbb A}}
\newcommand \bC{{\mathbb C}}
\newcommand \bF{{\mathbb F}}
\newcommand \bH{{\mathbb H}}
\newcommand \bR{{\mathbb R}}
\newcommand \bZ{{\mathbb Z}}

\newcommand\cha{{\check \alpha}}
\newcommand\chb{{\check \beta}}

\newcommand\one{1\!\!1}

\newcommand\CA{{\C A}}
\newcommand\CB{{\C B}}
\newcommand\CH{{\C H}}
\newcommand\CI{{\C I}}
\newcommand\CO{{\C O}}
\newcommand\CS{{\C S}}
\newcommand\CW{{\C W}}
\newcommand\CX{{\C X}}
\newcommand\CG{{\C G}}
\newcommand\CU{{\C U}}

\newcommand\CCI{{\C C(\CI)}}

\newcommand \LA{{\sL A}}
\newcommand\Lfa{{\sL {\fk a}}}
\newcommand \LG{{\sL G}}
\newcommand \LM{{\sL M}}
\newcommand \LB{{\sL B}}
\newcommand \Lg{{\sL{\fk g}}}

\newcommand\Ls{{\it C-symbol~ }}
\newcommand\DS{{\it Discrete series~ }}
\newcommand\Tr{{\it Tempered Representation~ }}
\newcommand\SpC{{\it Springer Correspondence~ }}
\newcommand\Sps{{\it S-symbol~ }}
\newcommand\LKT{{\it lowest K-type~ }}
\newcommand\LKTs{{\it lowest K-types~ }}
\newcommand\St{{\it Steinberg Representation~ }}
\newcommand\Sts{{\it Steinberg Representations~}}
\newcommand\sde[2]{{\ ^#1{#2}}}

\newcommand\ie{{\it i.e.}}
\newcommand\cf{{\it cf.~ }}
\newcommand\eg{{\it e.g.}}

\newcommand\ep{{\epsilon}}
\newcommand\la{{\lambda}}
\newcommand\om{{\omega}}
\newcommand\ome{{\omega}}
\newcommand\al{{\alpha}}
\newcommand\sig{{\sigma}}

\newcommand \vG{{\check G}}
\newcommand \vM{{\check M}}
\newcommand\oX{\ovl{X}}
\newcommand \vg{{\check \fk g}}
\newcommand \bXO{{\overline{X}(\C O)}}
\newcommand\bX{{\ovl{X}}}

\newtheorem*{theorem}{Theorem}
\newtheorem*{theorem 2}{Theorem 2}
\newtheorem*{corollary}{Corollary}
\newtheorem*{lemma}{Lemma}
\newtheorem*{proposition}{Proposition}
\newtheorem*{remark}{Remark}
\newtheorem*{definition}{Definition}
\newtheorem*{example}{Example}

\newcommand\Hom{\operatorname{Hom}}
\newcommand\Ind{\operatorname{Ind}}

\numberwithin{equation}{subsection}


\begin{title}[On unitary unipotent representations of $p$-adic groups]{On
    unitary unipotent representations of $p$-adic groups and affine Hecke algebras with unequal
    parameters}\end{title}

\author{Dan Ciubotaru}
\address[Dan Ciubotaru]{Department of Mathematics\\ University of
  Utah\\Salt Lake City, UT 84112}
\email{ciubo@math.utah.edu}

\date{\today}


\begin{abstract}
We determine the unitary dual of the geometric graded Hecke algebras with
{unequal} parameters which
appear in Lusztig's classification of unipotent representations for
{exceptional} $p$-adic groups. The largest such algebra is of type
$F_4.$ Via the Barbasch-Moy correspondence of unitarity
applied to this setting,  this
is equivalent to the identification of the corresponding unitary
unipotent representations with real
central character of the $p$-adic groups. In order for this
correspondence to be applicable here, we show (following Lusztig's
geometric classification, and Barbasch and Moy's original argument)
that the set of tempered modules with real central character for a
geometric graded Hecke algebra is linearly independent when restricted
to the Weyl group.
\end{abstract}

\maketitle

\setcounter{tocdepth}{1}

\begin{small}
\tableofcontents 
\end{small}


\section{Introduction}\label{sec:0}

\subsection{}The category of representations with unipotent cuspidal
support for the rational forms inner to a 
split $p$-adic group of adjoint type is equivalent to a
direct product of categories of affine Hecke algebras with
unequal parameters (\cite{L6}).  We are interested in the problem of
identifying the unitary representations in this category.

Let us briefly recall the Deligne-Langlands-Lusztig correspondence first. Let $\bF$ denote a finite
extension of the $p$-adic numbers, with residue field $\bF_q.$ Let
$\mathcal G$ denote a connected adjoint quasi-simple algebraic group
defined over the algebraic completion $\overline \bF.$ Let $\Psi_s$ be the inner class of
$\mathcal G$ which contains the split $\bF$-form of $\C G.$ For every
$\C G(\bF)\in \Psi_s,$ one defines the category of unipotent
representations $\mathsf{Uni}(\C G(\bF))$ as follows. For every
parahoric subgroup $J$ of $\C G(\bF)$, let $\overline J$ denote the
reductive quotient (a reductive group over $\bF_q)$. For every
unipotent cuspidal representation $\overline \rho$ of $\overline J$ (in the sense of
Deligne-Lusztig), let $\rho$ be the representation obtaining by
pulling back to $J.$ (The unipotent cuspidal representations appear
very rarely.) Then let $\mathsf{Uni}_{(J,\rho)}(\C G(\bF))$ denote the
category of admissible representations $\pi$ of $\C G(\bF)$ such that
$\Hom_J[\pi:\rho]$ is nonzero and generates $\pi.$ Set $\mathsf{Uni}(\C
G(\bF))=\sqcup_{(J,\rho)}\mathsf{Uni}_{(J,\rho)}(\C G(\bF)),$ and
$\mathsf{Uni}(\C G)=\sqcup_{\C G(\bF)\in\Psi_s}\mathsf{Uni}(\C
G(\bF)).$ This is the ``arithmetic side'' of the correspondence.

For every $(J,\rho)$ that appears, set $\CH(\C G(\bF),J,\rho)=\operatorname{End}
(\Ind _J^{\C G(\bF)}(\rho)).$ This is an affine Hecke algebra with
unequal parameters. 
The prototype is the subcategory of
representations generated by their vectors fixed under an Iwahori subgroup, that is, 
$\mathsf{Uni}(\C G(\bF),I,triv),$ where $I$ is an Iwahori
subgroup. The irreducible objects in this category are precisely the
subquotients of the unramified principal series of $\C G(\bF).$ By
the fundamental results of Borel and Casselman, there is an
equivalence of categories between $\mathsf{Uni}(\C
G(\bF),I,triv)$ and $\CH(\C G(\bF),I,triv)$-mod given by the functor
$\pi\mapsto \pi^I:=\Hom_I[\pi:triv].$ In general, by the
results of Morris and Moy-Prasad, there is an equivalence of
categories between  $\mathsf{Uni}_{(J,\rho)}(\C G(\bF))$ and $\CH(\C
G(\bF),J,\rho)$-mod given by $\pi\mapsto \Hom_J[\pi:\rho].$

\smallskip

On the ``geometric side'', one considers the complex simply-connected group $G$ with
 root datum dual to $\C G.$ Kazhdan-Lusztig in the case of equal
parameters, and Lusztig in general, gave a geometric construction of
affine Hecke algebras, respectively affine graded Hecke algebras
starting from $G.$  Let $L$ be a quasi-Levi subgroup of $G$ (i.e, the
centralizer of a semisimple element in $G$), with Lie
algebra $\fk l,$ $\C C$ a nilpotent class in $\fk l$ which admits a
cuspidal local system $\C L.$ (These are very rare.)  Then one
attaches an affine Hecke algebra $\C H(G,L,\C C,\C L)$ with unequal
parameters, whose representation theory is classified in terms of
perverse sheaves on certain graded nilpotent orbits of $\fg$. The truly
remarkable realization
of \cite{L6} is that the two collections of affine Hecke algebras
$\{\CH(\C G(\bF),J,\rho): G(\bF)\in\Psi_s, (J,\rho)\}$ and $\{\C
H(G,L,\C C,\C L): (L,\C C,\C L)\}$ coincide, and via this
identification, one translates the geometric classification to
$\mathsf{Uni}(\C G).$

\begin{theorem}[\cite{L6}]\label{t:dll} There is a one to one correspondence
  between $\mathsf{Uni}(\C G)$ and $G$-conjugacy classes in
\begin{equation}\label{eq:dll}
\{(s,e,\psi): s\in G \text{ semisimple}, e\in \fg, Ad(s)e=qe,
\psi\in \widehat{A(s,e)}\}.
\end{equation}
\end{theorem}
Note that $e\in \fg$ is necessarily nilpotent. In (\ref{eq:dll}),
$A(s,e)$ denotes the group of components of the centralizer in $G$ of
$s$ and $e.$ 

\medskip

\noindent{\bf Remarks.} 

\noindent a) We
emphasize that the left hand (arithmetic) side of the correspondence
is a union over all rational forms of $\C G$ inner to the split
one. 

\noindent b) An individual piece $\mathsf{Uni}(\C G(\bF))$, corresponding to a
fixed rational form, is distinguished by the action of $\psi\in
\widehat{A(s,e)}$ on the center $Z(G).$  For example, the split form
$G(\bF)^s$ is parameterized by those $\psi$ such that $\psi|_{Z(G)}=triv.$ The
Iwahori-spherical subcategory $\mathsf{Iw}(\C G(\bF)^s)$  of the split form is parameterized by the
representations $\psi$ of Springer type, i.e., those which appear in
the natural action of $A(s,e)$ on the Borel-Moore homology of the
variety $\C B_{s,e}$ of Borel subgroups of $G$ containing both $s$ and
$e$ (\cite{KL}). The trivial $A(s,e)$-representation is of Springer type.

\noindent c) By \cite{Re2} (also \cite{BM4} in the Iwahori-spherical case), the representations in
$\mathsf{Uni}(\C G(\bF)^s)$ which are generic in the sense that they admit
Whittaker models, are parameterized in theorem \ref{t:dll} by
$(s,e,\psi)$, where $e$ is in the unique open orbit of $Z_G(s)$ on
$\fg_q:=\{x\in\fg: Ad(s)x=qs\}$, and $\psi$ is trivial. Here $Z_G(s)$
denotes the centralizer of $s$ in $G.$ In particular,
these representations are all Iwahori-spherical.

\noindent d) The ($K$-)spherical representations of $\C G(\bF)^s$ correspond
to $(s,e,\psi)$ with $e=0$ and $\psi=triv.$ They are Iwahori-spherical
too. The Iwahori-Matsumoto involution interchanges spherical and
generic representations and by \cite{BM1,BM2}, it preserves unitarity.  

\noindent e) A first calculation we present (in section \ref{sec:4})
as an example  concerns unitary representations for $\C
G=E_6.$ The dual $G$ is also type $E_6$, simply connected and
$Z(G)=\bZ/3.$ Let $E_6^{nqs}$ denote any one of the two non-quasisplit form of type
$E_6$ (denoted $^3E_6$ in the tables of \cite{Ti}).  On the arithmetic
side, $\mathsf{Iw}(E_6^{nqs})$ is the category of Iwahori-spherical
representations of $E_6^{nqs}.$  On the geometric side, by \cite{L5}, there exist two
cuspidal local systems for the Levi $L=2A_2$ in $E_6$ (notation as in
\cite{Ca}), and the geometric Hecke algebras are isomorphic. The two
cuspidal local systems (actually representations of the component
group) are distinguished by how they act on $Z(G)$.
By \cite{L6}, the geometric Hecke algebra and the Iwahori-Hecke are
identical, of type $G_2$ with certain unequal parameters
(see the tables in \cite{Ti}). At the level of geometric
parameters (\ref{eq:dll}), the component group representations $\psi$
which appear, act on $Z(G)$ like a fixed character of $\bZ/3$ of order $3$.

\noindent f) The
hardest calculation in this paper refers to the case of unitary
representations for $\mathcal G=E_7.$
The dual $G$ is also of type $E_7$, simply connected. Then $Z(G)=\bZ/2,$
and the two representations of $Z(G)$ correspond to the two rational
forms in the split inner class. Let us denote by
$E_7^{ns}$ the nonsplit form. On the arithmetic side,
$\mathsf{Iw}(E_7^{ns})$ is the category of Iwahori-spherical
representations of $E_7^{ns}.$ The corresponding Iwahori-Hecke algebra
is of type $F_4$ with certain unequal parameters  On the geometric side, there exists a cuspidal
local system for the Levi $L=(3A_1)''$  in
$E_7.$ The affine Hecke algebra constructed from it has the same
presentation with generators and relations as the Iwahori-Hecke one
for $E_7^{ns}.$ At the level of geometric parameters (\ref{eq:dll}),
the component group representations $\psi$ which are nontrivial on
$Z(G)$ appear.  The nilpotent orbits of $E_7$ which allow such
representations are in figure \ref{3A1nil}.

\noindent g) By changing short roots with long roots in an affine
Hecke algebra which is not simply-laced, and scaling the unequal parameters accordingly, one
obtains isomorphic Hecke algebras. Using this isomorphism, the same
Hecke algebra as in e) controls two subcategories of unipotent
representations for $\C G(\bF)=E_8^s$ (split) and a parahoric of type
$E_6$, while the Hecke algebra in f) controls  a subcategory for $\C
G(\bF)=E_8^s$ and a parahoric of type $D_4.$ The weight structure of
the discrete series modules for these latter Hecke algebras, and their
formal degrees, were calculated in \cite{Re1}. We verified that the dimensions of the
discrete series, which we obtained by finding their $W$-structure,
agree with the dimensions listed in \cite{Re1}. Note also that by the
Barbasch-Moy transfer of unitarity (see section \ref{sec:0.2}) one
can obtain interesting correspondences between unitary representations
whenever two categories of unipotent representations are controlled by
the same affine Hecke algebra.

\subsection{}\label{sec:0.2}
The Iwahori-Hecke algebra is a convolution algebra of complex-valued functions, and
therefore has a natural $*$-operation: $f\mapsto f^*,$
$f^*(g):=\overline {f(g^{-1})}.$ One defines hermitian and unitary
modules with respect to this $*$-operation.
The approach to the determination of Iwahori-spherical unitary representations of a
split $p$-adic  group
via a reduction to Iwahori-Hecke algebra modules with real
central (``infinitesimal'') character was introduced in \cite{BM1,BM2}. It
is based on the geometric classification, and the use of
the graded Hecke algebra of \cite{L1}.

Let us recall succinctly the machinery of Barbasch-Moy's preservation
of unitarity under the real central character assumption. The central characters of $\CH(G,I,triv)$ are in one to
one correspondence with $G$-conjugacy classes of semisimple elements
$s$ in $G$. This is the semisimple element $s$ appearing in the
correspondence (\ref{eq:dll}). We say that a central character $s$ is
real if the elliptic part of $s$ in the polar decomposition is central
in $G$. The element $s$ is sometimes referred to as the infinitesimal
character of the representation.

\begin{theorem}[\cite{BM1}]\label{t:bm} In the equivalence of categories $\pi\mapsto
  \pi^I$ between
  $\mathsf{Iw}(\C G(\bF)^s)$ and $\CH(G,I,triv)$-mod, hermitian and
  unitary modules with real infinitesimal character correspond, respectively. 
\end{theorem}

One of the basic ingredients for the proof is the signature character
of a representation introduced in \cite{V1}. If $\pi$ a
hermitian representation in $\mathsf{Iw}(\C G(\bF)^s)$, let
$\Sigma(\pi)$ denote its signature character, and $\theta_K(\pi)$ its $K$-character (see
\cite{BM1}). Similarly, if $\overline \pi$ is a module in
$\CH(G,I,triv)$-mod, let $\Sigma(\overline\pi)$ denote its signature
character, and $\theta_W(\overline\pi)$ its $W$-character. (Recall that $\CH(G,I,triv)$ has the subalgebra of functions
supported on $K$, which is isomorphic with the Weyl group $W$ of $G.$)
The proof that $\pi^I$ unitary implies $\pi$ unitary (the other
implications are straightforward) in \cite{BM1} has the following steps:

\begin{enumerate}
\item There exist tempered representations $\pi_1,\dots,\pi_\ell$ in
  $\mathsf{Iw}(\C G(\bF)^s)$ such that 
\begin{equation}\label{eq:theta}\Sigma(\pi)=(\sum_i
  a_i\theta_K(\pi_i),\sum_i
  b_i\theta_K(\pi_i)),
\end{equation}
 for some integers $a_i,b_i.$ (This is proved
  by induction on the length of the Langlands parameter, and uses the
  fact that $\mathsf{Iw}(\C G(\bF)^s)$ is closed under taking
  subquotients.)
\item A module $\pi'$ in $\mathsf{Iw}(\C G(\bF)^s)$ is tempered if and
  only if $(\pi')^I$ is tempered in  $\CH(G,I,triv)$-mod. (One defines
  temperedness via Casselman's criterion, so the notion makes sense in
  both settings.)
\item \begin{equation}\label{eq:theta2}\Sigma_W(\pi^I)=(\sum_i
  a_i\theta_W(\pi_i^I),\sum_i
  b_i\theta_W(\pi_i^I) ),
\end{equation} for the same $a_i,b_i$ as in
  (\ref{eq:theta}). (One needs here not only that the irreducibles
  correspond in the equivalence of categories, but also the fact that
  the standard modules correspond.)
\item The set $\{\theta_W(\overline\pi):\overline\pi \text{ is
    tempered with real central character}\}$ is linearly independent
  in the Grothendieck group of $W,$ in fact it gives a basis. (This proof is based on the
  geometric realization of the $W$-structure of tempered modules and
  the results of Borho-MacPherson.)
\item Then one sees that if $\pi^I$ is unitary, $\Sigma(\pi^I)$
  being positive definite, it implies by (4) that $b_i=0,$ for all
  $i,$ and so $\Sigma(\pi)$ is positive definite as well.
\end{enumerate}

The essential step in this correspondence is (4). As stated, it does
not hold true if one drops the ``real central character''
assumption. Even if one considers the weaker (but sufficient) case of
tempered modules whose central character have a fixed elliptic part $s_e$, this statement
is still false in general. The subject of \cite{BM2} is to show that, however, one still has that, for the
tempered modules (with a fixed elliptic part $s_e$) which appear in
the equation (\ref{eq:theta2}), $\sum_i
  b_i\theta_W(\pi_i^I)=0$ implies $b_i=0.$ For this one uses (a
  variant of) the
  reduction to the graded Hecke algebra (\cite{L1}).

\smallskip

 In this setting, the spherical (equivalently, the generic
Iwahori-spherical) unitary dual for 
all split $p$-adic groups is determined: for classical types in
\cite{BM3} and \cite{Ba}, for types $G_2,F_4$ in \cite{Ci}, and for
types $E$ in \cite{BC1}. By different methods, the complete unitary dual for
types $GL(n)$ and $G_2$ had been previously obtained in \cite{Ta} and
\cite{Mu}, respectively. 

For non-generic modules (of types other than $A$ or $G_2$) less is
known: the unitary dual with real central character for the graded
Hecke algebra with equal parameters is known for $F_4$ (\cite{Ci}) and
$E_6$ (\cite{Ci2}). It is not known for classical groups, except in
some small rank cases.  

\smallskip

When one attempts to extend theorem \ref{t:bm} 
to the categories $\mathsf{Uni}(\C G(\bF),J,\rho)$, the first thing to
prove is item (4) above: the linear independence of $W$-characters
 for tempered modules with real central character in the affine Hecke
 algebra with unequal parameters appearing on the geometric side. This
 is discussed in section \ref{sec:3} of this paper. We work, as we
 may since the central character is assumed real, in the equivalent
 setting of the graded Hecke algebra. 

\subsection{}In \cite{BC1}, one proposed a direct method of matching signatures of hermitian
forms for irreducible modules of a graded Hecke algebra with equal parameters $\bH$ with
signatures of hermitian forms on spherical modules of smaller graded Hecke
algebras $\bH(\CO)$ constructed from root systems of centralizers of
nilpotent complex orbits $\CO.$ This method proved particularly effective in
the determination of the unitary generic modules, but it can be 
applied for the signatures of non-generic modules as well.  

In this paper, we apply this idea to the setting of geometric Hecke algebras $\bH$
with {unequal parameters} (\cite{L2,L3,L4}).  As an
application, we determine the unitary dual of certain Hecke algebras of
types $G_2$ and $F_4$, with unequal parameters, which appear in
\cite{L6}, in the
classification of unipotent representations of $p$-adic groups. 

We present next the philosophy of matching signatures and unitary duals. The geometric classification of $\bH$ is controlled by a
complex group $G$, with Lie algebra $\fg$,  and Cartan subalgebra
$\fk t,$ a Levi subgroup $L$ with Lie algebra $\fk l$,  and a cuspidal
local system $\C L$ supported on some nilpotent orbit $\C C$ of $\fk l$ (see
section 3). Note that the rank of $G$
 may be larger than that of 
$\bH$, e.g., the most interesting example in this paper is an algebra of type $F_4$ parameterized by
$G$ simply connected of type $E_7.$ Let $V$ be a simple
$\bH$-module. The graded Hecke algebra version of the geometric
classification (\ref{eq:dll}) says that $V$ is parameterized by a triple $(s,\CO_e,\psi),$
where $s$ is a (conjugacy class of a) semisimple element in $\fk t$,
$\CO_e$ is a $Z_G(s)$-orbit on $\fg_1=\{x\in\fg: [s,x]=x\},$
and $\psi$ is a certain representation of the component group $A(s,e)$.  The orbit $\C C$ is
necessarily distinguished in $\fk l$ (in the sense of Bala-Carter),
and there are restrictions about
which $\psi$ can appear depending on the cuspidal local system $\C
L.$ 
The nilpotent representative $e$ can be completed to a
Jacobson-Morozov triple $\{e,h,f\}$ of $\CO=G\cdot \CO_e,$ such that
$h\in \fk t.$  Let $Z(\CO)$ denote the centralizer in $G$ of
$\{e,h,f\}$. It is a reductive group, possibly disconnected. Let
$A(\CO)$ denote its group of components, and $\fz(\CO)$ its Lie
algebra. Then $s$
can be written as \begin{equation}\label{eq:shnu}s=\frac 12 h+\nu,\end{equation} where $\nu$ is a semisimple
element of $\fz(\CO).$ 

Now fix a nilpotent orbit $\CO$ in $\fg$, together with  a
representation $\phi$ of the component group $A(\CO)$. The requirement
is that 
the pair $(\CO,\phi)$ appears in the generalized
Springer correspondence (\cite{L5}), and so it has a representation $\mu_0$ of the
Weyl group $W$ defining $\bH$ attached to it. This is not the Weyl
group of $G,$ but rather $W=N_G(L)/L.$ (The fact that this is a
Coxeter group is due to the very particular form $L$ must have in
order to allow a cuspidal local system.) Let us denote by $\CU(\CO,\phi)$
the set of unitary $\bH$-modules with {\it lowest $W$-type} $\mu_0$, in the
sense of sections \ref{sec:3.4}-\ref{sec:3.6}. The triple $(s,\CO_e,\psi)$ 
parameterizing a module $V$ with lowest $W$-type $\mu_0$ must satisfy
$\CO_e\subset\CO,$ and $\Hom_{A(s,e)}[\phi:\psi]\neq 0.$ 

We attach a Hecke
algebra $\bH(\fz(\CO),c(\phi))$ with the root system that of $\fz(\CO)$,
and certain parameters $c(\phi)$ depending on $\phi$, and consider $\C S\C
U(\bH(\fz(\CO),c(\phi)))$, the set of unitary spherical
$\bH(\fz(\CO),c(\phi))$-modules. By the equation (\ref{eq:shnu}),
$s\mapsto\nu$, to
a module $V=V_{s,\CO_e,\psi}$ with lowest $W$-type $\mu_0$, we may attach a
spherical module $\rho(V)$ of $\bH(\fz(\CO),c(\phi))$ parameterized by $\nu$ (\ie,
the spherical subquotient of the minimal principal series defined by $\nu$).  

The hope would be that, in this way, the set
$\CU(\CO,\phi)$ can be put in a correspondence with  $\C S\C
U(\bH(\fz(\CO),c(\phi))).$  This is motivated by the case of generic
unitary modules of Hecke algebras with equal parameters
(\cite{Ba,BC1,Ci}), where a particular case of this correspondence worked
almost perfectly. More precisely, that was the case of $\phi=triv$, so that $\bH$ is a
Hecke algebra with equal parameters of the same type as $G$ ($L$ is
a maximal torus, and $\C C=0$). The attached Hecke algebra is
$\bH(\fz(\CO),1)$, meaning that it has equal parameters too. Let us denote by $\C G\C U
(\CO,triv)$ and $\C G\C S\C U(\bH(\fz(\CO),1)$ the corresponding
subsets consisting of (unitary) generic representations, in the sense
of containing the sign $W$-type.
It was shown in the above references that, except for
a handful of cases (none in the classical types, in $G_2$ or $E_6$, one in $F_4$ and in
$E_7$, and six in $E_8$), the sets $\C G\C U
(\CO,triv)$ and $\C G\C S\C U(\bH(\fz(\CO),1)$ are in one-to-one
correspondence. 
In more generality, such a relation is not as perfect,
but it still gives serious information about $\CU(\CO,\phi)$.

There are also two immediate technical problems. One obvious problem is that, in general, an
irreducible module may have more than one lowest $W$-type, so it is
actually better to match $\C U(\CO,\phi)$ with a ``quasi-spherical'' unitary dual of an
extension (by outer automorphisms) of the Hecke algebra defined from
$\fz(\CO)$. A second issue is that, unlike the case of
representations of real reductive groups, a lowest $W$-type may appear
with multiplicity greater than one. But if this is the case, it is a
fact (only empirical to our knowledge) that there exists
another lowest $W$-type with multiplicity one. This fact is needed in order
to be able to define a normalization of the intertwining operators
involved in the calculation of signatures of hermitian forms. However,
this second difficulty only arises for geometric algebras with equal parameters.
We ignore these problems for now and refer to \cite{BC1} and
\cite{Ci2} for details and examples of these phenomena. 

In order to exhibit a relation between $\CU(\CO,\phi)$ and $\C S\C
U(\bH(\fz(\CO),\phi))$, for every $(\CO,\phi)$, we
find a set of $W$-types $\{\mu_0,\mu_1,\dots,\mu_l\}$, and a
corresponding set of 
$W(\fz(\CO))$-representations $\{\rho(\mu_0),\rho(\mu_1),\dots,\rho(\mu_l)\}$,
\begin{equation}\label{corr}
\mu_j\longleftrightarrow\rho(\mu_j),
\end{equation}
  such that:

\begin{enumerate}
\item if $V$ is hermitian but $\rho(V)$ is not, then the hermitian
  form on $V$ is indefinite, \ie, $V$ is not unitary;
\item if $V$ is hermitian and $\rho(V)$ is hermitian, then the
  signature of the hermitian form of $V$ restricted to the $W$-type
  $\mu_j$ is identical with the signature of the hermitian form of
  $\rho(V)$ restricted to the $W(\fz(\CO))$-representation $\rho(\mu_j)$, for
  all $0\le j\le l$.
\end{enumerate}

We explain the relation between $\mu_j$ and $\rho(\mu_j)$ (see \cite{BC1}). Let $\fk
m\supset \fk l$ denote a Levi subalgebra of $\fg$, which is
Bala-Carter for
$\{e,h,f\},$ and let $M\subset G$ be the corresponding Levi subgroup. So $e$ is distinguished in $\fk
m$. There exists a geometric Hecke subalgebra $\bH_M$ of $\bH$ which
is parameterized by $(L,\C L)$ in $M$. By \cite{L4},  the irreducible $\bH_M$-modules
parameterized by $(h,e,\psi')$, for any allowable $\psi'\in
\widehat{A_{M}(s,e)}$, are discrete series. We find a $\psi'$ such that via
the natural map $A_{M}(h,e)=A_{M}(s,e)\rightarrow A(s,e)$,
$\Hom_{A_{M}}(h,e)[\psi,\psi']\neq 0$, and realize $V_{s,e,\psi}$
as a summand of the Langlands subquotient of the standard induced module
$X(M,\sigma,\nu)=\bH\otimes_{\bH_M}(\sigma\otimes\bC_\nu).$ 
Then $\Hom_W[\mu_j:X(M,\sigma,\nu)]$ carries a natural action of
$W(\fz(\CO))$ (independent of $\nu$) and we call the resulting
representation $\rho(\mu_j).$

Of course, this procedure is efficient only if we find a sufficiently
large number of matching types in this way. This is our main criterion
for proving non-unitarity of $\bH$-modules. Ideally, if sufficiently many
$W$-types are matched, one
would be able to conclude that a hermitian module $V$ is unitary only
if $\rho(V)$ is unitary. This doesn't always happen, and in fact we
find examples here of $\CU(\CO,\phi)$'s both smaller or larger than
expected. We exemplify this at the end of the introduction.

\smallskip

There is also the case of spherical unitary $\bH$-modules, whose
lowest $W$-type is the trivial. In that case, $\CO=G\cdot\C C$ (it is
the minimal orbit appearing in the parameterization),  $\fk
z(\CO)$ has the same type as $\bH$, and the above
correspondence is a tautology. 
We determine the spherical unitary $\bH$-modules by other methods,
known to the experts. In particular, we use the Iwahori-Matsumoto
involution $IM$ (definition (\ref{eq:4.2.1})), which preserves unitarity. If $V$ is a spherical $\bH$-module, then
$IM(V)$ contains the sign $W$-type. If $IM(V)\neq V,$ then $IM(V)$ is
parameterized by a larger orbit, for which the unitary set has already
been determined. But in order to apply this method, we need to
determine explicitly which unitary irreducible modules have the sign
$W$-type for any given pair $(\CO',\phi').$ This is achieved by 
computing composition series as part of section \ref{sec:5.1b}, and the result is in table  \ref{table:5.2}.

So this approach leaves the case $V=IM(V)$ and $V$ spherical, equivalently, $V$ must be
an irreducible spherical principal series. The unitarity of such
representations is discussed in section \ref{sec:5.5}, and the methods are again
well-known: continuous deformations of parameters
and unitary induction, but also certain explicit calculation of
signatures. 

\subsection{}We give an outline of the paper. In section \ref{sec:1.1}-\ref{sec:1.7}, we present the necessary
definitions and background on the representation theory of affine and
graded Hecke algebras, assuming the
parameters are arbitrary. We will mostly work with the affine graded
Hecke algebra. In particular, we
recall the Langlands classification (actually, the reduction to
tempered modules) as 
in \cite{E}, the unitarity of tempered modules, as it follows from
\cite{O1}, and the results about the Hermitian forms and intertwining
operators from \cite{BM3}. In section \ref{sec:1.8}, we present a type of reduction to
{\it real} central (``infinitesimal'') character for the unitary dual
of the Hecke algebra, which 
is the complete analogue of the result for real reductive groups, as
in \cite{K}, XVI.4. This is well-known to the specialists.

From section \ref{sec:3} on, we restrict to the case of {geometric} Hecke
algebras. We recall the relevant
results about the classification of simple $\bH$-modules, and the
generalized Springer correspondence (\cite{L5}). An
important consequence is that the set of irreducible tempered
$\bH$-modules with real central character are linearly independent in
the Grothendieck group of $W$. Therefore the method of
\cite{BM1} holds, and the correspondence with unitary
(unipotent) representations of $p$-adic groups can be established. So for geometric Hecke algebras, the unitarity of tempered $\bH$-modules
could also be obtained from the correspondence with the $p$-adic group.

In section \ref{sec:4}, we analyze Hermitian forms and intertwining operators
for simple modules. We restrict, as we may, to modules with real
central character. 
As an easy application, we
present the spherical unitary dual for type $G_2$ with arbitrary
unequal parameters. (In fact, if we factor in the {Iwahori-Matsumoto involution},
and the tempered modules, one obtains all the unitary modules with real
central character of $G_2$.)  

Section \ref{sec:5} presents a more interesting application. We determine the
unitary dual of the Hecke algebra (of type $F_4$) constructed in
\cite{L2} from a cuspidal local system on the principal nilpotent
orbit in the Levi $(3A_1)''$ of the simply connected $E_7.$ Besides the equal parameter
case, this is, essentially, the only
Hecke algebra of type $F_4$ which appears in the classification of
unipotent representations of $p$-adic groups from \cite{L6} (the
equal parameter $\bH(F_4)$ was treated in \cite{Ci}). We also remark that the
Hecke algebras of types $B/C$ with unequal parameters which appear in
the classification of the unipotent exceptional $p$-adic groups are
parabolic subalgebras of this one.

The main results are theorem \ref{t:5.1}, proposition and corollary \ref{p:5.3}.  Although the
methods employed are mostly uniform, details need to be checked case
by case for each nilpotent orbit in $E_7$ which appears in the
classification (see figure \ref{3A1nil}).   As mentioned above, we use
the fact that $IM$
 preserves unitarity of modules. In order to
use $IM,$ we need to compute the decompositions of standard
modules and the $W$-structure of unitary modules. 

As a consequence of the calculations with intertwining operators in
section 5, we obtain the $W$-structure of standard
modules for the cuspidal local system $(3A_1)''$ in $E_7$ (section \ref{green}). This is also (at
least in principle) computable by the generalized Green polynomials
algorithms in \cite{L7}. An implementation of such algorithms for small rank
classical groups was realized in \cite{LS}.

\subsection{Examples} We conclude the introduction with two examples relative to the matching
of signatures and unitary duals for this unequal parameter $\bH$ of
type $F_4.$ The nilpotent orbits are of type $E_7.$  

\smallskip

\noindent (a) $\CO=A_2+3A_1$, $A(\CO)=\bZ/2\bZ$, $\fz(\CO)=G_2$,
$\phi=sgn.$ The intertwining operator
calculations indicate a matching with the spherical unitary dual of a
Hecke algebra $\bH(G_2,(2,1))$: type $G_2$ with parameter $2$ on the
long root and $1$ on the short root. ${W(G_2)}$ has six
representations, in the notation of \cite{Ca}: $(1,0),$ $(1,3)'$,
$(1,3)''$, $(1,6)$, $(2,1),$ $(2,2).$ There are two minimal subsets of
$\widehat{W(G_2)}$ which determine $\C S\C U(\bH(G_2,(2,1))$, in the
sense that a hermitian module is unitary if and only if the hermitian
form is positive definite on this set of types:
$\{(1,0),(2,1),(2,2)\}$ or $\{(1,0),(1,3)',(1,3)'',(2,1)\}.$ It turns
out there are $W(F_4)$-types $\mu_j$ such that one realizes as
$\rho(\mu_j)$ (see the explanation around equation (\ref{corr})) the $W(G_2)$-representations $\{(1,0), (1,3)',(2,1),(1,6)\}.$ So we cannot conclude
that $\CU(A_2+3A_1,sgn)$ is a subset of $\C S\C U(\bH(G_2,(2,1))$, and
in fact, we find that $\CU(A_2+3A_1,sgn)\supsetneq\C S\C
U(\bH(G_2,(2,1))$ (see table \ref{table:unitary} and figure \ref{fig:A23A1}). 

\smallskip

\noindent (b) $\CO=D_4(a_1)+A_1$, $A(\CO)=(\bZ/2\bZ)^2$,
$\fz(\CO)=2A_1$. There are two representations $\phi_1$ and $\phi_2$
of $A(\CO)$ which enter 
the generalized Springer correspondence with $\widehat {W(F_4)}$
(\cite{Sp}): $\phi_1$ corresponds to a nine-dimensional $W(F_4)$-representation,
and $\phi_2$ to a two-dimensional. We consider here $\phi_1.$ The calculations indicate that the
matchings should be with the spherical unitary dual of
$\bH(2A_1,(5,5))$: type $A_1+A_1$, and the parameter is $5$ on both
simple roots. The set $\C S\C U(\bH(2A_1,(5,5))$ is detected on
$\{triv\otimes triv,triv\otimes sgn, sgn\otimes triv\}\subset
\widehat{W(2A_1)}.$ All three representations are matched by
certain $\mu_j$'s, so we can conclude that
$\CU(D_4(a_1)+A_1,\phi_1)\subset \C S\C U(\bH(2A_1,(5,5))$. The problem
here is that there is extra reducibility in the standard modules for $\bH$
parameterized by $(\CO,\phi)$ which is not detected by the hermitian
form on the $W(F_4)$-types $\mu_j.$ What we find is that
$\CU(D_4(a_1)+A_1,\phi_1)\subsetneq \C S\C U(\bH(2A_1,(5,5))$ (see
table \ref{table:unitary}).

\medskip

\begin{small}
\noindent{\bf Acknowledgments.} 
The basis for this paper stems from joint work with D. Barbasch. Most
of the ideas here can be traced there, and sometimes even before, to
the work of Barbasch and Moy. I thank D. Barbasch for sharing his
ideas with me.

I thank G. Lusztig for discussions about the generalized
Springer correspondence, and P. Trapa for encouraging me to write up
these calculations, and for his many helpful
suggestions.  This research was supported  by the NSF
grant FRG-0554278.
\end{small}

\section{Preliminaries}\label{sec:1}

\subsection{}\label{sec:1.1} Let $(X,\check X,R,\check R,\Pi)$ be a fixed based root
datum. Denote by $R^+$ the positive roots determined by $\Pi,$ and by
$\check R^+$ the corresponding positive coroots. Define $\fh=\check
X\otimes_\bZ \bC$ and $\fh^*=X\otimes_\bZ\bC.$  Let $W$ be the finite
Weyl group generated by the set $S$ of simple reflections in the roots
of $\Pi$. Denote by $Q=\bZ R\subset X$ the root lattice, and set
$\Omega=X/Q.$

Let $\wti W=W\ltimes X$ be the extended Weyl group, and
$W_{\text{aff}}=W\ltimes Q$ be the affine Weyl group with the set of simple
affine reflections $S_{\text{aff}}.$ Recall that $W_{\text{aff}}$ is normal in $\wti
W,$ and $\wti W/W_{\text{aff}}\cong \Omega.$ Let $\ell:\wti W\to \bZ_{\ge 0}$
denote the length function (\cite{L1}): it is the extension of the (Coxeter)
length function of $W_{\text{aff}}$ and it is identically $0$ on $\Omega.$

\subsection{}\label{sec:1.2}  Let $c:R\to\bZ_{>0}$ be a
function such that $c_\al=c_\beta$, whenever $\al$ and $\beta$ are
$W$-conjugate. Let $r$ denote an indeterminate. 
As a vector space, \begin{equation}\label{1.1.1}\Bbb H=\Bbb
  C[W]\otimes\bC[r]\otimes\Bbb A,\end{equation} where $\Bbb A$ is the symmetric
algebra over $\fh^*$. The generators are $t_w\in \Bbb C[W]$, $w\in W$
and $\om\in\fh^*$. The relations between the generators are:
\begin{align}\label{1.1.2}
&t_wt_w'=t_{ww'},&\text{ for all }w,w'\in W;\notag\\
&t_s^2=1, &\text{ for any simple reflection } s\in W;\notag\\
&\om t_s=t_ss(\om)+rc_\al\om(\check\alpha),&\text { for simple
    reflections } s=s_\alpha\in S.\notag\\
\end{align}
Assume the root system $R$ is irreducible. If $R$ is simply-laced, the
function $c$ must be constant. In the two root lengths case, we will denote the function
$c$ by the pair $(c_l,c_s)\in \bZ_{>0}^2$, where $c_l$, $c_s$ specify
the values of $c$ on the long and short roots, respectively (in this
order). Let $k\in \bZ_{>0}$ be the ratio
$k=(\al_l,\al_l)/(\al_s,\al_s)$ for some long root $\al_l$ and short root $\al_s.$ Then it is straightforward to see
that $\bH(R,(c_1,c_2))\cong \bH(R,(kc_2,c_1)).$ For example,
$\bH(F_4,(1,2))\cong \bH(F_4,(4,1)).$ 

\subsection{} By \cite{L1}, the center of $\bH$ is
  $\bC[r]\otimes \bA^W$.
On any simple (finite dimensional) $\bH$-module, the center of $\bH$
acts by a
character, which we will call a {\it central character}. The central
characters correspond to $W$-conjugacy classes of semisimple elements $(r_0,s)\in\bC\oplus\fh.$ 

We decompose $\fh$ into  real and imaginary parts; the complex
conjugation $\bar\ $ of $\bC$ induces a conjugation on $\fh$. We
set $\fh_\bR=\{a\in \fh: \overline a=a\}$ and
$\fh_{i\bR}=\{a\in\fh:\overline a=-a\}.$ For every $s\in \fh,$ we have
then a unique decomposition $s=Re(s)+\sqrt {-1} Im(s),$ where
$Re(s)\in \fh_\bR$ and $Im(s)\in \fh_{i\bR}.$

\begin{definition}\label{d:1.2} A central character $(r_0,s)$ is called
  {\it real} if $(r_0,s)\in \Bbb R\oplus\fh_\Bbb R$.
\end{definition}

\subsection{} $\Bbb H$ has a $*$- operation  given on generators as
follows (as in 
\cite{BM2}, section 5):
\begin{align}\label{1.3.1}
&t_w^*=t_{w^{-1}},\ w\in W;\ r^*=r;\\\notag
&\om^*=-\overline \om+r\sum_{\alpha\in
  R^+}c_\al\overline\om(\check\alpha) t_{s_\alpha},\
\om\in\fh^*.\notag 
\end{align} 

The $\bH$-module $V$ is called {\it Hermitian} if it admits a
Hermitian form $\langle~,~\rangle$ such that 
\begin{equation}\label{1.3.2}
\langle x\cdot v_1,v_2\rangle=\langle v_1,x^*\cdot v_2\rangle,\text{ for
  all } v_1,v_2\in V,~x\in \bH.
\end{equation}
It is called {\it unitary}, if in addition the Hermitian form is
positive definite.

\subsection{}\label{sec:1.4} We present the {Langlands classification} for $\bH$ as
in \cite{E}.

If $V$ is a (finite dimensional) simple $\bH$-module, $\bA$ induces a
generalized weight space decomposition 
\begin{equation}
V=\bigoplus_{\lambda\in \fk h} V_\lambda.
\end{equation}
 Call $\lambda$ a {\it weight} if $V_\lambda\neq 0.$

\begin{definition} The irreducible module $\sigma$ is called {\it
     tempered} if $\ome_i(\lambda)\le 0,$ for all
  weights $\lambda\in \fk h$ of $\sigma$ and all fundamental weights
  $\ome_i\in\fk h^*.$  
If $\sigma$ is tempered and $\ome_i(\lambda)<0,$ for all
  $\lambda,\omega_i$ as above, $\sigma$ is called a {\it discrete series}.
\end{definition}

For every
$\Pi_M\subset \Pi$, define $R_M\subset R$ to be the set of roots generated by
$\Pi_M$,  $\check{R}_M\subset \check R$ the corresponding set of
coroots, and $W(M)\subset W$ the corresponding Weyl subgroup. Let
$\bH_M$ be the Hecke algebra attached to $(\fh, R_M)$. It can be regarded
naturally as a subalgebra of $\bH.$ 

Define $\fk t=\{\nu\in\fk h:
\langle\al,\nu\rangle=0,\text{ for all } \al\in\Pi_M\}$ and $\fk
t^*=\{\lambda\in\fk h^*: \langle\lambda,\check\al\rangle=0,\text{ for
  all }\al\in \Pi_M\}.$ 
Then  $\bH_M$ decomposes as
$$\bH_M=\bH_{M_0}\otimes S(\fk t^*),$$ where $\bH_{M_0}$ is the Hecke
algebra attached to $(\bC\langle\Pi_M\rangle,R_M).$ 

We will denote by $I(M,U)$ the induced module
$I(M,U)=\bH\otimes_{\bH_M} U.$ 

\begin{theorem}[\cite{E}]\label{t:1.4} 
\ 

\noindent\begin{enumerate}
\item 
Every irreducible $\bH$-module
  is a quotient  of a standard induced module
  $X(M,\sigma,\nu)=I(M,\sigma\otimes \bC_\nu),$ where $\sigma$ is
  a tempered module for $\bH_{M_0},$ and $\nu\in \fk t^+=\{\nu\in\fk t:
    Re(\al(\nu))>0,\text{ for all }\al\in\Pi\setminus\Pi_M\}.$ 

\item Assume the notation from (1). Then $X(M,\sigma,\nu)$ has a unique
  irreducible quotient, denoted by $L(M,\sigma,\nu)$.

\item If $L(M,\sigma,\nu)\cong L(M',\sigma',\nu'),$ then $M=M',$ $\sigma\cong \sigma'$ as
$\bH_{M_0}$-modules, and $\nu=\nu'.$
\end{enumerate}
\end{theorem}

We will call a triple $(M,\sigma,\nu)$ as in theorem \ref{t:1.4} a {\it
  Langlands parameter}.

\subsection{} We need to recall some results about the affine (Iwahori-)Hecke
algebra as in \cite{O1}, and their implications for $\bH.$ Fix $q>1.$
The {\it affine Hecke algebra}
$\CH$ is the associative $\bC$-algebra with basis $\{T_w: w\in \wti W\}$
 and set of
parameters $d:\Pi_{\text{aff}}\to\Bbb Z_{>0}$ which satisfies the relations:

\begin{enumerate}
\item $T_wT_{w'}=T_{ww'}$ if $w,w'\in\wti W,$ \text{ with } $\ell(ww')=\ell(w)+\ell(w');$
\item $(T_{s_\al}+1)(T_{s_\al}-q^{d_\al})=0,$ for all simple
  affine reflections $s_\al\in S_{\text{aff}}.$
\end{enumerate}
This is the Iwahori-Matsumoto presentation of $\CH.$

The $*$-operation for $\CH$ is given by $T_w^*=T_{w^{-1}}.$ The
algebra $\CH$ has a {\it trace}, i.e., a linear functional $\tau:\CH\to\bC$, given by 
\begin{align}
\tau(T_w)=\left\{\begin{matrix} 1, &w=1\\0, &w\neq
1\end{matrix}\right.,\quad w\in\wti W.
\end{align}
One checks that $\tau$ defines an inner product on $\CH$ via
\begin{equation}(x,y):=\tau(x^*y),\ x,y\in\CH,\end{equation} and the
basis $ \{T_w: w\in \wti W\}$
is orthogonal with respect to $(\ ,\ ).$ 

\begin{definition}
Let $\fk H$ be the Hilbert space completion of $\CH$ with respect to
$(\ ,\ ).$
\end{definition}

\subsection{}\label{sec:1.7} The algebra $\CH$ admits a second presentation due to
Bernstein and Lusztig. The generators are $\{T_w:w\in W\}$ and
$\{\theta_x:x\in X\}$ with relations: 
\begin{enumerate}
\item $T_wT_{w'}=T_{ww'}$ if $w,w'\in W,$ \text{ with } $\ell(ww')=\ell(w)+\ell(w');$
\item $(T_{s_\al}+1)(T_{s_\al}-q^{d_\al})=0,$ for all simple
  reflections $s_\al\in S;$
\item $\theta_x\theta_x'=\theta_{x+x'},\ x,x'\in X,\text{ and
}\theta_0=1;$
\item
  $\theta_xT_{s_\al}-T_{s_\al}\theta_{s_\al(x)}=(q^{d_\al}-1)\frac{\theta_x-\theta_{s_\al(x)}}{1-\theta_{-\al}},$
  $x\in X, s_\al\in S.$ 
\end{enumerate}
Denote by $\CH_W$ and $\CA$ the subalgebras generated by $\{T_w\}$ and
$\{\theta_x\}$ respectively. We refer the reader to \cite{L1} for the more general version of this definition of
$\CH$ and for the relation between the two presentations.

Note that the trace $\tau$ is easy to define using the Iwahori-Matsumoto
presentation, but not so if one uses instead the Bernstein-Lusztig
presentation. (See \cite{O2} for details.)
Similarly to section \ref{sec:1.4}, one defines {\it tempered} and
{\it discrete series} modules for $\CH$ using the weights of the
abelian subalgebra $\CA$ (``Casselman's criteria'', see \cite{L1} and \cite{O1}).

\begin{proposition}[\cite{O1},2.22] A finite dimensional
  representation $\sigma$ of $\CH$ is a discrete series if and only
  if it is a subrepresentation of $\fk H.$ In particular, $\sigma$ is
  a unitary module of $\CH.$ 
\end{proposition}

A standard argument then implies the unitarity of tempered
$\CH$-modules as well.
Moreover, we can transfer this result to the graded Hecke algebra,
using the correspondence between tempered modules of $\CH$ and
tempered modules of $\bH$
(theorem 9.3 in \cite{L1}) and the preservation of unitarity (theorem
4.3 in \cite{BM2}). 

\subsection{}\label{sec:1.8}  Given a module $V$, let $V^h$ denote the Hermitian dual. 
Every element $x\in \bH$ can be written uniquely as $x=\sum_{w\in
  W/W(M)} t_w x_w,$ with $x_w\in \bH_M.$ Let $\ep_M:\bH\to\bH_M$ be
the map defined by $\ep_M(x)=x_1$, that is, the component of the
identity element $1\in W.$ In the particular case $\Pi_M=\emptyset,$ we will
denote the map by $\ep:\bH\to \bA.$ 

\begin{lemma}[\cite{BM3}, 1.4] If $U$ is module for $\bH_M$, and
  $\langle\ ,\ \rangle_M$ denotes the Hermitian pairing with $U^h,$
  then the Hermitian dual of $I(M,U)$ is $I(M,U^h)$, and the Hermitian
  pairing is given by $$\langle t_x\otimes v_x,t_y\otimes
  v_y\rangle_h=\langle \ep_M(t^*_yt_x)v_x,v_y\rangle_M, \ x,y\in
  W/W(M),\ v_x,v_y\in U.$$ 

\end{lemma}

Applying this result to a Langlands parameter $(M,\sigma,\nu)$ as in
section \ref{sec:1.4}, we find that the Hermitian dual of $X(M,\sigma,\nu)$
is $I(M,\sigma\otimes\bC_{-\overline\nu}).$

Let $w_0$ denote the longest Weyl group element in $W$, and
let $W(w_0M)$ be the subgroup of $W$ generated by the reflections in
$w_0 R_M.$ Let $w_m$
denote a shortest element in the double coset
$W(w_0M)w_0W(M)$. Then $w_m\Pi_M$ is a subset of $\Pi,$ which we
denote by $\Pi_{w_mM}.$

\begin{proposition}[\cite{BM3}, 1.5]\label{p:1.5} The Hermitian dual of the
  irreducible Langlands quotient $L(M,\sigma,\nu)$ is 
$L(w_mM,w_{m}\sigma,-w_m{\overline\nu}).$ In particular, $L(M,\sigma,\nu)$ is Hermitian if
and only if there exists $w\in W$ such that
$$
wM=M,\quad w\sigma\cong \sigma\quad\text{and}\quad w\nu=-\overline\nu.
$$
\end{proposition}
\noindent If this is the case, we will denote by $a_w:w\sigma\to \sigma$ the
corresponding isomorphism.

Let $w=s_1\dots s_k$ be a reduced decomposition of $w$.  For each
simple root $\alpha$, define 
\begin{equation}\label{1.5.1}
r_{s_\alpha}=t_{s_\alpha}\alpha-c_\al;\quad r_w=r_{s_{\al_1}}\dots r_{s_{\al_k}}.
\end{equation}
Lemma 1.6 in \cite{BM3} (based on proposition 5.2 in \cite{L1}) proves
that $r_w$ does not depend on the reduced expression of $w$. 

Assume $L(M,\sigma,\nu)$ is Hermitian. Define
\begin{equation}\label{1.5.2}
\CA(M,\sigma,\nu):X(M,\sigma,\nu)\to I(M,\sigma\otimes\bC_{-\overline\nu}),\  x\otimes
(v\otimes 1_\nu)\mapsto xr_{w}\otimes (a_w(v)\otimes 1_{-\overline\nu}).
\end{equation}
One can verify, as in section 1.6 of \cite{BM3}, that this is  an intertwining
operator. The image of $\CA(M,\sigma,\nu)$ is the Langlands quotient
$L(M,\sigma,\nu).$ 

\subsection{} Let $(M,\sigma,\nu)$
be a Langlands parameter as in section \ref{sec:1.4}, and $\nu=Re~\nu+\sqrt{-1} Im~\nu$ with
$Im~\nu\neq 0.$ Set 
\begin{align}
R_{M_1}=\{\al\in R: \langle Im~\nu,\al\rangle=0,\quad R_{N_1}=\{\al\in
R: \langle Im~\nu,\al\rangle>0\}.
\end{align}
Clearly, $\Pi_M\subset R_{M_1}.$ Moreover, $R_{M_1}$ is a root
subsystem of $R.$ Set $R_{M_1}^+=R_{M_1}\cap R^+,$ and let
$\Delta_{M_1}$ denote the set of corresponding simple roots. Note that
$\Pi_M\subset \Delta_{M_1}$, but
$\Delta_{M_1}$ need not be a subset of $\Pi.$ But $\bH_M$ is naturally
a subalgebra of $\bH_{M_1}.$ 

Assume $L(M,\sigma,\nu)$ is Hermitian and let $w\in W$ be as in proposition
\ref{p:1.5}. From $w\nu=-\overline\nu,$ it follows that $w\in W(M_1).$

The triple $(M,\sigma,Re~\nu)$ is a Langlands triple for $\bH_{M_1}$, so it
makes sense to consider $\sigma_1=L_{M_1}(M,\sigma,Re~\nu).$ Moreover,
$\sigma_1$ is Hermitian in $\bH_{M_1}.$ 

\begin{proposition}\label{p:1.7}
With the notation as above,
$$I(M_1,(\sigma_1\otimes \bC_{Im~\nu}))\cong L(M,\sigma,\nu).$$
\end{proposition}
 
\begin{proof}
It is known that there exists $w'\in W$ (of shortest length) which maps
$\Delta_{M_1}$ onto a subset $\Pi_{M'}$ of $\Pi.$ This defines an
isomorphism $a_{w'}:\bH_{M_1}\to \bH_{M'},$ and, similarly to section
\ref{p:1.5}, an intertwining operator $\CA(M_1:M',\C V,\nu').$ Setting
$\C V=X_{M_1}(M,\sigma,Re~\nu)$ and $\nu'=Im~\nu,$ we find that this
operator is invertible. 
\begin{equation}\label{1.7.2}
X(M,\sigma,\nu)\cong
I(M_1,X_{M_1}(M,\sigma,Re~\nu)\otimes \bC_{Im~\nu}).
\end{equation}
 Now,
$X(M,\sigma,\nu)$ maps onto $L(M,\sigma,\nu)$ via the operator $\CA(M,\sigma,\nu)$ in
(\ref{1.5.2}). On the other hand, $X_{M_1}(M,\sigma,Re~\nu)$ maps onto
$\sigma_1$ by the operator $\CA_{M_1}(M,\sigma,Re~\nu),$ and by
induction (which is an exact functor) we find a map from the right
hand side of (\ref{1.7.2}) onto
$I(M_1,(\sigma_1\otimes \bC_{Im~\nu})$ whose
kernel is $\bH\otimes_{\bH_{M_1}}(\ker\CA_{M_1}(M,\sigma,Re~\nu))\cong \ker
\CA(M',\sigma,\nu),$ which is identical to $\ker\CA(M,\sigma,\nu).$ 

\end{proof}

\begin{corollary}\label{c:1.7}
Assuming the previous notation, $L(M,\sigma,\nu)$ is unitary if and only if
$L_{M_1}(M,\sigma,Re~\nu)$ is unitary.
\end{corollary}

\begin{proof}
This follows immediately from proposition \ref{p:1.7}, and the fact
that if $\chi$ is a purely imaginary character and $I(M,U\otimes
\bC_{\chi})$ is irreducible, then $I(M,U\otimes \bC_{\chi})$ is
unitary for $\bH$ if and only if $U$ is unitary for $\bH_M.$  
\end{proof}

\section{Geometric Hecke algebras}\label{sec:3}

We will restrict to the case of {\it geometric Hecke
algebras}, as in \cite{L2}-\cite{L4}. We review some facts about these
algebras and their geometric classification.

\subsection{}
For an algebraic group $\mathbf G$ and a $\mathbf G$-variety $X$, let
$H_\mathbf G^\bullet(X)=H_\mathbf G^\bullet(X,\bC)$, respectively
$H^\mathbf G_\bullet(X)=H^\mathbf G_\bullet(X,\bC)$ denote the equivariant cohomology,
respectively homology (as in section 1 of \cite{L2}). The component
group of $\mathbf G$, $A(\mathbf G)=\mathbf G/\mathbf G^0$ acts naturally on $H_{\mathbf G^0}^\bullet(X)$ and
$H^{\mathbf G^0}_\bullet(X)$.
The cup product defines a structure of graded $H_\mathbf G^\bullet(X)$-module on
$H^\mathbf G_\bullet(X).$ If $pt$ is a point of $X$, one uses the notation
$H_\mathbf G^\bullet=H_\mathbf G^\bullet(\{pt\}),$ respectively $H^\mathbf
G_\bullet=H^\mathbf G_\bullet(\{pt\}).$ There is a
$\bC$-algebra homomorphism $H_\mathbf G^\bullet\to H_\mathbf G^\bullet(X)$ induced by the map
$X\to\{pt\}$, and therefore $H_\mathbf G^\bullet$ and $H^\mathbf G_\bullet$ can both
be considered as $H_\mathbf G^\bullet$-modules.

For a subset $\C S$ of $\mathbf G$, or $\mathbf{\fg}$, let $Z_{\mathbf
  G}(\C S)$, $N_{\mathbf G}(\C S)$ denote the centralizer,
respectively the normalizer of $\C S$ in $\mathbf G.$ 

\medskip

Let $G$ be a reductive connected complex algebraic group, with Lie
algebra $\fg.$ Let $P=LN$ denote a parabolic subgroup, with $\fk p=\fk
l+\fk n$ the corresponding Lie algebras, such that $\fk l$ admits an
irreducible $L$-equivariant
cuspidal local system (as in \cite{L2},\cite{L5}) $\C L$ on a
nilpotent $L$-orbit $\C C\subset \fk l.$ The classification of
cuspidal local systems can be found in \cite{L5}. In particular, $W=N(L)/L$ is
a Coxeter group. 

Let $H$ be the center of $L$ with Lie algebra $\fk h$, and let $R$ be
the set of nonzero weights $\al$ for the $ad$-action of $\fk h$ on $\fk g,$
and $R^+\subset R$ the set of weights for which the corresponding
weight space $\fg_\al\subset\fk n.$ For each parabolic $P_j=L_jN_j$, $j=1,n$, such that
$P\subset P_j$ maximally and $L\subset L_j$, let $R_j^+=\{\al\in R^+:
\al(\fz(\fk l_j))=0\},$ where $\fz(\fk l_j)$ denotes the center of
$\fk l_j.$ It is shown in \cite{L2} that each $R_j^+$ contains a unique
$\al_j$ such that $\al_j\notin 2R.$ 

Let $Z_G(\C C)$ denote the centralizer in $G$ of a Lie triple for $\C
C,$ and $\fz(\C C)$ its Lie algebra. 

\begin{proposition}[\cite{L2}]\ 

(a) $R$ is a (possibly non-reduced) root system in $\fk h^*$, with
simple roots $\Pi=\{\al_1,\dots,\al_n\}.$ Moreover, $W$ is the
corresponding Weyl group.

(b) $H$ is a maximal torus in $Z^0=Z^0_G(\C C)$.

(c) $W$ is isomorphic to $W(Z^0_G(\C C))=N_{Z^0}(H)/H.$

(d) The set of roots in $\fz(\C C)$ with respect to $\fk h$ is exactly
the set of reduced roots in $R.$  

\end{proposition}

For each $j=1,n$, let $c_j\ge 2$ be such that 
\begin{equation}
ad(e)^{c_j-2}:\fk l_j\cap\fk n\to\fk l_j\cap\fk n \neq 0, \text{ and }
ad(e)^{c_j-1}:\fk l_j\cap\fk n\to\fk l_j\cap\fk n =0.
\end{equation}

By proposition 2.12 in \cite{L2}, $c_{i}=c_j$ whenever $\al_i$ and
$\al_j$ are $W$-conjugate. Therefore, one can define a Hecke algebra
$\bH$ as in (\ref{1.1.1}),(\ref{1.1.2}). In view of the proposition
above, one can think of this algebra as essentially a graded Hecke
algebra with unequal parameters for the centralizer $\fz(\C C).$  The
explicit algebras which may appear are listed in 2.13 of
\cite{L2}. The more familiar case of Hecke algebras with equal
parameters arise when one takes $P$ to be a Borel subgroup, and $\C C$
and $\C L$ to be trivial.

\medskip

The geometric realization of $\bH$ is obtained as follows. 
Consider the varieties
\begin{align}
\dot{\fg}_N&=\{(x,gP)\in \fg\times G/P|\ Ad(g^{-1})x\in \fk n\},\\\notag
\ddot{\fg}_N&=\{(x,gP,g'P)\in \fg\times G/P\times G/P|\ (x,gP),(x,g'P)\in
\dot{\fg}_N\}.
\end{align} 
Clearly $\ddot{\fg}_N\subset \dot{\fg}_N\times \dot{\fg}_N.$ The local
system $\C L$ gives a local systems $\dot {\C L}$ on $\dot{\fg}_N$ and
$\ddot {\C L}$ on $\ddot{\fg}_N$. When $P$ is a Borel subgroup and $\C
L$ is trivial, one recovers the classical objects: $\dot{\fg}_N$ is
the cotangent bundle of the flag variety of $G$, and
$\ddot{\fg}_N$ is the Steinberg variety.

The group $G\times \bC^*$ acts on $\fg$ by $(g_1,\lambda)\cdot
x=\lambda^{-2} Ad(g_1) x,$ for every $x\in \fg, g_1\in G,\lambda\in
\bC^*.$ 
In
\cite{L2}, the vector space $H_\bullet^{G\times
  \bC^*}(\ddot{\fg}_N,\ddot{\C L})$ is endowed with left and right actions of $W$ and
$S(\fh^*\oplus \bC)\cong H^\bullet_{G\times \bC^*}(\dot{\fg}_N)$, and
it is proved (theorem 6.3 and corollary 6.4) that
\begin{equation}
H_\bullet^{G\times \bC^*}(\ddot\fg_N,\bC)\cong \bH, \text{ as }\bH\text{-bimodules}.
\end{equation}

\subsection{} Now we recall the construction of standard modules for
$\bH$. Fix  a nilpotent element $e$ in $\fg,$ and let $\C P_e$ be the
variety
\begin{equation}
\C P_e=\{gP\in G/P: Ad(g^{-1})e\in \C C+\fk n\}.
\end{equation}

The centralizer $Z_{G\times \bC^*}(e)$ acts on $\C P_e$ by
$(g_1,\lambda).gP=(g_1g)P.$ 

\cite{L2} constructs actions of $W$ and $S(\fh^*\oplus\bC)$ on
$H_\bullet^{Z^0_{G\times \bC^*}(e)}(\C P_e,\dot{\C L})$, and proves that these are compatible
with the relations between the generators of $\bH,$ therefore obtaining a module of $\bH$
(theorem 8.13).
The component group $A_{G\times \bC^*}(e)$ acts on
$H_\bullet^{Z^0_{G\times \bC^*}(e)}(\C P_e,\dot{\C L})$, and commutes with the
$\bH$-action (8.5).

Consider the variety $\C V$ of semisimple $Z^0_{G\times \bC^*}(e)$-orbits
on the Lie algebra $\fz_{G\times \bC^*}(e)=\{(x,r_0)\in \fg\oplus\bC:
[x,e]=2r_0e\}$ of $Z_{G\times \bC^*}(e).$ The affine variety $\C V$ has
$H^\bullet_{Z^0_{G\times \bC^*}(e)}$ as the coordinate ring. Define the
$\bH$-modules
\begin{align}
X(s,r_0,e)=\bC_{(s,r_0)}\otimes_{H^\bullet_{Z^0_{G\times
      \bC^*}(e)}} H_\bullet^{Z^0_{G\times \bC^*}(e)}(\C P_e,\dot{\C L}), 
\end{align} 
where $\bC_{(s,r_0)}$ denotes the ${H^\bullet_{Z^0_{G\times
      \bC^*}(e)}}$-module given by the evaluation at $(s,r_0)\in \C
V,$ ${H^\bullet_{Z^0_{G\times
      \bC^*}(e)}}\to \bC.$ 

Let $A_{G\times \bC^*}(e,s,r_0)$ denote the stabilizer of $(s,r_0)$ in
$A_{G\times \bC^*}(e).$ For each $\psi\in \widehat {A_{G\times \bC^*}(e,s,r_0)},$ define 
\begin{align}
X(s,r_0,e,\psi)=\Hom_{A_{G\times \bC^*}(e,s,r_0)}[\psi:X(s,r_0,e)].
\end{align}
In particular, when $(s,r_0)=\mathbf 0$, we have (7.2, 8.9 in \cite{L2} and 10.12(d) in \cite{L3}) 
\begin{align}\label{1.9.4}
X(\mathbf 0,e)\cong H_\bullet^{\{1\}}(\C P_e,\dot{\C L})\text{ as } W\times A_{G}(e)\text{-modules.} 
\end{align}

Let $\widehat {A_{G}(e)^0}$ and $\widehat {A_{G}(e,s)^0}$ denote the set of 
representations $\psi$ which appear in the $A_{G}(e)$-module $H_\bullet^{\{1\}}(\C
P_e,\dot{\C L})$, respectively in the restriction of this module  to
$A_{G}(e,s).$ 

\begin{theorem}[\cite{L2},\cite{L3}] Assume $r_0\neq 0.$

\noindent(a)(\cite{L2},8.10) $X(s,r_0,e,\psi)\neq 0$ if and only if $\psi\in \widehat {A_{G}(e,s,r_0)^0}.$

\noindent(b)(\cite{L2},8.15) Any simple $\bH$-module on which $r$ acts
by $r_0$ is a quotient $\overline X(s,r_0,e,\psi)$ of an
$X(s,r_0,e,\psi)$, where $\psi\in \widehat {A_{G}(e,s)^0}.$

\noindent(c)(\cite{L3},8.18) The set of isomorphism classes of simple
$\bH$-modules with central character $(s,r_0)$ is in 1-1
correspondence with the set 
\begin{equation}\label{3.2.5}
\C M_{s,r_0}=Z_G(s)\text{-conjugacy
  classes on }
\{(e,\psi): [s,e]=2r_0e,\ \psi\in \widehat {A_G(e,s)^0}\}.
\end{equation}
\end{theorem}

\subsection{}\label{sec:3.3} Fix a semisimple class $(s,r_0)\in \C V.$ The following
theorem is immediately implied by the results in \cite{L3}, 10.5-10.7.

\begin{theorem}[\cite{L3}]\label{t:3.3} If a composition factor $Y$ of the standard
  module $X(s,r_0,e,\psi)$ is parameterized by the class $(e',\psi')\in
  \C M_{s,r_0}$,
  then necessarily $\CO\subset \overline{\CO'}$, where $\CO$, $\CO'$
  are the $Z_G(s)$-orbits of $e$, respectively $e'$ in $\C M_{s,r_0}.$
  Moreover, $\CO=\CO'$ if and only if $Y=\overline X(s,r_0,e,\psi).$ 
\end{theorem}

This is the geometric equivalent of the classical results for real and
$p$-adic groups concerning
the minimality of the Langlands parameter in a standard module (see
\cite{BW}, IV.4.13 and XI.2.13).

Finally, an immediate corollary of theorem \ref{t:3.3}  (see 8.17 in
\cite{L2}, equivalently 10.9 in \cite{L3}) is that, if $\CO$ is the
unique open $Z_G(s)$-orbit in $\C M_{s,r_0},$ then $X(s,r_0,e,\psi),$
$\psi\in \widehat {A_G(e,s)^0}$ is simple.

\subsection{}\label{sec:3.4} An important role in the determination of the unitary
dual will be played by the $W$-structure of standard and simple
modules. 

The continuation argument in 10.13 in \cite{L3} in conjunction with
(\ref{1.9.4}) shows that
\begin{equation}\label{3.4.1}
X(s,r_0,e)|_W\cong H_\bullet^{\{1\}}(\mathcal P_e,\dot{\C L}), \text{ as } W\times A_G(s,e)\text{-representations}. 
\end{equation}
(By definition, $H_j^{\{1\}}(X,\C L)=H_c^{2dim(X)-j}(X,\C L^*)^*$,
where $H^\bullet_c$ denotes the cohomology with compact support, while
$^*$ denotes the dual vector space or local system.)

Denote
\begin{equation} 
\C N_G=\text{ $G$-conjugacy classes of pairs
 }\{(e,\phi): e\in\fg\text{ nilpotent}, \phi\in \widehat {A_G(e)}\}.
\end{equation}
The {\it
  generalized Springer correspondence} (\cite{L5}) gives, in particular,
an injection
\begin{equation}\label{3.4.2}
\Phi_{\C C,\C L}:\widehat W\hookrightarrow \C N_G.
\end{equation}
(Recall that $W=N_G(L)/L.$) The case of the classical Springer
correspondence appears when
$P$ a Borel subgroup, $\C C$ and $\C L$ are trivial, and $W$ is the
Weyl group of $G$. 

More precisely, let $(e,\phi)$ be a $G$-conjugacy class in $\C N_G$. Set
$\CO_e=G\cdot e,$ and let $\varepsilon_\phi$ be the local system on
$\CO_e$ corresponding to $\phi.$ Then
\cite{L5} (6.2,6.3) attaches to $(e,\phi)$ a unique $G$-conjugacy class $(L',\C C',\C
L')$, where $L'\subset P'$ is a Levi subgroup, $\C C'$ a nilpotent
$L'$-orbit in $\fk l'$, and $\C L'$ is a local system on $\C
C'$, such that:
\begin{align}\label{3.4.3}
 &(a)\ H_c^{dim(\CO_e)-dim(\C C')}(\CO_e\cap(\C C'+\fk n'),\varepsilon_\phi)\neq 0;\\\notag
 &(b)\ P'\text{ is minimal with respect to }(a). 
\end{align}
The local system $\C L'$ on $\C C'$ is constructed from
$\varepsilon_\phi$ (see \cite{L5}, 6.2, for the precise
definition). 
(Note that in fact, (\ref{3.4.3}) gives a  definition of cuspidal: one
could define $\C L'$  to be a cuspidal local system for $G$  if
in (\ref{3.4.3}), $P'=G.$) It is shown in \cite{L5} that all $\C L'$
appearing in this way must be cuspidal for the corresponding $L'$.

If we denote by $\C M_{\C C,\C L}$ the subset of $\C N_G$ attached to
$(L,\C C,\C L)$ by (\ref{3.4.3}), the generalized Springer
correspondence (\ref{3.4.2}) can be reformulated as a bijection $\C
M_{\C C,\C L}\leftrightarrow \widehat W.$ For $(e,\phi)\in \C M_{\C
  C,\C L}$, there corresponds
an irreducible $W$-representation, which we will denote by
$\mu(\CO_e,\phi),$ constructed in $\Hom_{A(e)}[\phi: H_c^{\bullet}(\C
P_e,\dot{\C L})].$  

With respect to the closure
ordering of nilpotent orbits, the smallest orbit, $\CO_{min}$, appearing in
$\text{Im}\Phi_{\C C,\C L}$ is $G\cdot\C C,$ and the largest orbit,
$\CO_{max}$, is the 
Lusztig-Spaltenstein induced orbit $\text{Ind}_L^G(\C C).$ Moreover,
each one of these two orbits supports exactly one local system
(denoted $\phi_{max}$ and $\phi_{min},$ respectively) which enters the parameterization of $\bH.$

The correspondence is normalized such that
\begin{equation}\notag
\mu(\CO_{min},\phi_{min})=triv\text{ and
}\mu(\CO_{max},\phi_{max})=sgn.
\end{equation}

\subsection{}\label{sec:3.5} The previous discussion gives a classification of simple
$\bH$-modules via {\it lowest $W$-types}. This is the analogue of
Vogan's classification by lowest $K$-types for real reductive groups
(\cite{green}), and it appeared in the setting of equal parameter
affine Hecke algebras in \cite{BM1}.
 
Fix $\psi\in \widehat {A_G(s,e)}$. If
$\phi\in \widehat {A_G(e)^0}$  and 
\begin{equation}\label{2.3.4}
 \text{Hom}_W[(\mu(\CO_e,\phi):X(s,r_0,e,\psi)]\neq
 0,\end{equation}
then we will call $\mu(\CO,\phi)$ a {\it lowest $W$-type} for
$X(s,r_0,e,\psi)$. 

\begin{proposition}\label{p:3.5}
The simple module $\overline X(s,r_0,e,\psi)$ is the unique
composition factor of $X(s,r_0,e,\psi)$ which contains the lowest $W$-types $\mu(\CO_e,\phi)$
  with multiplicity $[\phi\mid_{A_G(s,e)}\ :\ \psi].$ Moreover, if
  $\mu'$ is a $W$-type appearing in $\overline X(s,r_0,e,\psi)$,
  different than the lowest $W$-types, then $\Phi_{\C C,\C L}(\mu')=(\CO',\phi'),$
  for some $\CO'$ such that $\CO\subset \overline
  {\CO'}\setminus\CO'.$ 
\end{proposition}

\begin{proof}
It follows directly from theorem \ref{t:3.3} and the discussion in
section \ref{sec:3.4}, particularly, equation (\ref{3.4.1}).
\end{proof}

\subsection{}\label{sec:3.6} If $s\in\fg$ is semisimple, let $s=s_{hyp}+s_{ell}$ denote its
decomposition into hyperbolic and elliptic parts.

\begin{theorem}[\cite{L4}] A simple $\bH$-module $\overline
  X(s,r_0,e,\psi)$ is tempered if and only if $\{s_{hyp},e\}$ can be
  embedded into a Lie triple of $\CO_e.$ In this case $\overline
  X(s,r_0,e,\psi)=X(s,r_0,e,\psi).$ 

If in addition, $\CO_e$ is a
  distinguished nilpotent orbit, then $
  X(s,r_0,e,\psi)$ is a discrete series.
\end{theorem}

Note that, if $(s,e,\phi)$ is the geometric parameter of a real
($s_{ell}=0$) tempered simple $\bH$-module, $A_G(s,e)=A_G(e),$ and
$X(s,r_0,e,\phi)$ has a unique lowest $W$-type, namely
$\mu(\CO_e,\phi)$. In addition,  $\mu(\CO_e,\phi)$ has
multiplicity one. We need one more fact, in the notation before (\ref{3.4.3}): 
\begin{equation}
 H_c^{dim(\CO_e)-dim(\C C')}(\CO_e\cap(\C C'+\fk
 n'),\varepsilon_\phi)\neq 0 \text{ if and only if }H_c^{*}(\CO_e\cap(\C C'+\fk
 n'),\varepsilon_\phi)\neq 0.
\end{equation} 
For Borel-Moore homology and the classical Springer correspondence,
this fact was proved first in \cite{Sh1} (theorem 2.5) for
classical groups and \cite{Sh2} for $F_4$,
\cite{BS} (section 2) for $E_6,E_7,E_8$, and for $G_2$ it follows from
\cite{Spr1} (7.16). In the generality needed here, it is proved in
\cite{L7}, theorem 24.4.(d). The proof is a consequence of the
algorithm for computing Green functions.

Therefore, we have a bijection between $\widehat W$ and simple tempered
$\bH$-modules with real central character, when $\bH$ is of geometric
type. Combining this with proposition \ref{p:3.5}, we obtain the following result.

\begin{corollary}\label{c:ind}
Assume $\bH$ is of geometric type. The set of simple tempered $\bH$-modules with real central character
is linearly independent in the Grothendieck group of $W$. 
\end{corollary}

\noindent{\bf Remark.} In the Grothendieck group of $W$, one can
consider two sets: one set given by the $W$-characters of tempered
$\bH$-modules with real central character, and the other set, the
natural basis of irreducible $W$-representations. The statements
above say that the first set is also a basis, and that, in some
ordering (coming from the lowest $W$-type map), the change of
bases matrix is uni-triangular.
It is natural to expect that this should hold for arbitrary
parameters Hecke algebras $\bH.$ For $\bH$
of type $G_2$ with arbitrary parameters, this correspondence is part
of \cite{KR}. When $\bH$ is of type $B/C,$ the
cases not covered here are all accounted for by the recent paper
\cite{Ka} and the geometry of the ``exotic nilpotent cone''. Since a
geometric framework exists, this result will most likely follow, but
the ``exotic'' lowest $W$-type correspondence may not be compatible
with the one constructed for Lusztig's parameters. For multi-laced
Hecke algebras with unequal parameters (including $F_4$), the
results of \cite{OS} may also imply this statement.

\section{Hermitian forms}\label{sec:4}

We retain the notation from the previous sections. The graded Hecke
algebra $\bH$ will be assumed geometric and all parameters and central
characters are assumed to be real. We will fix $r_0=1/2$ and drop it
from the notation. (It is sufficient to determine the unitary dual of
$\bH$ for one particular value $r_0\neq 0.$) 

\subsection{} Let $(M,\sigma,\nu)$, $\nu$ real, be a Hermitian Langlands
parameter as in section \ref{p:1.5}. We constructed the intertwining
operator 
\begin{equation}\label{4.1.1}
\CA(M,\sigma,\nu):X(M,\sigma,\nu)\to I(M,\sigma\otimes\bC_{-\nu}),\  x\otimes
(v\otimes 1_\nu)\mapsto xr_{w}\otimes (a_w(v)\otimes 1_{-\nu}),
\end{equation}
where $w$ is such that $wM=M,$ $a_w:w\sigma\overset{\cong}\longrightarrow
\sigma$, and $w\nu=-\nu.$ 

By the results in the previous section, $X(M,\sigma,\nu)$ contains a
special set of $W$-representations, the lowest $W$-types. It is an
empirical fact that every $L(M,\sigma,\nu)$ contains one lowest $W$-type
$\mu_0$ with multiplicity $1$. (This is verified case by case for
the exceptional Hecke algebra, for those of classical types this is
automatic from the fact that the component groups $A_G(e)$ of
nilpotent elements $e$ are always abelian.) 

As a $\Bbb C[W]$-module, 
\begin{equation}\label{2.5.2}
I(M,\sigma\otimes\bC_\nu)\mid_W=\Bbb C[W]\otimes_{\Bbb C[W(M)]} (\sigma\mid_W).
\end{equation}
 For any $W$-type $(\mu,V_\mu)$, $\CA(M,\sigma,\nu)$ induces an operator
\begin{equation}\label{2.5.3}
r_\mu(M,\sigma,\nu):\ \text {Hom}_{W}(\mu,\Bbb
C[W]\otimes_{\Bbb C[W(M)]} \sigma)\to\text {Hom}_{W}(\mu,\Bbb
C[W]\otimes_{\Bbb C[W(M)]} \sigma).
\end{equation}
By Frobenius reciprocity, 
\begin{equation}\label{2.5.4}
\text {Hom}_{W}(\mu,\Bbb C[W]\otimes_{\Bbb C[W(M)]} \sigma)\cong \text
      {Hom}_{W(M)}(\mu, \sigma).
\end{equation}
In conclusion, $\CA(M,\sigma,\nu)$ gives rise to operators
\begin{equation}\label{2.5.5}
r_\mu(M,\sigma,\nu):\ \text {Hom}_{\Bbb C[W(M)]}(\mu, \sigma)\to \text
{Hom}_{\Bbb C[W(M)]}(\mu, \sigma).
\end{equation}

The operator $r_{\mu_0}(M,\sigma,\nu)$ is a scalar, and we normalize
the intertwining operator $\CA(M,\sigma,\nu)$ so that this scalar is $1.$

Recall the map $\ep:\bH\to\bA$ (section \ref{p:1.5}). We denote by $\ep(x)(\nu)$, the
evaluation of an element $\ep(x)\in\bA=S(\fh^*)$ at $\nu\in \fh.$  

\begin{theorem}[\cite{BM3}]\label{Hermitian} Let $(M,\sigma,\nu)$ be a Hermitian
  Langlands parameter. 
\begin{enumerate}
 
\item The map $\CA(M,\sigma,\nu)$ is an intertwining operator. 

\item The image of the operator $\CA(M,\sigma,\nu)$ is  $L(M,\sigma,\nu)$ and
  the Hermitian form on $L(M,\sigma,\nu)$ is
  given by: 
\begin{align}\notag
\langle t_x\otimes (v_x\otimes 1_\nu),t_y\otimes (v_y\otimes
  1_\nu)\rangle&=\langle t_x\otimes (v_x\otimes 1_\nu), t_yr_{w_m}\otimes
  (a_m(v_y)\otimes 1_{-\nu})\rangle_h,\\\notag
&=\langle\ep(t_y^*t_xr_{w_m})(\nu)a_m(v_x),v_y\rangle_M.
\end{align}

\end{enumerate}
\end{theorem}

\noindent The discussion in this section can be summarized in the following
corollary.

\begin{corollary}
A Langlands parameter $(M,\sigma,\nu)$, $\nu$ real, is unitary if and only if the
following two conditions are satisfied:
\begin{enumerate}
\item $w_mM=M$, $w_mV\cong V,$ $w_m\nu=-\nu$;
\item the normalized operators $r_\mu(M,\sigma,\nu)$ are positive
  semidefinite  for all $\mu\in\widehat W,$ such that
  $\Hom_{W(M)}[\mu:V]\neq 0.$  
\end{enumerate}
\end{corollary}

One of the main tools for showing
$\bH$-modules are {\it not} unitary is to compute the signature of
certain $r_\mu(M,\sigma,\nu).$

\subsection{}\label{sec:4.2} We consider now the particular case of spherical
modules. 

\begin{definition}\label{d:4.2}
An $\bH$-module $V$ is called {\it spherical} if $\Hom_W[V,triv]\neq
0.$ It is called {\it generic} if $\Hom_W[V,sgn]\neq 0.$ 
\end{definition}

In the case of Hecke algebra with equal parameters, these definitions
are motivated by the correspondence with the representations of split
$p$-adic groups with Iwahori fixed vectors (\cite{BM4,Re2}). In that case, generic
$\bH$-modules correspond to 
representations of the $p$-adic group admitting Whittaker models, and
spherical $\bH$-modules to representations having fixed vectors under
a hyperspecial maximal compact open subgroup. We
emphasize that for a split adjoint $p$-adic group, the unipotent
representations admitting Whittaker models have necessarily Iwahori
fixed vectors (\cite{Re2}).

The Iwahori-Matsumoto involution, $IM$, defined by
\begin{equation}
  \label{eq:4.2.1}
  \begin{aligned}
    &IM(t_w):=(-1)^{l(w)}t_w,\\
    &IM(\om):=-\om,\quad \om\in\fk h^*,
  \end{aligned}
\end{equation}
interchanges spherical and generic modules. From (\ref{1.3.1}), it
follows that $IM$ commutes with the $*$-operation on $\bH.$ From
(\ref{1.3.2}), it follows immediately then that $IM$ preserves
Hermitian and unitary modules.

Let $X(\nu)=\bH\otimes_\bA\bC_\nu$, $\nu\in \fh_\bR,$ be a principal
series module. Clearly, $X(\nu)\cong \bC[W],$ as $W$-modules, and in
particular, $X(\nu)$ has a unique subquotient containing the $triv$
representation of $W$. Denote this by $L(\nu).$ Let $\fh_\bR^+$ denote
the dominant parameters in $\fh_\bR.$ The Langlands
classification and the considerations about intertwining operators
presented in the previous sections are summarized as follows.

\begin{proposition}\label{p:4.2} Assume $\nu\in \fh_\bR^+.$ 

\noindent (1) The principal series $X(\nu)$ 
has a unique irreducible quotient, $L(\nu)$, which is
spherical. Moreover, every spherical $\bH$-module appears in this way.

\noindent (2) A spherical module $L(\nu)$ is Hermitian if and only if
$w_0\nu=-\nu.$ 

\noindent (3) If $w_0\nu=-\nu$, the image of the intertwining operator 
$$A(\nu):X(\nu)\to X(-\nu),\ A(\nu)(x\otimes 1_\nu)=xr_{w_0}\otimes 1_{-\nu},$$
normalized so that it is $+1$ on the spherical vector, is $L(\nu).$ 
\end{proposition}

In particular, $X(\nu)$ is reducible if and only if
$\langle\al,\nu\rangle=c_\al,$ for some $\al\in R^+.$ For every
$\mu\in \widehat W,$ let 
\begin{equation}
r_\mu(\nu):\mu^*\to\mu^*
\end{equation}
 be the operator defined by (\ref{2.5.5}).  Let
$w_0=s_1\dotsb s_k$ be a reduced decomposition of $w_0$ ($k=|R^+|$),
where $s_j=s_{\al_j},$ $\al_j\in \Pi$. The operator
$r_\mu(\nu)$ has a decomposition 
\begin{equation}
r_\mu(\nu)=r_{\mu,s_1}(s_2\dotsb s_k\nu)\cdot
r_{\mu,s_2}(s_3\dotsb s_k\nu)\dotsb r_{\mu,s_k}(\nu),
\end{equation}
where each $r_{\mu,s_j}(\nu')$ is determined by the equation
\begin{equation}\label{eq:4.2.2}
r_{\mu,s_j}(\nu')=\left\{\begin{matrix} 1, &\text{ on the
    }(+1)\text{-eigenspace of } s_j\text{ on } \mu^*\\
\frac {c_{\al_j}-\nu'}{c_{\al_j}+\nu'},  &\text{ on the
    }(-1)\text{-eigenspace of } s_j\text{ on } \mu^*.\end{matrix}\right.
\end{equation}

\subsection{} As in corollary \ref{Hermitian}, a spherical Hermitian module
$L(\nu)$ is unitary if and only if $r_\mu(\nu)$ are positive
semidefinite for all $\mu\in \widehat W.$ 

\begin{definition}
Define the {\it $0$-complementary series} to be the set $\{\nu\in
\fh_\bR^+: X(\nu)=L(\nu) \text{ unitary}\}.$
\end{definition}

The $0$-complementary series when $\bH$ is of type $B_n/C_n$, with
arbitrary unequal parameters, were determined in \cite{BC}. For type
$G_2$, using the machinery presented in the previous section, in
particular the signatures of $r_\mu(\nu)$, it is an easy
calculation. We record the result, without proof, next. We use the
simple roots
\begin{equation}\notag
\al_1=(\frac 23,-\frac 13,-\frac 13)\text{ and } \al_2=(-1,1,0).
\end{equation}
The Hecke algebra $\bH(G_2)$ has $c_{\al_1}=1$ and
$c_{\al_2}=c>0.$ 

\begin{proposition} Let 
  $\nu=(\nu_1,\nu_1+\nu_2,-2\nu_1-\nu_2),$ $\nu_1\ge 0,$ $\nu_2\ge 0,$
  be a dominant (spherical) parameter for $\bH(G_2)$.  There are
  nine cases depending on the parameter $c$. The $0$-complementary
  series are as follows.

\noindent (1) $0<c<1:$
$\{3\nu_1+2\nu_2<c\}\cup\{3\nu_1+\nu_2>c,\nu_2<c,2\nu_1+\nu_2<1\}.$ 

\noindent (2) $c=1:$
$\{3\nu_1+2\nu_2<1\}\cup\{3\nu_1+\nu_2>1,2\nu_1+\nu_2<1\}.$ 

\noindent (3) $1<c<\frac 32:$
$\{3\nu_1+2\nu_2<c\}\cup\{3\nu_1+\nu_2>c,2\nu_1+\nu_2<1\}.$ 

\noindent (4) $c=\frac 32:$ $\{3\nu_1+2\nu_2<\frac 32\}.$ 

\noindent (5) $\frac 32<c<2:$ $\{3\nu_1+2\nu_2<c, 2\nu_1+\nu_2<1\}.$ 

\noindent (6) $c=2:$ $\{2\nu_1+\nu_2<1\}.$

\noindent (7) $2<c<3:$
$\{2\nu_1+\nu_2<1\}\cup\{\nu_1+\nu_2>1,3\nu_1+2\nu_2<c\}.$ 

\noindent (8) $c=3:$ $\{2\nu_1+\nu_2<1\}\cup\{\nu_1+\nu_2>1,3\nu_1+2\nu_2<3\}.$

\noindent (9) $c>3:$ $\{2\nu_1+\nu_2<1\}\cup\{\nu_1+\nu_2>1,\nu_1<1,
3\nu_1+2\nu_2<c\}.$  

In addition, the only other unitary spherical representations are the
endpoints of the $0$-complementary series, and the trivial
representation, for which $(\nu_1,\nu_2)=(1,c).$ 
\end{proposition}

In this notation, the Hecke algebras of type $G_2$ which appear
geometrically in \cite{L6} are the cases $c=1$ (equal parameters) and
$c=9$. The latter corresponds to a cuspidal local system on the
nilpotent $2A_2$ in $E_6.$ 

\subsection{}\label{sec:4.4} We present the unitary dual with real
central character for the geometric Hecke algebra of type $G_2$ in
more detail. In the notation as above, this is $\bH(G_2,(1,9)),$ but
it is better for our parameterization to think of it as the isomorphic
$\bH(G_2,(3,1)).$ The geometry is attached to the cuspidal local
system on $2A_2$ in $E_6.$ The nilpotent orbits $\CO$ below and their
closure relations are as in \cite{Ca}, and the notation for
$W(G_2)$-representations is as in \cite{A}.

\begin{proposition} The unitary irreducible modules with real central
  character for $\bH(G_2,(3,1))$ are
\begin{small}
\begin{center}
\begin{longtable}{|c|c|c|c|}
\caption{Unitary dual for $\bH(G_2,(3,1))$}
\label{table:G2}\\
\hline
\multicolumn{1}{|c|}{$\mathbf{\CO}$ } &\multicolumn{1}{c|}{\bf Central
  character} &\multicolumn{1}{c|}{\bf LWT} &\multicolumn{1}{c|}{\bf
    Unitary parameters}\\\hline
\endfirsthead
\multicolumn{4}{c}%
{{ \tablename\ \thetable{} -- continued from previous page}}
\\
\endhead
\hline\hline
\endlastfoot

 $2A_2$ &$(\nu_1,\nu_1+\nu_2,-2\nu_1-\nu_2)$ &$1_1$
  &$\{0\le\nu_1,0\le\nu_2,3\nu_1+2\nu_2\le 1\}\cup$\\
 &&&$\{0\le\nu_1,3\nu_1+\nu_2\ge 1, \nu_2\ge 1, 2\nu_1+\nu_2\le 3\}$\\

\hline

 $2A_2+A_1$ &$(-\frac 12+\nu,\frac 12+\nu,-2\nu)$ &$1_3$
&$\{0\le\nu\le \frac 32\}$\\

\hline

 ${A_5}$ &$(3,-\frac 32+\nu,-\frac 32-\nu)$ &$2_2$ &$\{0\le\nu\le
\frac 12\}$\\

\hline

 ${E_6(a_3)}$ &$(1,2,-3)$ &$2_1$
&$D.S.$\\

\hline

 ${E_6(a_1)}$ &$(2,3,-5)$
&$1_4$ &$D.S.$\\

\hline

 ${E_6}$ &$(3,4,-7)$
&$1_2$ &$St$\\

\end{longtable}
\end{center}
\end{small}

\end{proposition}

Here $D.S.$ stands for discrete series, and $St$ for the Steinberg
module (which is a discrete series). The proof is by direct computations, similar, but much easier than
those in the next section. We don't include the explicit
calculations. We remark that the unitary pieces attached to the
nilpotent orbits $2A_2+A_1$ and $A_5$ match precisely the spherical
unitary dual for the Hecke algebra of type $A_1$ with parameter $3$,
respectively $1.$ The centralizers of these two orbits are both of
type $A_1,$ so this is consistent with our general philosophy.

To help visualize the answer, we include a picture of the spherical
unitary dual in this case, see figure \ref{fig:G2}.

\begin{figure}[h]\label{fig:G2}
\input{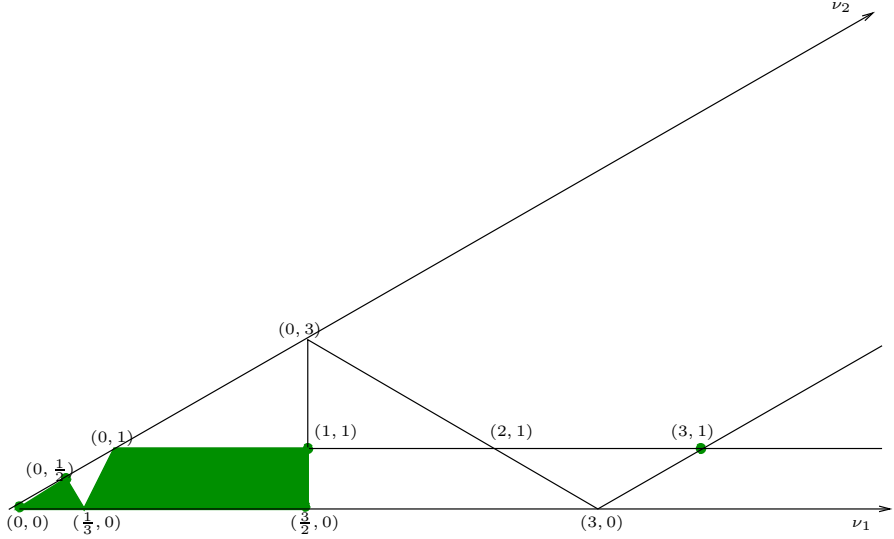}
\caption{Spherical unitary dual for $\bH(G_2,(3,1))$}
\end{figure}

\section{An example: geometric Hecke algebra of type $F_4$}\label{sec:5}

We present the classification of the unitary dual for the Hecke
algebra $\Bbb H$ with unequal parameters of type $F_4$, 
which appears geometrically in the classification of \cite{L6}. It has
labels $1$ 
on the long roots, and $2$ on the short roots. We use the following
choice for the roots:
$$\al_1=\ep_1-\ep_2-\ep_3-\ep_4,\quad \al_2=2\ep_4,\quad
\al_3=\ep_3-\ep_4,\quad \al_4=\ep_2-\ep_3.$$
This algebra is attached to the cuspidal local system on the nilpotent
orbit  $(3A_1)''$ in $E_7.$ The simple modules are parameterized
therefore by a subset of the nilpotent orbits in $E_7$, larger in the
closure ordering than $(3A_1)''.$ We list these nilpotent orbits $\CO$, each
with the corresponding infinitesimal character in $\Bbb H$, 
the $W$-types coming from the generalized Springer correspondence (LWT), and
the centralizer $\fz(\CO)$ in $E_7.$ The generalized Springer
correspondence for this case was computed in \cite{Sp}. The closure
ordering for complex nilpotent orbits in $E_7$ can be found in
\cite{Ca}. (We mention that there is a typographical error in
\cite{Ca}, the nilpotent orbits $D_6(a_2)$ and $D_5(a_1)+A_1$ are
comparable, see figure \ref{3A1nil} and \cite{Sp2}.)

\begin{small}
\begin{center}
\begin{longtable}{|c|c|c|c|}
\caption{Nilpotent orbits for $(3A_1)''$ in $E_7$}
\label{table:nilpotents}\\
\hline
\multicolumn{1}{|c|}{$\mathbf{\CO}$ } &\multicolumn{1}{c|}{\bf Central
  character} &\multicolumn{1}{c|}{\bf $\mathbf{\fz(\CO)}$ }&\multicolumn{1}{c|}{\bf LWT}\\\hline
\endfirsthead
\multicolumn{4}{c}%
{{ \tablename\ \thetable{} -- continued from previous page}}
\\
\endhead
\hline\hline
\endlastfoot

 ${(3A_1)''}$ &$(\nu_1,\nu_2,\nu_3,\nu_4)$ &$F_4$
  &$1_1$\\

\hline

 ${4A_1}$ &$(\nu_1,\nu_2,\nu_3,\frac 12)$ &$C_3$ &$2_3$\\

\hline

 ${A_2+3A_1}$ &$(\frac 12,-\frac 12,-\frac 12,\frac
12)+\nu_1(2,1,1,0)+\nu_2(1,1,0,0)$ &$G_2$ &$1_3$\\

\hline

 ${(A_3+A_1)''}$ &$(0,0,1,-1)+\nu_1(\frac 12,\frac
12,0,0)+\nu_2(\frac 12,-\frac 12,0,0)$ &$B_3$
&$4_2$\\
 &$+\nu_3(0,0,\frac 12,\frac 12)$ &&\\
\hline

 ${A_3+2A_1}$ &$(\nu_1,1+\nu_2,-1+\nu_2,\frac 12)$
&$2A_1$ &$8_3$\\

\hline

 ${D_4(a_1)+A_1}$ &$(\nu_1,\nu_2,\frac 32,\frac 12)$
&$2A_1$ &$9_1,2_1$\\

\hline

 ${A_3+A_2+A_1}$ &$(\frac 12,\frac 12,-\frac 32,\frac
12)+\nu(2,1,1,0)$ &$A_1$ &$4_4$\\

\hline

 ${D_4+A_1}$ &$(\nu_1,\nu_2,\frac 52,\frac 12)$
&$B_2$ &$9_3$\\

\hline

 ${(A_5)''}$ &$(\nu_2+\frac
	 {3\nu_1}2,2+\frac{\nu_1}2,\frac{\nu_1}2,-2+\frac{\nu_1}2)$
	& $G_2$ &$8_1$\\

\hline

 ${D_5(a_1)+A_1}$ &$(\frac 32+\nu,-\frac 32+\nu,\frac
32,\frac 12)$ &$A_1$ &$4_1$\\

\hline

 ${A_5+A_1}$ &$(\frac 14,\frac 74,-\frac 14,-\frac
94)+\nu(\frac 32,\frac 12,\frac 12,\frac 12)$ &$A_1$ &$6_1$\\

 \hline

 ${D_6(a_2)}$ &$(\nu,\frac 52,\frac 32,\frac 12)$
&$A_1$ &$16_1$\\

\hline

 ${E_7(a_5)}$ &$(\frac 52,\frac 32,\frac 12,\frac 12)$
&$1$ &$12_1,6_2$\\

\hline

 ${D_5+A_1}$ &$(2+\nu,-2+\nu,\frac 52,\frac 12)$
&$A_1$ &$2_2$\\

\hline

 ${D_6(a_1)}$ &$(\nu,\frac 72,\frac 32,\frac 12)$
&$A_1$ &$9_2$\\

\hline

 ${E_7(a_4)}$ &$(\frac 72,\frac 32,\frac 12,\frac 12)$
&$1$ &$4_3,8_2$\\

\hline

 ${D_6}$ &$(\nu,\frac 92,\frac 52,\frac 12)$
&$A_1$ &$9_4$\\

\hline

 ${E_7(a_3)}$ &$(\frac 92,\frac 52,\frac 12,\frac 12)$
&$1$ &$8_4,1_2$\\

\hline

 ${E_7(a_2)}$ &$(\frac {11}2,\frac 52,\frac 32,\frac
12)$ &$1$ &$4_5$\\

\hline

 ${E_7(a_1)}$ &$(\frac {13}2,\frac 72,\frac 32,\frac
12)$ &$1$ &$2_4$\\

\hline

 ${E_7}$ &$(\frac {17}2,\frac 92,\frac 52,\frac
12)$ &$1$ &$1_4$

\end{longtable}
\end{center}
\end{small}

\begin{small}
\begin{figure}[h]
\scalebox{.7}{\input{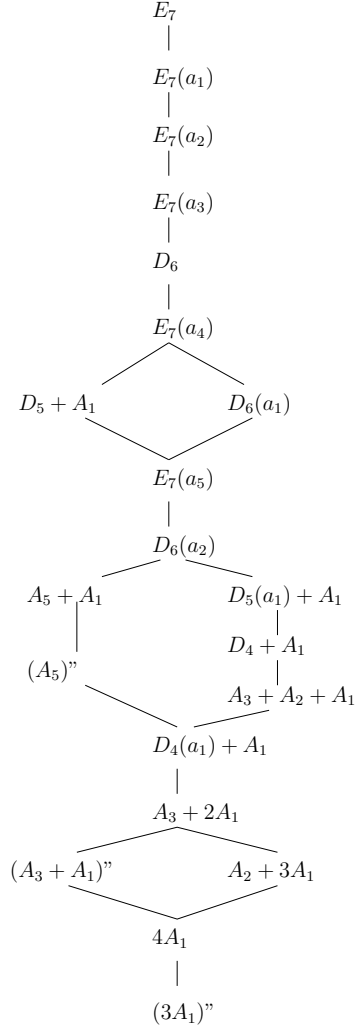}}
\caption{Nilpotent orbits parameterizing the $\bH(F_4,(1,2))$-modules.}\label{3A1nil}
\end{figure}
\end{small}

\subsection{Results}\label{sec:5.1}

The spherical unitary dual is listed in table \ref{table:5.2}. These
are the unitary modules parameterized by $\CO_{min}=(3A_1)''.$ For the rest of
the nilpotents, the unitary sets are presented next. For every
nilpotent $\CO$ in table \ref{table:nilpotents}, with
$\{e,h,f\}\subset \CO$ a
Lie triple, and $\phi\in
\widehat {A_G(e)^0}$, let 
\begin{align}\notag
\CU(\CO,\phi)=\{(\nu,\psi): &\ s=\frac h2+\nu, \phi\in \widehat
{A_G(s,e)^0}, \Hom_{A(s,e)}[\psi:\phi]\neq 0,\\\label{eq:5.1.1}
&\ \overline X(s,e,\psi)
\text{ is unitary} \}
\end{align}
denote the set of unitary parameters associated to $(\CO,\phi).$ The
goal is to relate $\CU(\CO,\phi)$ with the spherical unitary dual $\C
S\C U$
of a Hecke algebra of different type.

\begin{theorem}\label{t:5.1} With the notation as before (particularly
  table \ref{table:nilpotents} and (\ref{eq:5.1.1})), the unitary
  irreducible modules with real central character for $\bH(F_4,(1,2))$ are:

\begin{small}
\begin{center}
\begin{longtable}{|c|c|c|}
\caption{Unitary dual of $\bH(F_4,(1,2))$}
\label{table:unitary}\\
\hline
\multicolumn{1}{|c|}{$\mathbf\CO$} &\multicolumn{1}{c|}{$\mathbf
 {\mu(\CO,\phi)}$} &\multicolumn{1}{c|}{$\CU(\CO,\phi)$}\\
\hline 
\endfirsthead
\multicolumn{3}{c}%
{{ \tablename\ \thetable{} -- continued from previous page}}
\\
\endhead
\hline\hline
\endlastfoot

${E_7}$ &$1_4$ & St\\

\hline
${E_7(a_1)}$ &$2_4$ & D.S.\\

\hline
${E_7(a_2)}$ &$4_5$ & D.S.\\

\hline

${E_7(a_3)}$ &$8_4$ &D.S.\\
${E_7(a_3)}$ &$1_2$ &D.S.\\

\hline

$D_6$ &$9_4$ &$\CS\CU(A_1,(1))$\\

\hline

$E_7(a_4)$ &$4_3$ &D.S.\\
           &$2_2$ &D.S.\\

\hline

$D_6(a_1)$ &$9_2$ &$\CS\CU(A_1,(1))$\\

\hline

$D_5+A_1$ &$8_2$ &$\CS\CU(A_1,(1))$\\

\hline

$E_7(a_5)$ &$12_1$ &D.S.\\
           &$6_2$  &D.S.\\

\hline

$D_6(a_2)$ &$16_1$ &$\CS\CU(A_1,(1))$\\

\hline

$A_5+A_1$ &$6_1$ &$\CS\CU(A_1,(1))$\\

\hline

$D_5(a_1)+A_1$ &$4_1$ &$0\le\nu\le 1$\\

\hline

$(A_5)''$ &$8_1$ &$\CS\CU(G_2,(1,1))$\\

\hline

$D_4+A_1$ &$9_3$     &$\CS\CU(B_2,(1,\frac 32))$\\

\hline

$A_3A_2A_1$ &$4_4$   &$\CS\CU(A_1,(2))$ \\

\hline

$D_4(a_1)+A_1$ &$9_1$  &$\{0\le\nu_2<\nu_1\le 3-\nu_2,\
\nu_1\le\frac 52\}\cup \{0\le\nu_2=\nu_1\le\frac 32\}$  \\

             &$2_1$   & $\{0\le\nu_2<\nu_1\le 3-\nu_2,\
\nu_1\le\frac 52\}\cup \{0\le\nu_2=\nu_1\le\frac 52\}$  \\

\hline

$A_3+2A_1$ &$8_3$     & $\{0\le\nu_1\le\frac 12,\ \nu_1+2\nu_2\le\frac
32\}\ \cup$   \\
&&$\{0\le\nu_1\le\frac 12,\ 0\le\nu_2\le\frac 32,\ 2\nu_2-\nu_1\ge
\frac 32\}$\\

\hline

$(A_3+A_1)''$ &$4_2$   & $\CS\CU(B_3,(1,1))\cup\{(1+\nu,\nu,-1+\nu): 1<\nu<2\}$    \\

\hline

$A_2+3A_1$    &$1_3$   & $\CS\CU(G_2,(2,1))\cup\{\nu_1+\nu_2=1,\
0<\nu_2<1\}\ \cup$    \\
&&$\{\nu_1=1,\ 0\le\nu_2\le 1\}\cup\{3\nu_1+\nu_2=2,\ \frac
12<\nu_2\le 2\}$\\
\hline

$4A_1$       &$2_3$    &$\{\nu_1+\nu_2+\nu_3\le\frac
32\}\cup\{\nu_1+\nu_2-\nu_3\ge \frac 32,\nu_1<\frac 32,\nu_1+\nu_2<2\}$    \\

&&$\cup \{(\nu_1,\nu_2,\frac 32): \nu_1+\nu_2\le 3,\nu_1\le\frac
52\}$\\

&&$\cup\{(\nu_1,1+\nu_2,-1+\nu_2): 0\le\nu_1\le\frac 12,
0\le\nu_2\le\frac 12\}$\\
&&$\cup\{(\nu_1,1+\nu_2,-1+\nu_2):
\nu_1+2\nu_2=\frac 32, \frac 12\le\nu_1\le\frac 32\}$\\

&&$\cup \{(\nu,\frac 32,\frac 12): 0\le\nu\le\frac
52\}\cup\{(1+\nu,-1+\nu,\frac 32): 0\le\nu<\frac 32\}$\\

&&$\cup\{(2+\nu,\nu,-2+\nu): 0\le\nu<\frac 12\}\cup\{(\nu,\frac
72,\frac 32): 0\le\nu<\frac 12\}$\\
&&$\cup\{(\frac 52,\frac 32,\frac 12)\}\cup\{(\frac 72,\frac 32,\frac
12)\}\cup\{(\frac {13}2,\frac 72,\frac 32)\}$\\

\hline

$(3A_1)''$   &$1_1$    & $\CS\CU(F_4,(1,2))$, see table \ref{table:5.2}\\

\hline

\end{longtable}
\end{center}
\end{small}

\end{theorem}

We will present, in detail, in the next two sections, the
calculations giving the 
unitary dual of $\bH(F_4,(1,2)).$ The discussion is organized by
nilpotent orbits.

For the distinguished orbits, there is nothing to do. They parameterize
discrete series, and the lowest $W$-types and infinitesimal characters
can be read from tables \ref{table:nilpotents} and \ref{table:unitary}.

\subsection{Maximal parabolic cases}\label{sec:5.1a} First we treat
the cases of those nilpotent orbits which, in the Bala-Carter
classification, correspond to Levi subalgebras of maximal parabolic
subalgebras. In terms of the (classical) Langlands classification,
they correspond to the induced modules from discrete series on Levi
components of maximal parabolic subalgebras. For any Hecke algebra
with arbitrary parameters, the Steinberg module $St$ is a discrete
series. In this case, our calculation is more general, and we record
it for general parameter $c$. The classification of discrete series
depends in general of the particular value of $c$; for the
non-Steinberg discrete series, we restrict to the case $c=2.$

\begin{proposition}The Langlands quotient $\overline X(M,St,\nu)$ is unitary if
  and only if: 

\medskip
\begin{small}
\noindent\begin{tabular}{|c|c|c|c|}
\hline
Levi $M$ &Nilpotent $\CO$ \footnotemark &Central character $\chi$
&Unitary\\
\hline
$C_3$ &$D_6$ &$(\nu,2c+\frac 12,c+\frac 12,\frac 12)$ &$0\le\nu\le
|c-\frac 32|$\\
\hline
$B_3$ &$D_5+A_1$ &$(\frac c2+1+\nu,-\frac c2-1+\nu,c+\frac 12,\frac
12)$ &$0\le\nu\le|\frac {c-3}2|$\\
\hline
$\wti A_2+A_1$ &$A_5+A_1$ &$(\frac 14,c-\frac 14,-\frac 14,-\frac
14-c)+\nu(\frac 32,\frac 12,\frac 12,\frac 12)$
&$0\le\nu\le\min\{\frac 12,|2c-\frac 32|\}$\\
\hline
$A_2+\wti A_1$ &$A_3A_2A_1$ &$(\frac 12,\frac {c-1}2,-\frac
{c+1}2,\frac 12)+\nu(2,1,1,0)$ &$0\le\nu\le\min\{\frac c2,|\frac {3c-2}2|\}$\\
\hline
\end{tabular}
\end{small}
\footnotetext{when $c=2.$}
\medskip

In addition, we have the following complementary series for
   $\overline X(M,\sigma,\nu),$ where $\sigma\neq St$ is a discrete series:

\medskip
\begin{small}
\noindent\begin{tabular}{|c|c|c|c|}
\hline
Levi $M$ &Nilpotent $\CO$ &Central character $\chi$
&Unitary\\
\hline
$C_3$ &$D_6(a_1)$ &$(\nu,\frac 72,\frac 32,\frac 12)$
&$0\le\nu\le\frac 12$\\
      &$D_6(a_2)$ &$(\nu,\frac 52,\frac 32,\frac 12)$
&$0\le\nu\le\frac 12$\\
\hline
$B_3$ &$D_5(a_1)+A_1$ &$(\frac 32+\nu,-\frac 32+\nu,\frac 32,\frac
12)$ &$0\le\nu\le 1$\\
\hline
\end{tabular}
\end{small}
\end{proposition}

\begin{proof} The proof is case by case. We compute the intertwining
  operators for each standard module restricted to every
  $W$-type. For the necessary calculations, we constructed matrix
  realizations (with rational entries) for all irreducible
  representations of the Weyl group of $F_4$, and used the software
  ``Mathematica'' to multiply matrices in the intertwining operators.

Note that in order to determine the unitarity of the Langlands quotient, we
  would not need {\it all} these intertwining operators, but the
  byproduct of this calculation is that we obtain the $W$-structure of
  various irreducible subquotients which appear in the standard modules
  at reducibility points, in particular (but most importantly) the
  $W$-structure of the discrete series. (This is the Jantzen-type
  filtration from \cite{V1}.)

\medskip

\noindent$\mathbf {D_6:}$ The infinitesimal character is
$(\nu,\frac 92,\frac 52,\frac 12)$, the lowest $W$-type is $9_4,$ and
$\fz=A_1.$ 
The standard module is induced from the Steinberg representation on
the Hecke subalgebra of type $C_3$. The $W$-structure is
$$X(D_6)|_W=9_4+8_4+4_5+2_4+1_4.$$
We calculate the intertwining operators for arbitrary parameter
$c$. In that case, the infinitesimal character is $(\nu,2c+\frac
12,c+\frac 12,\frac 12).$
\begin{small}
\begin{align}\notag
&9_4: 1;\ 
8_4: \frac {c-\frac 32-\nu}{c-\frac 32+\nu};
\ 4_5: \frac {c+\frac 32-\nu}{c+\frac 32+\nu};\ 2_4: \frac {(c-\frac 32-\nu)(3c-\frac 12-\nu)}{(c-\frac 32+\nu)(3c-\frac 12+\nu)};\\
&1_4: \frac {(c+\frac 32-\nu)(3c+\frac 52-\nu)}{(c+\frac 32+\nu)(3c+\frac 52+\nu)}.
\end{align} 
\end{small}
In particular, when $c=2,$ the calculation implies that
$\overline{X}(D_6,\nu)$ is unitary if and only if $0\le\nu\le\frac
12.$

\noindent$\mathbf {D_6(a_1):}$ The infinitesimal character is
$(\nu,\frac 72,\frac 32,\frac 12),$ the lowest $W$-type is $9_2,$ and $\fz=A_1.$ In
terms of Langlands classification, the standard module is induced from
a one-dimensional discrete series in $C_3$. The $W$-structure is
$$X(D_6(a_1))|_W=\text{Ind}_{W(C_3)}^{W(F_4)}(111\times
0)=9_2+4_3+8_4+1_2+2_4.$$ 

The intertwining operators in this case give:
\begin{small}
\begin{align}\notag
&9_2: 1;\
4_3: \frac{\frac 12-\nu}{\frac 12+\nu};\ 
8_4: \frac{\frac 72-\nu}{\frac 72+\nu};\ 
1_2: \frac{(\frac 12-\nu)(\frac 72-\nu)}{(\frac 12+\nu)(\frac 72+\nu)};\\
&2_4: \frac{(\frac 72-\nu)(\frac {13}2-\nu)}{(\frac 72+\nu)(\frac {13}2+\nu)}. 
\end{align}
\end{small}
This indicates that $\overline X(D_6(a_1),\nu)$ is unitary if and only
if $0\le\nu\le \frac 12.$ At $\nu=\frac 12$, the standard module
decomposes as
\begin{align}\notag
X(D_6(a_1),\frac 12)&=\overline X(D_6(a_1),\frac
12)+X(E_7(a_4),\phi_1);\\
\overline X(D_6(a_1),\frac 12)&=9_2+8_4+2_4.\notag
\end{align}

\noindent$\mathbf{D_5+A_1:}$ The infinitesimal character is
$(2+\nu,-2+\nu,\frac 52,\frac 12),$ the lowest $W$-type is $8_2$, and
$\fz=A_1.$  The standard module is induced from the Steinberg representation on
the Hecke subalgebra of type $B_3$. The $W$-structure is
\begin{align}\notag
X(D_5+A_1)|_W=8_2+2_2+9_4+4_5+1_4.\end{align}

We calculate the intertwining operators for arbitrary parameter
$c$. In that case, the infinitesimal character is $(\frac
c2+1+\nu,-\frac c2-1+\nu,c+\frac 12,\frac 12).$ 
\begin{small}
\begin{align}\notag
&8_2: 1;\
2_2: \frac {\frac {c-3}2+\nu}{\frac {c-3}2-\nu};\
9_4: \frac {\frac {3c-1}2-\nu}{\frac {3c-1}2+\nu};\
4_5: \frac {(\frac {3c-1}2-\nu)(\frac {3c+1}2-\nu)}{(\frac {3c-1}2+\nu)(\frac {3c+1}2+\nu)};\\
&1_4: \frac {(\frac {3c-1}2-\nu)(\frac {3c+1}2-\nu)(\frac {5c+3}2-\nu))}{(\frac {3c-1}2+\nu)(\frac {3c+1}2+\nu)(\frac {5c+3}2+\nu)}.
\end{align}
\end{small}
In particular, if $c=2,$ $\overline X(D_5+A_1,\nu)$ is unitary if and
only if $0\le\nu\le \frac 12.$ At $\nu=\frac 12,$ the standard module
decomposes as 
\begin{align}\notag
&X(D_5+A_1,\frac 12)=\overline X(D_5+A_1,\frac 12)+X(E_7(a_4),\phi_2)\\\notag
&\overline X(D_5+A_1,\frac 12)|_W=8_2+9_4+4_5+1_4.\notag
\end{align}

\noindent$\mathbf{D_6(a_2):}$ The infinitesimal character is
$(\nu,\frac 52,\frac 32,\frac 12)$, the lowest $W$-type is $16_1$, and $\fz=A_1$. In
terms of Langlands classification, the standard module is induced from
a four dimensional discrete series on $C_3$. The $W$-structure is
\begin{align}\notag
X(D_6(a_2))|_W&=\text{Ind}_{W(C_3)}^{W(F_4)}(1\times 11+0\times 111)\\ \notag
              &=16_1+2_4+9_2+2\cdot 9_4+2\cdot 8_4+12_1+8_2+2\cdot
              4_5+6_2+1_4.\notag
\end{align}
The intertwining operators are (for the $W$-types with multiplicity
greater than $1$, we only give the determinant):
\begin{small}
\begin{align}\notag
&16_1: 1;\
12_1: \frac {\frac 12-\nu}{\frac 12+\nu};\
6_2: \frac {\frac 12-\nu}{\frac 12+\nu};\
2_4: \frac {(\frac 52-\nu)(\frac 92-\nu)}{(\frac 52+\nu)(\frac
    92+\nu)};\
8_2: \frac {(\frac 12-\nu)(\frac 52-\nu)}{(\frac 12+\nu)(\frac
    52+\nu)};\\\notag
&1_4: \frac {(\frac 12-\nu)(\frac 52-\nu)(\frac 92-\nu)(\frac {11}2-\nu)}{(\frac 12+\nu)(\frac 52+\nu)(\frac 92+\nu)(\frac {11}2+\nu)};\
9_2: \frac {\frac 52-\nu}{\frac 52+\nu};\
8_4: Det=\frac {(\frac 12-\nu)(\frac 52-\nu)(\frac 92-\nu)}{(\frac
    12+\nu)(\frac 52+\nu)(\frac 92+\nu)};\\
&4_5: Det=\frac {(\frac 12-\nu)^2(\frac 52-\nu)(\frac 92-\nu)(\frac
  {11}2-\nu)}{(\frac
    12+\nu)^2(\frac 52+\nu)(\frac 92+\nu)(\frac {11}2+\nu)}.
\end{align}
\end{small}
It follows that $\overline X(D_6(a_2),\nu)$ is unitary if and only if
$0\le\nu\le\frac 12.$ At $\nu=\frac 12,$ the standard module
decomposes as 
\begin{align}\notag
&X(D_6(a_2),\frac 12)=\overline X(D_6(a_2),\frac
12)+X(E_7(a_5),\phi_1)+X(E_7(a_5),\phi_2);\\\notag
&\overline X(D_6(a_2),\frac 12)|_W=16_1+2_4+9_2+8_4+9_4.\notag
\end{align}

\noindent$\mathbf{A_5+A_1}:$ The infinitesimal character is
$(\frac 14,\frac 74,-\frac 14,-\frac 94)+\nu(\frac 32,\frac 12,\frac
12,\frac 12),$ the lowest $W$-type is $6_1$, and $\fz=A_1.$ In terms of
Langlands classification, the standard module is induced from the
Steinberg representation on $\wti A_2+A_1.$ 
\begin{align}\notag
X(A_5+A_1)|_W&=\text{Ind}_{W(\wti A_2+A_1)}^{W(F_4)}((111)\otimes (11))\\ \notag
              &=6_1+16_1+2\cdot 8_4+8_2+4_3+4_5+12_1+9_2+2\cdot 9_4+2_4+1_4.\notag
\end{align}
We compute the intertwining operator for arbitrary parameter $c$. In
that case, the infinitesimal character is $(\frac 14,c-\frac 14,-\frac
14,-\frac 14-c)+\nu(\frac 32,\frac 12,\frac 12,\frac 12).$
\begin{small}
\begin{align}
&6_1: 1;\
16_1: \frac{\frac 12-\nu}{\frac 12+\nu};\
8_2: \frac {(\frac 12-\nu)^2(\frac 32-\nu)}{(\frac 12+\nu)^2(\frac 32+\nu)};\
12_1: \frac{(\frac 12-\nu)^2}{(\frac 12+\nu)^2};\
9_2: \frac {(\frac 12-\nu)(2c-\frac 32-\nu)}{(\frac 12+\nu)(2c-\frac 32+\nu)};\\\notag
&4_3: \frac{2c-\frac 32-\nu}{2c-\frac 32+\nu};\ 
2_4: \frac {(\frac 12-\nu)(c+\frac 12-\nu)(2c-\frac 32-\nu)(2c+\frac
    12-\nu)}{(\frac 12+\nu)(c+\frac 12+\nu)(2c-\frac 32+\nu)(2c+\frac
    12+\nu)};\\\notag
&4_5: \frac {(\frac 12-\nu)^2(\frac 32-\nu)(c+\frac 12-\nu)(2c+\frac 12-\nu)}{(\frac 12+\nu)^2(\frac 32+\nu)(c+\frac 12+\nu)(2c+\frac
    12+\nu)};\\\notag
&1_4: \frac {(\frac 12-\nu)^2(\frac 32-\nu)(c+\frac 12-\nu)(2c+\frac
    12-\nu)(2c+\frac
    32-\nu)}{(\frac 12+\nu)^2(\frac 32-\nu)(c+\frac 12+\nu)(2c+\frac
    12+\nu)(2c+\frac 32+\nu)};\\\notag
&8_4: Det=\frac{(\frac 12-\nu)^3(c+\frac 12-\nu)(2c-\frac
    32-\nu)(2c+\frac 12-\nu)}{(\frac 12+\nu)^3(c+\frac 12+\nu)(2c-\frac
    32+\nu)(2c+\frac 12+\nu)};\\
&9_4: Det=\frac{(\frac 12-\nu)^3(\frac 32-\nu)(c+\frac
    12-\nu)(2c+\frac 12-\nu)}{(\frac 12+\nu)^3(\frac 32+\nu)(c+\frac 12+\nu)(2c+\frac 12+\nu)}.
\end{align}
\end{small}
It follows that in the case $c=2$, $\overline X(A_5+A_1,\nu)$ is
unitary if and only if $0\le\nu\le\frac 12.$ At $\nu=\frac 12,$ the standard
module decomposes as 
\begin{align}\notag
&X(A_5A_1,\frac 12)=\overline X(A_5A_1,\frac 12)+\overline
X(D_6(a_2),\frac 12)+X(E_7(a_5),\phi_1)\\\notag
&\overline X(A_5A_1,\frac 12)=6_1+4_3.\notag
\end{align}

\bigskip

\noindent$\mathbf{D_5(a_1)+A_1}:$ The infinitesimal character is
$(\frac 32+\nu,-\frac 32+\nu,\frac 32,\frac 12),$ the lowest $W$-type is
$4_1$, and $\fz=A_1.$ In terms of Langlands classification, the
standard module is induced from a two-dimensional discrete series on
$B_3$. The $W$-structure is
$$X(D_5(a_1)+A_1))|_W=\text{Ind}_{W(B_3)}^{W(F_4)}(0\times 12)=4_1+16_1+2_4+9_4+9_2+8_4.$$
The intertwining operators are:
\begin{small}
\begin{align}\notag
&4_1: 1;\
16_1: \frac {1-\nu}{1+\nu};\
9_2: \frac {(1-\nu)(2-\nu)}{(1+\nu)(2+\nu)};\
9_4: \frac {(1-\nu)(4-\nu)}{(1+\nu)(4+\nu)};\\
&8_4: \frac {(1-\nu)(2-\nu)(4-\nu)}{(1+\nu)(2+\nu)(4+\nu)};\ 
2_4: \frac {(1-\nu)(2-\nu)(4-\nu)(5-\nu)}{(1+\nu)(2+\nu)(4+\nu)(5+\nu)}.
\end{align}
\end{small}
It follows that $\overline X(D_5(a_1)+A_1)$ is unitary if and only if
$0\le\nu\le 1.$ At $\nu=1$, the standard module decomposes as 
\begin{align}\notag
&X(D_5(a_1)+A_1,1)=\overline X(D_5(a_1)+A_1,1)+\overline
X(D_6(a_2),\frac 12)\\\notag
&\overline X(D_5(a_1)+A_1,1)|_W=4_1.
\end{align}

\noindent$\mathbf{A_3+A_2+A_1}:$ The infinitesimal character is
$(\frac 12,\frac 12,-\frac 32,\frac 12)+\nu(2,1,1,0),$ the lowest
$W$-type is $4_4$, and $\fz=A_1.$ In terms of Langlands classification,
the standard module is induced from the Steinberg  representation on
$A_2+\wti A_1.$ The $W$-structure is:
\begin{align}\notag
X(A_3+A_2+A_1)|_W&=\text{Ind}_{W(A_2+\wti A_1)}^{W(F_4)}((111)\otimes
(11))\\\notag
 &=4_4+16_1+8_4+2\cdot 8_2+4_5+12_1+6_1+9_3+2\cdot 9_4+2_2+1_4.\notag
\end{align} 
We compute the intertwining operators for general parameter $c$. In
that case the infinitesimal character is $(\frac 12,\frac
{c-1}2,-\frac {c+1}2,\frac 12)+\nu(2,1,1,0).$
\begin{small}
\begin{align}\notag
&4_4: 1;\
9_3: \frac {\frac c2-\nu}{\frac c2+\nu};\
6_1: \frac {\frac {3c-2}2-\nu}{\frac {3c-2}2+\nu};\
16_1: \frac {(\frac c2-\nu)(\frac {3c-2}2-\nu)}{(\frac c2+\nu)(\frac {3c-2}2+\nu)};\\\notag
&12_1: \frac {(\frac c2-\nu)^2(\frac {3c-2}2-\nu)}{(\frac c2+\nu)^2(\frac {3c-2}2+\nu)};\
8_4: \frac {(\frac c2-\nu)^2(\frac {3c-2}2-\nu)(\frac
    {3c}2-\nu)}{(\frac c2+\nu)^2(\frac {3c-2}2+\nu)(\frac {3c}2+\nu)};\
2_2: \frac{(\frac c2-\nu)(\frac {c+1}2-\nu)(\frac {c+2}2-\nu)}{(\frac
    c2+\nu)(\frac {c+1}2+\nu)(\frac {c+2}2+\nu)};\\\notag
&4_5: \frac{(\frac c2-\nu)^2(\frac {c+1}2-\nu)(\frac
    {c+2}2-\nu)(\frac{3c-2}2-\nu)(\frac {3c}2-\nu)}{(\frac
    c2+\nu)^2(\frac {c+1}2+\nu)(\frac {c+2}2+\nu)(\frac{3c-2}2+\nu)(\frac {3c}2+\nu)};\\\notag
&1_4: \frac{(\frac c2-\nu)^2(\frac {c+1}2-\nu)(\frac
    {c+2}2-\nu)(\frac{3c-2}2-\nu)(\frac {3c}2-\nu)(\frac {3c+2}2-\nu)}{(\frac
    c2+\nu)^2(\frac {c+1}2+\nu)(\frac {c+2}2+\nu)(\frac{3c-2}2+\nu)(\frac {3c}2+\nu)(\frac {3c+2}2+\nu)};\\\notag
&8_2: Det=\frac{(\frac c2-\nu)^3(\frac {c+1}2-\nu)(\frac
    {c+2}2-\nu)(\frac{3c-2}2-\nu)}{(\frac
    c2+\nu)^3(\frac {c+1}2+\nu)(\frac {c+2}2+\nu)(\frac{3c-2}2+\nu)};\\
&9_4: Det=\frac{(\frac c2-\nu)^3(\frac {c+1}2-\nu)(\frac
    {c+2}2-\nu)(\frac{3c-2}2-\nu)^2(\frac {3c}2-\nu)}{(\frac
    c2+\nu)^3(\frac {c+1}2+\nu)(\frac
    {c+2}2+\nu)(\frac{3c-2}2+\nu)^2(\frac {3c}2+\nu)}.
\end{align}
\end{small}
In the case $c=2$, it follows that $\overline X(A_3+A_2+A_1,\nu)$ is
unitary if and only if $0\le\nu\le 1.$ At $\nu=1$, the standard module
decomposes as 
\begin{align}\notag
&X(A_3A_2A_1,1)=\overline X(A_3A_2A_1,1)+\overline X(D_4A_1,(\frac
32,\frac 12))+X(E_7(a_5),\phi_1);\\\notag
&\overline X(A_3A_2A_1,1)|_W=4_4+6_1.\notag
\end{align}
Note that $\overline X(A_3A_2A_1,1)$ is the $IM$-dual of $\overline
X(A_5A_1,\frac 12).$

\end{proof}

\subsection{Matching of intertwining operators}\label{sec:5.1b} Let $(M,\sigma,\nu)$ be a (Hermitian)
Langlands parameter, where $\sigma$ is a discrete series module for
$\bH_M$ (or rather $\bH_{M_0}$, notation as in section
\ref{sec:1.4}). 
We construct a Hecke algebra $\bH_\sigma$ with
possibly unequal parameters, with a Weyl group $W_\sigma.$ Let
$\{e,h,f\}$ be a Lie triple parameterizing the tempered module
$\sigma.$ Then in all cases which we need to consider except for the
case $\{e,h,f\}\subset D_4(a_1)+A_1$, the root system for $\bH_\sigma$
and $W_\sigma$ is the root system of the reductive Lie algebra
$\fz(e,h,f)$ of the centralizer in $\fg$ of $\{e,h,f\}.$ In the case
$D_4(a_1)+A_1$, the centralizer is of type $2A_1$, but the root system
we consider is of type $B_2.$ The reason is the presence of a
component group $\bZ/2\bZ$ in that case.

Firstly, let us collect the relevant results from the operators
calculations in the maximal parabolic case, in the previous
section. For the case of $\bH(F_4,(1,2))$, each generalized discrete
series induced from maximal parabolics has a unique lowest $W$-type, $\mu_0.$

\begin{lemma}\label{l:match} For every maximal parabolic case
  $(M,\sigma)$, with lowest $W$-type $\mu_0$, there is a $W$-type
  $\mu_1$ such that
\begin{enumerate}
\item[(i)]  the operator $r_{\mu_1}(M,\sigma,\nu)$ is
  identical to the operator $r_{sgn}(\nu)$ for a Hecke algebra of type
  $A_1$, and
\item[(ii)] the Langlands quotient $\overline X(M,\sigma,\nu)$ is
  unitary if and only if $r_{\mu_1}(M,\sigma,\nu)$ is positive
  semidefinite.
\end{enumerate}
\end{lemma}

\begin{proof} The explicit operators and cases are

\begin{table}\label{t:max}
\begin{tabular}{|c|c|c|c|c|}
\hline
{Levi $M$} &$\CO$ &$\mu_0$ &$\mu_1$ &$r_{\mu_1}(M,\sigma,\nu)$\\
\hline
$C_3$ &$D_6$ &$9_4$ &$8_4$ &$\frac{1/2-\nu}{1/2+\nu}$\\
      &$D_6(a_1)$ &$9_2$ &$4_3$ &$\frac{1/2-\nu}{1/2+\nu}$\\
      &$D_6(a_2)$ &$16_1$ &$12_1$ or $6_2$
&$\frac{1/2-\nu}{1/2+\nu}$\\
\hline
$B_3$ &$D_5+A_1$ &$8_2$ &$2_2$ &$\frac{1/2-\nu}{1/2+\nu}$\\
      &$D_5(a_1)+A_1$ &$4_1$ &$16_1$ &$\frac{1-\nu}{1+\nu}$\\
\hline
$\wti A_2+A_1$ &$A_5+A_1$ &$6_1$ &$16_1$ &$\frac{1/2-\nu}{1/2+\nu}$\\
\hline
$A_2+\wti A_1$ &$A_3+A_2+A_1$ &$4_4$ &$9_3$ &$\frac{1-\nu}{1+\nu}$\\
\hline
\end{tabular}
\end{table}

\end{proof} 

\begin{definition}\label{d:petite} Let $\bH$ be an arbitrary graded
  Hecke algebra.
For every $(M,\sigma)$ maximal parabolic case, we call $\sigma$-petite
types, the following $W$-types which appear in $X(M,\sigma,\nu)$:
the lowest $W$-types, and any $\mu_1$ which satisfy the properties from lemma \ref{l:match}.
Now assume $(M,\sigma)$ is not maximal parabolic. A $W$-type $\mu$ is
called $\sigma$-petite, if for every Levi $M'$ such that $M\subset M'$
is maximal, the restriction of $\mu$ to $W(M')$ contains only
$\sigma$-petite $W(M')$-types. 
\end{definition}

The main idea
is summarized in the next statement, which is proposition 5.6 in
\cite{BC1} (see also section 2.10 in \cite{Ci2}). Recall from
(\ref{corr}) that to every $W$-type $\mu$ which appears in
$X(M,\sigma,\nu)$ we attach a $W_\sigma$-type $\rho(\mu).$

\begin{proposition}\label{match}
With the notation above, the set of $\sigma$-petite $W$-types
$\{\mu_0,\mu_1,\dots,\mu_l\}$ in $X(M,\sigma,\nu)$ with  the
corresponding $W_\sigma$-types 
$\{\rho(\mu_0),\rho(\mu_1),\dots,\rho(\mu_l)\}$, has the property that
$$r_\mu(M,\sigma,\nu)=r_{\rho(\mu)}(\nu),\text{ for all }\mu\in
  \{\mu_0,\dots,\mu_l\},$$
where $r_\mu(M,\sigma,\nu)$ is the operator in $\bH$ (defined as in
  (\ref{2.5.5})), and $r_{\rho(\mu)}$ is the spherical operator in
  $\bH_\sigma.$  
\end{proposition}

The explicit calculations and details will occupy the rest of this
section. The matching of intertwining operators from proposition
\ref{match} is used to rule out non-unitary modules. Ideally, one
would like that the remaining modules are all unitary. This is not the
case in general, and more (ad-hoc) arguments are needed to prove
non-unitarity. 

To prove unitarity, we use deformations of unitary irreducible modules
(\ie, complementary series methods), calculations of composition
factors and the Iwahori-Matsumoto involution.

\bigskip

\noindent$\mathbf{(A_5)''}:$ The infinitesimal character is
$(\nu_2+\frac {3\nu_1}2,2+\frac {\nu_2}2,\frac
	 {\nu_1}2,-2+\frac{\nu_1}2),$ the lowest $W$-type is $8_1$, and
	 $\fz=G_2.$ In the Langlands classification, the standard
	 module is induced from the Steinberg representation on $\wti
	 A_2$. The $W$-structure is $$X((A_5)'')|_W=\text{Ind}_{W(\wti A_2)}^{W(F_4)}((111)).$$
The restrictions of the nearby $W$-types are as follows:

\noindent\begin{tabular}{llllll}
Nilpotent &$(A_5)''$ &$A_5A_1$ &$D_6(a_2)$ &$E_7(a_5)$ &$E_7(a_5)$\\
$W$-type    &$8_1$     &$6_1$    &$16_1$     &$12_1$     &$6_2$\\
Multiplicity &$1$    &$1$      &$2$        &$2$        &$1$\\
$\wti A_2\subset \wti A_2+A_1$ &$(2)$  &$(11)$  &$(2),(11)$
&$(2),(11)$ &$(2)$\\
$\wti A_2\subset C_3$ &$11\times 1$ &$11\times 1$ &$11\times 1,1\times 11$
&$11\times 1,1\times 11$ &$1\times 11$\\
$W(\fz)=W(G_2)$ &$(1,0)$  &$(1,3)''$ &$(2,2)$ &$(2,1)$ &$(1,3)'$\\

\end{tabular}

We need the intertwining operators for the induced module from the
Steinberg representation on $A_2$ in the Hecke algebra of type $C_3$,
with parameters 2--2$\Leftarrow$1. The operators are $$11\times 1: 1,\quad 1\times 11:
\frac {\frac 12-\nu}{\frac 12+\nu},\quad 111\times 0: \frac
      {\frac 32-\nu}{\frac 32+\nu}.$$
Therefore, the matching of operators is with the Hecke algebra of type
$G_2$ with equal parameters, $\bH(G_2).$ The nearby $W$-types match all
the relevant $W(G_2)$-types, so the unitary set $\CU(A_5'')$ is
included in the spherical unitary dual of $\bH(G_2).$ The modules of
$\bH(G_2)$ are parameterized by nilpotent orbits in the Lie algebra of
type $G_2$. We analyze the composition series and unitarity of
$X(A_5'')$ by cases corresponding to these nilpotent orbits.

$(1):$ In $\bH(G_2)$, this is the spherical complementary series 
\begin{equation}\label{eq:CS(A5'')}
\{3\nu_1+2\nu_2<1\}\cup\{3\nu_1+\nu_2>1>2\nu_1+\nu_2\}.
\end{equation}
The standard module $X(A_5'')$ is irreducible in these regions, and
therefore unitary (being unitarily induced and unitary when
$\nu_2=0$).

$A_1:$ The parameters are of the form $(\nu_1,\nu_2)=(-\frac
12+\nu,1).$ They are unitary for $0\le\nu<\frac 12$, being endpoints
of the complementary series (\ref{eq:CS(A5'')}). The decomposition of
the standard module is
\begin{align}\notag
X(A_5'',(1,-\frac 12+\nu))&=\overline X(A_5'',(1,-\frac
12+\nu))+X(A_5+A_1,\nu)\\\notag
\overline X(A_5'',(1,-\frac 12+\nu))|_W&=8_1+16_1+12_1+6_2+2\cdot
8_4+4_3+4_5+2\cdot 9_2\\\notag
&+9_4+2_4+1_2.\notag
\end{align}

$\wti A_1:$ The parameters are of the form $(\nu_1,\nu_2)=(1, -\frac
32+\nu,1).$ They are unitary for $0\le\nu<\frac 12$, being endpoints
of the complementary series (\ref{eq:CS(A5'')}). The decomposition of
the standard module is
\begin{align}\notag
X(A_5'',(-\frac 32+\nu,1))&=\overline X(A_5'',(-\frac
32+\nu,1))+X(D_6(a_2),\nu)\\\notag
\overline X(A_5'',(-\frac 32+\nu,1))|_W&=8_1+6_1+16_1+12_1+2\cdot
8_4+2\cdot 4_3+2\cdot 9_2\\\notag
&+9_4+2_4+1_2.\notag
\end{align}

$G_2(a_1):$ This is the parameter $(\nu_1,\nu_2)=(0,1).$ It is
unitary, being an endpoint of the complementary series
(\ref{eq:CS(A5'')}). The decomposition of the standard module is 
\begin{align}\notag
X(A_5'',(0,1))&=\overline X(A_5'',(0,1))+\overline X(A_5A_1,\frac
12)+\overline X(D_6(a_2),\frac 12)\\\notag
&+X(E_7(a_5),\phi_1)+X(E_7(a_5),\phi_2)\\\notag\overline X(A_5'',(0,1))|_W&=8_1+12_1+8_4+9_2+1_2.\notag
\end{align}

$G_2:$ This is the parameter $(\nu_1,\nu_2)=(1,1).$ In $\bH(G_2)$, it
is isolated. We compute the $W$-structure of the standard module. In
addition to the operators given by the nearby $W$-types, we need to
check the operator on $9_2.$
\begin{align}\notag
X(A_5'',(1,1))&=\overline X(A_5'',(1,1))+\overline X(A_5A_1,\frac
32)+\overline X(D_6(a_2),\frac 52)\\\notag&+\overline X(D_5A_1,\frac
12)+\overline X(D_6(a_1),\frac 12)+X(E_7(a_4),\phi_2)\\\notag
\overline X(A_5'',(1,1))|_W&=8_1+9_2.
\end{align}
Moreover, the operator on $9_2$ is positive at $(1,1)$, so this point
is unitary.

\medskip

In conclusion, the unitary parameter set $\CU(A_5'')$ is identical
with the spherical unitary dual $\CS\CU(\bH(G_2)).$

\bigskip

\noindent$\mathbf{D_4+A_1}:$ The infinitesimal character is
$(\nu_1,\nu_2,\frac 52,\frac 12),$ the lowest $W$-type is $9_3$, and
$\fz=B_2.$ In Langlands classification, the standard module is induced
from the Steinberg representation on $C_2$. The $W$-structure is
$$X(D_4+A_1)|_W=\text{Ind}_{W(C_2)}^{W(F_4)}(0\times 11).$$
The restrictions of the nearby $W$-types are as follows:

\noindent\begin{tabular}{llllll}
Nilpotent &$D_4+A_1$ &$D_5(a_1)A_1$ &$D_6(a_2)$ &$E_7(a_5)$ &$D_6(a_1)$\\
$W$-type    &$9_3$     &$4_1$         &$16_1$     &$12_1,6_2$     &$9_2$\\
Multiplicity &$1$    &$1$           &$2$        &$1,1$        &$1$\\
$B_2\subset B_3$  &$1\times 11$ &$0\times 12$ &$0\times 12,1\times 11$
&$1\times 11$ &$0\times 12$\\
$C_2\subset C_3$ &$0\times 12$ &$0\times 12$ &$0\times 12,1\times 11$
&$1\times 11$ &$1\times 11$\\
$W(\fz)=W(B_2)$ &$2\times 0$ &$11\times 0$ &$1\times 1$ &$0\times 2$
&$0\times 11$

\end{tabular} 

We need the intertwining operators for the induced modules from the
Steinberg representation of $C_2=B_2$ in the Hecke algebras of types
$B_3$, and $C_3$, with parameters 1--1$\Rightarrow$2, respectively
2--2$\Leftarrow$1. 

In $B_3$ (1--1$\Rightarrow$2), the infinitesimal character is
$(\nu,3,2)$, and the standard module is the induced from the Steinberg
representation on $B_2.$ The operators are $$1\times 11: 1,\quad
0\times 12: \frac {1-\nu}{1+\nu},\quad 0\times 111: \frac
{4-\nu}{4+\nu}.$$
In $C_3$ (2--2$\Leftarrow$1), the infinitesimal character is
$(\nu,\frac 52,\frac 12)$, and the standard module is the induced from the Steinberg
representation on $C_2.$ The operators are $$0\times 12: 1,\quad
1\times 11: \frac {\frac 32-\nu}{\frac 32+\nu},\quad 0\times 111: \frac
{(\frac 32-\nu)(\frac 92-\nu)}{(\frac 32+\nu)(\frac 92+\nu)}.$$

Therefore, the matching of operators is with the Hecke algebra of type
$B_2$ with parameters 1$\Rightarrow$3/2. The nearby $W$-types match all
the $W(B_2)$-representations of this $B_2,$ therefore, the unitary
parameter set $\CU(D_4+A_1)$ is included in the spherical unitary dual
of $\bH(B_2,1,3/2).$ The modules of $\bH(B_2,1,3/2)$ are parameterized
by a cuspidal local system on $\fk {sp}(2)\oplus \bC^2\subset \fk
{sp}(6).$ Specifically the nilpotent orbits of $\fk{sp}(6)$ which
enter in the parameterization are $(6),(42),(411),(222),(21^4).$ We
analyze the composition series and unitarity of $X(D_4+A_1)$ by cases
corresponding to these nilpotent orbits.

$(21^4):$ In $H(B_2,1,3/2)$, this is the spherical complementary
series 
\begin{equation}\label{eq:CS(D4A1)}
\{\nu_1+\nu_2<1\}\cup\{\nu_1-\nu_2>1,\nu_1<\frac 32\}. 
\end{equation}
The
standard module $X(D_4+A_1)$ is irreducible in these regions, and
therefore unitary (being unitarily induced and unitary when
$\nu_2=0$). 

$(222):$ The parameters are of the form $(\nu_1,\nu_2)=(\frac
12+\nu,-\frac 12+\nu).$ They are unitary for $0\le\nu<\frac 32$, being
endpoints of the complementary series (\ref{eq:CS(D4A1)}). The standard module
decomposes as follows 
\begin{align}\notag
X(D_4A_1,(\frac 12+\nu,-\frac 12+\nu))&=\overline X(D_4A_1,(\frac
12+\nu,-\frac 12+\nu)+X(D_5(a_1)A_1,\nu)\\\notag
\overline X(D_4A_1,(\frac 12+\nu,-\frac
12+\nu)|_W&=9_3+16_1+12_1+2\cdot 9_4+8_4+2\cdot 8_2+6_2\\\notag
&+2\cdot 4_5+2_2+1_4.\notag
\end{align}
$(411):$ The parameters are $(\nu_1,\nu_2)=(\nu,\frac 32).$ They are
unitary for $0\le\nu\le\frac 12$, being endpoints of the complementary
series (\ref{eq:CS(D4A1)}). The standard module decomposes as follows
\begin{align}\notag
X(D_4A_1,(\nu,\frac 32))&=\overline X(D_4A_1,(\nu,\frac
32))+X(D_6(a_2),\nu)\\\notag
\overline X(D_4A_1,(\nu,\frac 32))|_W&=9_3+16_1+8_2+4_1+2_2+9_4.
\end{align}
$(42):$ The parameters are $(\nu_1,\nu_2)=(\frac 32,\frac 12).$ This
is an endpoint of the complementary series (\ref{eq:CS(D4A1)}),
therefore unitary. The decomposition of the standard module is 
\begin{align}\notag
X(D_4A_1,(\frac 32,\frac 12))&=\overline X(D_4A_1,(\frac 32,\frac
12))+\overline X(D_5(a_1)A_1,1)+\overline X(D_6(a_2),\frac
12)\\\notag
&+X(E_7(a_5),\phi_1)+X(E_7(a_5),\phi_2)\\\notag
\overline X(D_4A_1,(\frac 32,\frac 12))|_W&=9_3+16_1+8_2+9_4+2_2.
\end{align}
$(6):$ The parameters are $(\nu_1,\nu_2)=(\frac 52,\frac 32).$ This
point is isolated. From the $W$-structure of $\overline
X(D_4A_1,(\nu,\frac 32))$, we see that the only $W$-types which can be
present in $\overline X(D_4A_1,(\frac 52,\frac 32))$ are
$9_3,16_1,8_2,4_1,2_2,9_4.$ By virtue of the matching with operators
in $\bH(B_2,1,3/2)$, we know $16_1$ and $4_1$ cannot be present. We
compute the operators on the remaining ones and find that the
$W$-structure is $$\overline X(D_4A_1,(\frac 52,\frac 32))=9_3+8_2.$$
Moreover, the operator on $8_2$ is positive, so $\overline
X(D_4A_1,(\frac 52,\frac 32))$ is unitary. Note also that $\overline
X(D_4A_1,(\frac 52,\frac 32))$ is the $IM$-dual of $\overline X(A_5'',(1,1)).$

\medskip

In conclusion, the unitary set $\CU(D_4+A_1)$ is identical with $\CS\CU(\bH(B_2,1,3/2).$

\bigskip

\bigskip

\noindent$\mathbf{D_4(a_1)+A_1:}$ The infinitesimal character is
$(\nu_1,\nu_2,\frac 32,\frac 12),$ the lowest $W$-types are $9_1$ and
$2_1$, and $\fz=2A_1.$ The two lowest $W$-types are separate when
$\nu_1=\nu_2.$ If $\nu_1\neq\nu_2,$ then the (Langlands
classification) standard module is
induced from a one dimensional discrete series on $B_2$
(1$\Rightarrow$2). 
As a W-module
$$X(D_4(a_1)A_1,(\nu_1,\nu_2))|_W=\text{Ind}_{W(B_2)}^{W(F_4)}(0\times
2),\quad\text{if }\nu_1\neq\nu_2.$$  
When $\nu_1=\nu_2=\nu$, there are two standard modules, corresponding
to the two lowest $W$-types, induced from
two tempered representations of $B_3$ (1--1$\Rightarrow$2). As
W-modules,
\begin{align}\notag
  X(D_4(a_1)A_1,(\nu,\nu),\phi_1)|_W&=\text{Ind}_{W(B_3)}^{W(F_4)}(1\times
  2+0\times 12)=9_1+8_1+4_3+6_1\\\notag&+2\cdot 9_2+2\cdot 16_1+2\cdot 8_4+9_4+12_1+4_1+2_4.\\\notag
X(D_4(a_1)A_1,(\nu,\nu),\phi_2)|_W&=\text{Ind}_{W(B_3)}^{W(F_4)}(0\times
  3)=2_1+8_1+4_3+9_2+1_2.
\end{align}
The restrictions of the nearby $W$-types are:

\noindent\begin{tabular}{llllll}
Nilpotent &$D_4(a_1)A_1$ &$D_4(a_1)A_1$ &$A_5''$ &$A_3A_2A_1$ &$D_5(a_1)A_1$\\
$W$-type &$9_1$ &$2_1$ &$8_1$ &$4_3$ &$4_1$\\
Multiplicity &$1+0$ &$0+1$ &$1+1$ &$1+1$ &$1+0$\\
$B_2\subset B_3$ &$1\times 2$ &$0\times 3$ &$1\times 2,0\times 3$
&$1\times 2,0\times 3$ &$0\times 12$\\
$C_2\subset C_3$ &$12\times 0$ &$12\times 0$ &$12\times 0,11\times 1$
&$11\times 1,111\times 0$ &$0\times 12$\\
$W(B_2)=$ &$2\times 0$ &$11\times 0$ &$1\times 1$ &$1\times 1$
&$11\times 0$\\
$W(\fz)\rtimes \bZ/2$ 
\end{tabular}

The matching of the intertwining operators is with the Hecke algebra
$\bH(B_2,0,5/2).$ More specifically,

\begin{tabular}{ll}
$9_1:$ &$1$;\\
$2_1:$ &$1$;\\
$8_1:$ &$\left(\begin{matrix} \frac {\frac 52-\nu_1}{\frac 52-\nu_2} &0\\
    0&\frac{\frac 52-\nu_2}{\frac 52+\nu_2}\end{matrix}\right).$
\end{tabular}

For the calculations, we need some intertwining operators on modules
in the Hecke algebras of type $B_3$ and $C_3$. 

In $\bH(B_3,1,2)$, the infinitesimal character is $(\nu,2,1)$. There
are two lowest $W$-types $1\times 2$ and $0\times 3,$ which are separate
at $\nu=0$ only. When $\nu>0,$ the operators are
$$1\times 2: 1,\quad 0\times 3: 1,\quad 0\times 12: \frac
{3-\nu}{3+\nu}.$$
In $\bH(C_3,2,1)$, the infinitesimal character is $(\nu,\frac 32,\frac
12).$ The lowest $W$-type is $12\times 0.$ The operators are
$$12\times 0: 1,\quad 11\times 1: \frac {\frac 52-\nu}{\frac
  52+\nu},\quad 111\times 0: \frac {\frac 72-\nu}{\frac 72+\nu}.$$

From this calculations, it follows that the reducibility lines for
$X(D_4(a_1),A_1,(\nu_1,\nu_2)),$ when $\nu_1\neq\nu_2,$ are
$\nu_1\pm\nu_2=3,$ $\nu_i=\frac 52,\frac 72,$ $i=1,2.$ The operator on
$8_1$ shows that the unitary set $\CU(D_4(a_1)A_1)$ is included in
$0\le\nu_2\le\nu\le\frac 52.$ However, the line $\nu_1+\nu_2=3$ cuts
this region. The restrictions in the table above imply that the
operator on $4_1$ is $$4_1: \frac
{(3-(\nu_1+\nu_2))(3-(\nu_1-\nu_2))}{(3+(\nu_1+\nu_2))(3+(\nu_1-\nu_2))}.$$
So for $\nu_1\neq\nu_2,$ $X(D_4(a_1)A_1,(\nu_1,\nu_2))$ is unitary if
and only if $$\{\nu_1+\nu_2<3,\ \nu_1<\frac 52\}.$$
We also note that $\overline X(D_4(a_1)A_1,(\nu,\frac 52))$, $0\le
\nu<\frac 12$ is the $IM$-dual of $\overline X(D_4A_1,(\nu,\frac 32))$,
and $\overline X(D_4(a_1)A_1,(\frac 52,\frac 12))$ is the $IM$-dual of
$\overline X(D_4A_1,(\frac 32,\frac 12)).$ On $\nu_1+\nu_2=3,$ we
write the parameters as $(\frac 32+\nu,-\frac 32+\nu)$ (which are
unitary for $0\le\nu\le 1$). The $W$-structure here is
\begin{align}\notag
X(D_4(a_1)A_1,(\frac 32+\nu,-\frac 32+\nu))&=\overline
X(D_4(a_1)A_1,(\frac 32+\nu,-\frac 32+\nu))+X(D_5(a_1)A_1,\nu)\\\notag
\overline
X(D_4(a_1)A_1,(\frac 32+\nu,-\frac 32+\nu))|_W&=9_1+2_1+2\cdot
9_4+16_1+8_4+2\cdot 8_1\\\notag&+2\cdot 4_3+12_1+6_1+1_2.
\end{align}

It remains to analyze the case $\nu_1=\nu_2=\nu.$ The calculation on
$8_1$ can be used to conclude that both Langlands quotients
$\overline X(D_4(a_1)A_1,(\nu,\nu))$ fail to be unitary for $\nu>\frac 52.$
Moreover, the calculation for $4_1$ implies that $\overline
X(D_4(a_1)A_1,(\nu,\nu),\phi_1)$ is not unitary for $\nu>\frac
32.$ 

For $\phi_1$, the first possible reducibility point is $\nu=\frac 32$
(this is seen by looking at the $W$-types in $
X(D_4(a_1)A_1,(\nu,\nu),\phi_1)$ and their corresponding nilpotent
orbits). So it is unitary for $0\le\nu\le\frac 32.$ The decomposition at $\nu=\frac 32$ is 
\begin{align}\notag
X(D_4(a_1)A_1,(\frac 32,\frac 32),\phi_1)&=\overline
X(D_4(a_1)A_1,(\frac 32,\frac 32),\phi_1)+X(D_5(a_1)A_1,0)\\\notag
X(D_4(a_1)A_1,(\frac 32,\frac
32),\phi_1)|_W&=9_1+9_2+16_1+8_4+8_1+12_1+6_1+4_3.
\end{align}

For $\phi_2$, the first possible reducibility point is $\nu=\frac 52$,
so this is unitary if and only if $0\le\nu\le\frac 52.$ Note that
$\overline X(D_4(a_1)A_1,(\frac 52,\frac 52),\phi_2)$ is the $IM$-dual
of $X(E_7(a_4),\phi_1)$. As a W-representation, it is just $2_1.$

\medskip

In conclusion, the unitary set $\CU(D_4(a_1)A_1)$ is given by:
\begin{enumerate} 
\item $\{\nu_1+\nu_2\le 3,\ \nu_1\le \frac 52\}$, if $\nu_1\neq
  \nu_2.$
\item $0\le\nu\le\frac 32$, if $\nu_1=\nu_2=\nu$ for the
  $\phi_1$-factor.
\item $0\le\nu\le\frac 52$, if $\nu_1=\nu_2=\nu$ for the
  $\phi_2$-factor.
\end{enumerate}

\bigskip

\noindent$\mathbf{A_3+2A_1:}$ The infinitesimal character is
$(\nu_1,1+\nu_2,-1+\nu_2,\frac 12),$ the lowest $W$-type $8_3$, and
$\fz=2A_1.$ In terms of Langlands classification, the standard module
is induced from the Steinberg representation on the Hecke algebra of
type $A_1+\wti A_1.$ As a W-module,
$$X(A_3+2A_1)|_W=\text{Ind}_{W(A_1+\wti A_1)}^{W(F_4)}((11)\otimes
(11)).$$

The restrictions of nearby $W$-types are:

\noindent\begin{tabular}{lllllll}
Nilpotent &$A_3+2A_1$ &$D_4(a_1)A_1$ &$A_3A_2A_1$ &$A_5''$ &$D_4A_1$ &$D_5(a_1)A_1$\\
$W$-type &$8_3$ &$9_1$ &$4_4$ &$8_1$ &$9_3$ &$4_1$\\
Multiplicity &$1$ &$1$ &$1$ &$1$ &$2$ &$1$\\
$C_1\times A_1\subset C_3$ &$1\times 2$ &$1\times 2$ &$1\times 2$
&$11\times 1$ &$1\times 2,0\times 12$ &$0\times 12$\\
$B_1\times A_1\subset B_3$ &$11\times 1$ &$1\times 2$ &$11\times 1$
&$1\times 2$ &$1\times 2,2\times 1$ &$0\times 12$\\
$A_1\subset A_2$ &$(21)$ &$(21)$ &$(1^3)$ &$(21)$ &$(21),(1^3)$ &$(21)$\\
$\wti A_1\subset \wti A_2$ &$(21)$ &$(21)$ &$(1^3)$ &$(21)$ &$(21),(21)$ &$(21)$

\end{tabular}

We need the following calculations from the Hecke algebras of type
$B_3$ and $C_3$.

In $\bH(C_3,2,1),$ the infinitesimal character is $(1+\nu,-1+\nu,\frac
12).$ The standard module is induced from the Steinberg module on
$C_1\times A_1$, and it is reducible at $\nu=\frac 32,\frac 52,\frac
72.$ The relevant operators are:
$$1\times 2: 1,\quad 11\times 1: \frac {\frac 52-\nu}{\frac
  52+\nu},\quad 0\times 12: \frac {\frac 32-\nu}{\frac 32+\nu}.$$

In $\bH(B_3,1,2),$ the infinitesimal character is $(\frac
12+\nu,-\frac 12+\nu,2).$ The standard module is induced from the
Steinberg module on $B_1\times A_1,$ and it is reducible at $\nu=\frac
12,\frac 52,\frac 72.$ The relevant operators are:
$$11\times 1: 1,\quad 1\times 2: \frac {\frac 12-\nu}{\frac
  12+\nu},\quad 1\times 11:\frac {\frac 52-\nu}{\frac 52+\nu}.$$  

The reducibility lines for $X(A_3+2A_1,(\nu_1,\nu_2))$ are
$\nu_1=\frac 12,\frac 52,\frac 72$, $\nu_2=\frac 32,\frac 52,\frac
72,$ and $\pm\nu_1\pm 2\nu_2=\frac 32$ and $\pm\nu_1\pm\nu_2=3.$ (The
last two types come from the restrictions of the operators to $A_1\subset
A_2$ and $\wti A_1\subset \wti A_2.$)

From the restrictions we that the operators are as follows:

\begin{tabular}{ll}
$8_3:$ &$1$;\\
$9_1:$ &$\frac {\frac 12-\nu_1}{\frac 12+\nu_1}$;\\
$4_1:$ &$\frac {(\frac 12-\nu_1)(\frac 52-\nu_1)(\frac
    32-\nu_2)}{(\frac 12+\nu_1)(\frac 52+\nu_1)(\frac 32+\nu_2)}$;\\
$8_1:$ &$\frac {(\frac 12-\nu_1)(\frac 52-\nu_2)}{(\frac
    12+\nu_1)(\frac 52+\nu_2)}$;\\
$4_4:$ &$\frac {(\frac 32-(\nu_1+2\nu_2))(\frac 32-(\nu_1-2\nu_2)}{(\frac 32+(\nu_1+2\nu_2))(\frac 32+(\nu_1-2\nu_2)}$;\\
$9_3:$ &Det=$\frac {(\frac 32-\nu_2)(\frac 52-\nu_1)(\frac 32-(\nu_1+2\nu_2))(\frac 32-(\nu_1-2\nu_2)}{(\frac 32+\nu_2)(\frac 52+ \nu_1)(\frac 32+(\nu_1+2\nu_2))(\frac 32+(\nu_1-2\nu_2)}.$
\end{tabular}

In this case, it is natural to try to match the unitary set
$\CU(A_3+2A_1)$ with the spherical unitary dual of $A_1+A_1$ with
parameters $1$ and $3$. This is just the set $\{0\le\nu_1\le\frac 12,\
0\le\nu_2\le\frac 32\}.$ The operators above imply that actually the
unitary set is $$\CU(A_3+2A_1)=\{0\le\nu_1\le\frac 12,\nu_1+2\nu_2\le
\frac 32\}\cup\{0\le\nu_1\le\frac 12,0\le\nu_2\le\frac
32,2\nu_2-\nu_1\ge\frac 32\}.$$
We also record the relevant decompositions and $W$-structure of the standard
module. 
On $\nu_2=\frac 32$, $0\le\nu_1<\frac 12$, the decomposition is 
\begin{align}\notag
X(A_3+2A_1,(\nu,\frac 32))&=\overline X(A_3+2A_1,(\nu,\frac
32))+X(D_4+A_1,(\nu,\frac 12).
\end{align}
Moreover, $\overline X(A_3+2A_1,(\nu,\frac 32))$ is self $IM$-dual.

On $\nu_1=\frac 12,$ $0\le\nu_2<\frac 12$ and $1<\nu_2<\frac 32$, the
decomposition is 
\begin{align}\notag
X(A_3+2A_1,(\frac 12,\nu))=\overline X(A_3+2A_1,(\frac 12,\nu))+X(D_4(a_1)+A_1,(\nu,\nu),\phi_1).
\end{align}
The $IM$-dual of $\overline X(A_3+2A_1,(\frac 12,\nu))$ is a
spherical module.

On $2\nu_2\pm\nu_1=\frac 32,$ we write the parameter as
$(\nu_1,\nu_2)=(-\frac 12+2\nu,\frac 12+\nu).$ Then it is unitary for
$0<\nu<\frac 12.$ Here, the decomposition is
\begin{align}
X(A_3+2A_1,(-\frac 12+2\nu,\frac 12+\nu))=\overline X(A_3+2A_1,(-\frac 12+2\nu,\frac 12+\nu))+X(A_3+A_2+A_1,\nu).
\end{align}
The $IM$-dual of $\overline X(A_3+2A_1,(-\frac 12+2\nu,\frac
12+\nu))$ is parameterized by $4A_1.$

At $(\frac 12,\frac 12)$, the decomposition is
\begin{align}\notag
X(A_3+2A_1,(\frac 12,\frac 12))&=\overline X(A_3+2A_1,(\frac 12,\frac
12))+X(D_4(a_1)A_1,(\frac 12,\frac 12),\phi_1)+X(A_3A_2A_1,0)\\\notag
\overline X(A_3+2A_1,(\frac 12,\frac
12))|_W&=8_3+16_1+8_2+12_1+6_2+9_3+9_4+4_5.
\end{align}
The similar decomposition (and identical $W$-structure) holds at
$(1,1).$

Finally, at the point $(\frac 12,\frac 32),$ the decomposition is
\begin{align}\notag
X(A_3+2A_1,(\frac 12,\frac 32))&=\overline X(A_3+2A_1,(\frac 12,\frac
32))+\overline X(D_4(a_1)A_1,(\frac 32,\frac 32),\phi_1)\\\notag
&+\overline
X(D_4A_1,(\frac 12,\frac 12))+X(D_5(a_1)A_1,0)\\\notag
\overline X(A_3+2A_1,(\frac 12,\frac
32))|_W&=8_3+16_1+8_2+4_4+12_1+6_1+9_3+9_4.
\end{align}
Note that $\overline X(A_3+2A_1,(\frac 12,\frac 32))$ is the $IM$-dual
of $\overline X(D_4(a_1)A_1,(\frac 32,\frac 32),\phi_1).$

\bigskip

\noindent$\mathbf{(A_3+A_1)'':}$ The infinitesimal character is
$(\frac {\nu_1+\nu_2}2,\frac {\nu_1-\nu_2}2,1+\frac
	 {\nu_2}2,-1+\frac{\nu_2}2),$ the lowest $W$-type is $4_2,$ and
	 centralizer $B_3.$ In the Langlands classification, the
	 standard module is induced from the Steinberg representation
	 on $\wti A_1,$ so as a $W$-module, it is
	 $X((A_3+A_1)'')|_W=\text{Ind}_{W(\wti A_1)}^{W(F_4)}((11)).$

The restrictions of nearby $W$-types are:

\begin{tabular}{llllll}
Nilpotent &$(A_3A_1)''$ &$A_32A_1$ &$D4(a_1)A_1$ &$D_4(a_1)A_1$ &$A_5''$\\
$W$-types &$4_2$ &$8_3$ &$9_1$ &$2_1$ &$8_1$\\
$\wti A_1\subset \wti A_2$ &$(21)$ &$2\cdot (21)$ &$3\cdot (21)$
&$(21)$ &$3\cdot (21),(1^3)$\\
$A_1$ &$(2)$ &$(2),(11)$ &$2\cdot (2),(11)$ &$(2)$ &$3\cdot(2),(11)$\\
$A_1\subset C_2$ &$1\times 1$ &$2(1\times 1)$ &$2(1\times 1),11\times
0$ &$11\times 0$ &$2(1\times 1), 2(11\times 0)$\\
$W(\fz)$ &$3\times 0$ &$12\times 0$ &$2\times 1$ &$0\times 3$
&$3\times 0+1\times 2$
\end{tabular}

We need a calculation from the Hecke algebra of type $C_2$,
2$\Leftarrow$1. The infinitesimal character is $(1+\nu,-1+\nu)$, and
the standard module is induced from the Steinberg representation on
$\wti A_1$. Then the reducibility points are $\nu=\frac 12,\frac 32,$
and the operators are $$1\times 1: 1,\quad 11\times 0: \frac {\frac
  12-\nu}{\frac 12+\nu}.$$

It follows that the reducibility hyperplanes for
$X((A_3+A_1)'',(\nu_1,\nu_2))$ are $\nu_i=1,3$ (from $A_1\subset
C_2$), $\nu_i\pm\nu_j=1$ (from $\wti A_1\subset \wti A_1+A_1$), and
$\nu_1\pm\nu_2\pm\nu_3=6$ (from $\wti A_1\subset \wti A_2$).

The matching of operators is with the spherical operators for
$\bH(B_3)$,  the
Hecke algebra of type $B_3$ with equal parameters. From the tables, we
see that we can match the operators on all relevant $W(B_3)$-types,
except $1\times 2.$ (The operator on $8_1$ fails to match it because
of the restriction to $\wti A_1\subset \wti A_2.$) The representation
$1\times 2$ in $B_3$ is the only one which rules out the interval
$\{(1+\nu,\nu,-1+\nu):1<\nu<2\}$ in the spherical dual. It follows
that the unitary set $\CU((A_3A_1)'')$ is included in the union
$\CS\CU(\bH(B_3))\cup\{(1+\nu,\nu,-1+\nu):1<\nu<2\}.$ We also note that the
extra hyperplanes of reducibility of $X((A_3+A_1)'')$ which do not
correspond to $B_3$, intersect this union only in the point $(3,2,1).$ 
We analyze the parameters in
$\CU((A_3A_1)'')$ partitioned by the nilpotent orbits in type $\fk{so}(7)$
(that is, in the same way we parameterize the spherical dual of $\bH(B_3)$).

$(1^7):$ The parameters are $(\nu_1,\nu_2,\nu_3)$; the spherical
complementary series for $\bH(B_3)$ is 
\begin{equation}\label{eq:CS(A_3A_1)}
\{\nu_1+\nu_2<1\}\cup\{\nu_1+\nu_3>1,\nu_2+\nu_3<1,\nu_1<1\}.
\end{equation}
In these regions, $X((A_3+A_1)'')$ is irreducible, and it is unitarily
induced and unitary for $\nu_3=0$, so it is unitary in
(\ref{eq:CS(A_3A_1)}).

$(221^3):$ The parameters are $(\frac 12+\nu_1,-\frac
12+\nu_1,2\nu_2)$, unitary for $\{0\le\nu_1<\frac 12,0\le\nu_2<\frac 12\},$
being endpoints of the complementary series (\ref{eq:CS(A_3A_1)}). The
decomposition of the standard module in this region is
\begin{align}\notag
X((A_3A_1)'',\frac 12+\nu_1,-\frac 12+\nu_1,2\nu_2))&=\overline X((A_3A_1)'',\frac 12+\nu_1,-\frac 12+\nu_1,2\nu_2))\\\notag&+X(A_3+2A_1,(\nu_1,\nu_2)).
\end{align}
The $IM$-dual of $\overline X((A_3A_1)'',\frac 12+\nu_1,-\frac
12+\nu_1,2\nu_2))$ is parameterized by $4A_1.$

$(31^4):$ The parameters are $(\nu_1+\nu_2,\nu_1-\nu_2,1)$, unitary for
$\{\nu_1<\frac 12\},$ being endpoints of  (\ref{eq:CS(A_3A_1)}). The
decomposition of the standard module is
\begin{align}\notag
X((A_3A_1)'',(\nu_1+\nu_2,\nu_1-\nu_2,1))&=\overline X((A_3A_1)'',(\nu_1+\nu_2,\nu_1-\nu_2,1))\\\notag&+X(D_4(a_1)A_1,(\nu_1,\nu_2)).
\end{align}
The $IM$-dual of $\overline X((A_3A_1)'',(\nu_1+\nu_2,\nu_1-\nu_2,1))$
is a spherical module.

$(322):$ The parameters are $(\frac 12+\nu,-\frac 12+\nu,1),$ unitary
for $0\le\nu<\frac 12$ (again endpoints of (\ref{eq:CS(A_3A_1)}). The
standard module decomposes as
\begin{align}\notag
X((A_3A_1)'',(\frac 12+\nu,-\frac 12+\nu,1))&=\overline
X((A_3A_1)'',(\frac 12+\nu,-\frac 12+\nu,1))\\\notag
&+X(A_3+2A_1,(\nu,\frac 12)+X(D_4(a_1)A_1,(\nu,\frac 12))
\end{align}
Moreover, $\overline X((A_3A_1)'',(\frac 12+\nu,-\frac 12+\nu,1))$ is
self $IM$-dual.

$(331):$ The parameters are $(1+\nu,\nu,-1+\nu).$ At $\nu=0,$ we have
a unitary module, endpoint of (\ref{eq:CS(A_3A_1)}). (Note also that
$\overline X((A_3A_1)'',(1,1,0))$ is the $IM$-dual of $\overline
X(A_3+2A_1,(\frac 12,\frac 12)).$) In $\bH(B_3)$,
this is the only such parameter which is unitary. In $F_4$ however, by
the observations preceding this analysis, we need to also consider the
segment $1<\nu<2.$ In this interval, the decomposition is 
\begin{align}\notag
X((A_3+A_1)'',(1+\nu,&\nu,-1+\nu))=\overline
X((A_3+A_1)'',(1+\nu,\nu,-1+\nu))\\\notag
&+2\cdot \overline X(A_3+2A_1,(-\frac
12+\nu,\frac 12+\frac {\nu}2))+X(A_3A_2A_1,\frac \nu 2)\\\notag
\overline X((A_3+A_1)'',(1+\nu,&\nu,-1+\nu))|_W=4_2+9_1+2_1+6_2+2\cdot
8_1+4_3\\\notag
&+2\cdot 9_2+12_1+16_1+8_4+1_2.
\end{align}
This segment is isolated, so to prove that it is unitary, we compute
explicitly the signatures on all the $W$-types which appear in the
restriction of $\overline X((A_3+A_1)'',(1+\nu,\nu,-1+\nu))|_W.$ The
$IM$-dual of $\overline X((A_3+A_1)'',(1+\nu,\nu,-1+\nu))$ is a module
parameterized by $A_2+3A_1$, and our calculations will be confirmed by
the unitarity of those duals.

$(511):$ The parameters are $(\nu,2,1)$, which are not unitary for
$\nu>0.$ At $\nu=0,$ $\overline X((A_3A_1)'',(2,1,0))$ is the $IM$-dual
of $\overline X(A_3+2A_1,(\frac 12,1))$, and therefore it is unitary.

$(7):$ The parameter is $(3,2,1).$ Then $\overline
X((A_3+A_1)'',(3,2,1))$ is the $IM$-dual of $X(E_7(a_5),\phi_2)$, and
therefore it is unitary.

\medskip

In conclusion, $\CU((A_3+A_1)'')=\CS\CU(\bH(B_3))\sqcup \{(1+\nu,\nu,-1+\nu):1<\nu<2\}.$

\bigskip

\noindent$\mathbf{A_2+3A_1}:$ The infinitesimal character is
$(\frac 12,-\frac 12,-\frac 12,\frac
12)+\nu_1(2,1,1,0)+\nu_2(1,1,0,0),$ the lowest $W$-type is $1_3,$ and
$\fz=G_2.$ In terms of Langlands classification, the standard module
is induced from the Steinberg representation on $A_2,$ and as a
W-module, $X(A_2+3A_1)=\text{Ind}_{W(A_2)}^{W(F_4)}((111)).$

The restrictions of nearby $W$-types are:

\noindent\begin{tabular}{llllll}
Nilpotent &$A_23A_1$ &$A_32A_1$ &$A_3A_2A_1$ &$D_4A_1$ &$A_5A_1$\\
$W$-type &$1_3$ &$8_3$ &$4_4$ &$9_3$ &$6_1$\\
Multiplicity &$1$ &$1$ &$2$ &$3$ &$1$\\
$A_2\subset B_3$ &$111\times 0$ &$11\times 1$ &$111\times 0$ &$111\times 0,1\times 11$ &$11\times 1$\\
&&&$11\times 1$ &$11\times 1$\\
$A_2\subset A_2+\wti A_1$ &$(2)$ &$(2)$ &$(2),(11)$ &$2\cdot(2),(11)$ &$(11)$\\
$W(\fz)$ &$(1,0)$ &$(1,3)'$ &$(2,1)$ &$(1,0)+(2,2)$ &$(1,6)$
\end{tabular}  

We need a calculation in the Hecke algebra $\bH(B_3,1,2)$. The
infinitesimal character is $(\nu+1,\nu,-1+\nu)$, and the standard
module is induced from the Steinberg representation on $A_2.$ The
reducibility points are $1,2,3,$ and the relevant operators are
$$111\times 0: 1,\quad 11\times 1:\frac {1-\nu}{1+\nu},\quad 1\times
11: \frac{(1-\nu)(2-\nu)}{(1+\nu)(2+\nu)}.$$

Then the reducibility lines for $X(A_2+3A_1,(\nu_1,\nu_2))$ are
$\nu_1,\nu_1+\nu_2,2\nu_1+\nu_2=1,2,3,$ and
$\nu_2,3\nu_1+\nu_2,3\nu_1+2\nu_2=2.$ The matching of operators is
with $\bH(G_2,1,2),$ the Hecke algebra of type $G_2$:
1$<$$\equiv$2. From the tables, it follows that we are able to match
only the operators on $(1,0)$, $(1,3)'$, and $(2,1)$ in $G_2$. We will
also use the operator on $9_3$ (which is the closest one to match
$(2,2)$). The spherical unitary dual of $\bH(G_2,1,2)$ consists of the
closed region $\{2\nu_1+\nu_2\le 1\}$, and the isolated points $(\frac
12,\frac 12)$ and $(1,2).$ The matched operators are sufficient to
conclude that, when irreducible,
$X(A_2+3A_1,(\nu_1,\nu_2))$ is  unitary only in the
region $2\nu_1+\nu_2<1$, but on the lines we need more information.

On the line $\nu_1=-1+\nu,\nu_2=2,$ the decomposition is 
\begin{align}\notag
X(A_2+3A_1,(-1+\nu,2))&=\overline
X(A_2+3A_1,(-1+\nu,2))+X(A_3+A_2+A_1,\nu)\\\notag
\overline X(A_2+3A_1,(-1+\nu,2))|_W&=1_3+8_3+4_4+6_2+16_1+2\cdot
8_2+2\cdot 9_3\\\notag
&+12_1+9_4+2_2+4_5.
\end{align}
To simplify notation, set $X_1(\nu)=X(A_2+3A_1,(-1+\nu,2)).$
The matched operators give $$8_3: -\frac {(\frac
  12-\nu)(2-\nu)}{(\frac 12+\nu)(2+\nu)},\quad 4_4: \frac {2-\nu}{2+\nu},$$
showing that $X_1(\nu)$ can only be unitary for $\frac 12\le\nu\le 2.$
The operator on $9_3$ has two nonzero eigenvalues with product $-\frac
  {(\frac 12-\nu)(1-\nu)(2-\nu)^2(3-\nu)}{(\frac
  12+\nu)(1+\nu)(2+\nu)^2(3+\nu)},$ so the only segment which can be
  unitary is $\frac 12\le\nu\le 1.$ At $\nu=\frac 12,$ $X_1(\nu)$
  decomposes further:
\begin{align}\notag
&X_1(\frac 12)=\overline X(A_2+3A_1,(\frac 12,\frac 12))+\overline
X(A_3+2A_1,(\frac 12,\frac 12))\\\notag
&\overline X(A_2+3A_1,(\frac 12,\frac 12))|_W=1_3+4_4+8_2+9_3+2_2.
\end{align}
Since $X(A_2+3A_1,(\frac 12,\frac 12))$ is the $IM$-dual of
$X(D_4(a_1)A_1,\phi_2)$ (which is tempered), it must be unitary. Also
the other factor is unitary, because $\overline
X(A_3+2A_1,(\frac 12,\frac 12))$ is. Since $9_2$ appears in both
factors, and the operator on it is positive for $\frac 12<\nu<1,$ it
follows that $X_1(\nu)$ is unitary for $\frac 12<\nu<1.$ Note also
that $X_1(\nu)$ is the $IM$-dual of $\overline
X((A_3+A_1)'',1+2\nu,2\nu,-1+2\nu),$ which confirms its unitarity.

\medskip

On the line $\nu_1=1,\nu_2=-\frac 32+\nu,$ the decomposition is 
\begin{align}\notag
X(A_2+3A_1,(1,-\frac 32+\nu))&=\overline X(A_2+3A_1,(1,-\frac
32+\nu))+\overline X(A_3+2A_1,(\frac 12,\nu))\\
\overline X(A_2+3A_1,(1,-\frac 32+\nu))|_W&=1_3+4_4+8_2+9_3+2_2.
\end{align}
Denote $X_2(\nu)=\overline X(A_2+3A_1,(1,-\frac 32+\nu)).$ The
intertwining operators calculations give us $$4_4:\frac {\frac
  72-\nu}{\frac 72+\nu},\quad 9_3:\frac {(\frac 52-\nu)(\frac
  72-\nu)}{(\frac 52+\nu)(\frac 72+\nu)}.$$ So the only interval on
which $X_2(\nu)$ can be unitary is $0\le\nu\le\frac 52.$ But on this
interval, $X_2(\nu)$ is the $IM$-dual of
$X(D_4(a_1)A_1,(\nu,\nu),\phi_2)$, so it is in fact unitary.
Note that at the endpoint, $\nu=\frac 52$, $\overline
X(A_2+3A_1,(1,1))$ is the $IM$-dual of the discrete series
$X(E_7(a_4),\phi_2).$

\medskip

Finally, at the point $(1,2)$, $\overline X(A_2+3A_1,(1,2))$ is the
$IM$-dual of the discrete series $X(E_7(a_3),\phi_2).$

\medskip

In conclusion, the unitary set $\CU(A_2+3A_1)$ is the one pictured in
the figure \ref{fig:A23A1}. It is strictly larger than the spherical
unitary dual $\CS\CU(\bH(G_2,1,2)).$

\begin{figure}[h]
\input{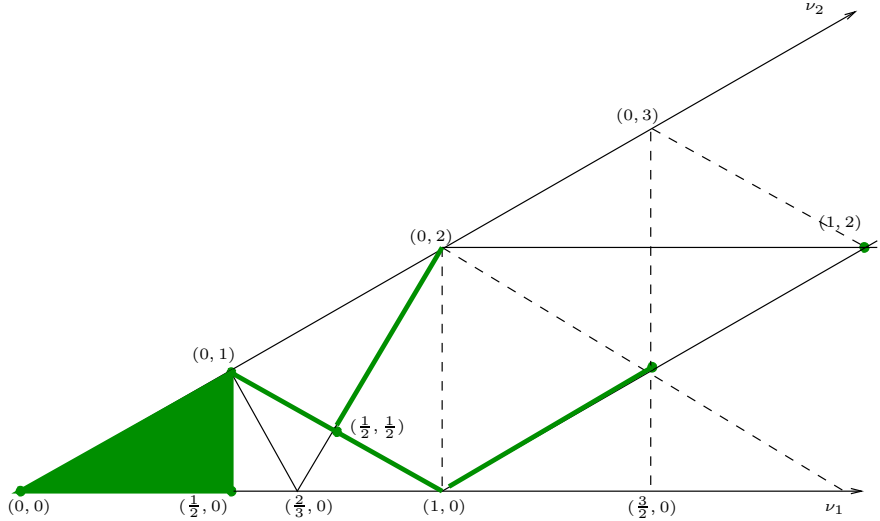}
\caption{Unitary parameters and reducibility lines for $A_2+3A_1$}\label{fig:A23A1}
\end{figure}

\bigskip

\noindent$\mathbf{4A_1}:$ The parameters are
$(\nu_1,\nu_2,\nu_3,\frac 12),$ the lowest $W$-type is $2_3$, and
$\fz=C_3.$ In terms of Langlands classification, the standard module
is induced from the Steinberg representation on $A_1,$ and as a
W-representation, $X(4A_1)=\text{Ind}_{W(A_1)}^{W(F_4)}((11)).$

The restrictions of the nearby $W$-types are:

\noindent\begin{tabular}{lllllll}
Nilpotent &$4A_1$ &$(A_3A_1)''$ &$A_23A_1$ &$A_32A_1$ &$D_4(a_1)A_1$
&$A_3A_2A_1$ \\
$W$-type &$2_3$ &$4_2$ &$1_3$ &$8_3$ &$9_1$ &$4_4$ \\
Multiplicity &$1$ &$1$ &$1$ &$4$ &$3$ &$3$ \\
$C_1\subset C_2$ &$0\times 2$ &$1\times 1$ &$0\times 2$ &$2(0\times 2)$ &$0\times
2$ &$2(0\times 2)$ \\
&&&&$2(1\times 1)$  &$2(1\times 1)$ &$1\times 1$\\
$A_1\subset A_2$ &$(21)$ &$(21)$ &$(1^3)$ &$3(21),(1^3)$ &$3(21)$
&$(21),2(1^3)$ \\
$\wti A_1$\\
$W(\fz)$ &$3\times 0$ &$0\times 3$ &$3\times 0$ &$0\times 3+2\times 1$
&$1\times 2$ &$2\times 1$ 
\end{tabular}

We need one calculation in the Hecke algebra $\bH(C_2,2,1)$ of type
$C_2$, 2$\Leftarrow$1. The standard module is induced from the
Steinberg representation on $C_1$, and it has infinitesimal character
$(\nu,\frac 12).$ The operators are
$$0\times 2: 1,\quad 1\times 1: \frac {\frac 32-\nu}{\frac
  32+\nu},\quad 0\times 11: \frac {(\frac 32-\nu)(\frac 52-\nu)}{(\frac 32+\nu)(\frac 52+\nu)}.$$

The matching of intertwining operators is with the spherical operators
for $\bH(C_3,2,3)$, the Hecke algebra of type $C_3$,
2--2$\Leftarrow$3. The only $W$-types which give matchings are

\begin{tabular}{ll}
$2_3$ &with $3\times 0$;\\
$4_2$ &with $0\times 3$;\\
$9_1$ &with $1\times 2$;\\
$8_1$ &with $0\times 12$.
\end{tabular}
 
The last one is irrelevant for the calculation. To get an inclusion
of $\CU(4A_1)$ into $\CS\CU(\bH(C_3,2,3))$, we would have needed to
find matchings for the $W(C_3)$-types $2\times 1$ and $12\times 0$ as
well. The $W(F_4)$-type $8_3$ fails to match $2\times 1$ because of
its restriction to $A_1\subset A_2,$ but we will have to use it in the
calculation nevertheless. We will also use the operators on $1_3$ and
$4_1.$ 

The hyperplanes of reducibility for $X(4A_1),(\nu_1,\nu_2,\nu_3))$
are: $\nu_i\pm\nu_j=2,$ $\nu_i=\frac 32$ (as in $\bH(C_3,2,3)$), but
also $\nu_i=\frac 52$, and $\nu_1\pm\nu_2\pm\nu_3=\frac 32.$ Especially
the second extra family of hyperplanes of reducibility will affect
the unitarity in an essential way.

The Hecke algebra $\bH(C_3,2,3)$ arises geometrically from a
cuspidal local system on $SL(3)^3$ in $Spin(13)$; its dual is parametrized by
a subset of the nilpotent orbits in type $B_6.$. We organize our
analysis of the unitary set $\CU(4A_1)$ by infinitesimal characters
corresponding to these orbits.

$(2^61):$ In $\bH(C_3,2,3),$ these are parameters
$(\nu_1,\nu_2,\nu_3)$ in the complement of the hyperplanes
$\nu_i\pm\nu_j=2,$ $\nu_i=\frac 32.$ In $\bH(C_3,2,3)$ the unitary
regions are:

(i) $\{\nu_1<\frac 32,\nu_1+\nu_2<2\}.$

(ii) $\{\nu_1<\frac 32,\nu_1+\nu_3>2,\nu_2+\nu_3<2\}.$

But the operators on the matched $W(C_3)$-types $3\times 0$, $0\times
3$, $1\times 2$ (equivalently, on the $W(F_4)$-operators $2_3$, $4_2$,
respectively $9_1$) are also positive in the region

(iii) $\{\nu_3<\frac 32,\nu_2>\frac
  32,\nu_1+\nu_3>2,\nu_1-\nu_3<2,\nu_2+\nu_3<2\}.$ 

In $F_4$, $X(4A_1)$ becomes reducible also on the hyperplanes
$\nu_1\pm\nu_2\pm\nu_3=\frac 32,$ some of which cut the regions
(i)--(iii). More specifically, region (i) is cut by the hyperplanes
$\nu_1+\nu_2+\nu_3=\frac 32$ and $\nu_1+\nu_2-\nu_3=\frac 32,$ and region
(ii) is cut by $\nu_1+\nu_2-\nu_3=\frac 32.$ We use the operators on
$1_2$ and $8_4$ and these are indefinite in all the resulting
(sub)regions except
\begin{equation}\label{eq:CS(4A1)}
\{\nu_1+\nu_2+\nu_3\le\frac 32\}\cup\{\nu_1+\nu_2-\nu_3\ge \frac
32,\nu_1<\frac 32,\nu_1+\nu_2<2\}.
\end{equation}
In these two remaining regions, one can deform $\nu_3$ to $0$, and
$X(4A_1)$ stays irreducible, and for $\nu_3=0$, it is unitarily
induced and unitary. So the parameters in (\ref{eq:CS(4A1)}) are
unitary. On the hyperplanes $\nu_1+\nu_2\pm\nu_3=\frac 32,$ in
(\ref{eq:CS(4A1)}), the decomposition is 
\begin{align}\notag
X(4A_1)=\overline X(4A_1)+X(A_2+3A_1),
\end{align}
and the factor $\overline X(4A_1)$ is self $IM$-dual.

$(52^4):$ The parameter is $(\nu_1,\nu_2,\frac 32).$ In
$\bH(C_3,2,3)$, the corresponding parameters are unitary for
$\{\nu_1<\frac 12\}\cup\{\nu_2>\frac 12,\nu_1+\nu<2\}.$ However, out of
  the matched operators, the only nontrivial one is on $9_1$, and it
  is $\frac {(\frac 72-\nu_1)(\frac 72-\nu_2)}{(\frac 72+\nu_1)(\frac
    72+\nu_2)}.$ We also need the operators on $1_3$ and $4_1$:
$$1_3:\left\{\begin{matrix}\frac
  {(3-(\nu_1-\nu_2))(3-(\nu_1+\nu_2))}{(3+(\nu_1-\nu_2))(3+(\nu_1+\nu_2))},
  &\nu_1\neq\nu_2\\ 0, &\nu_1=\nu_2\end{matrix}\right.,\ 4_1:
  \frac {(\frac 52-\nu_1)(\frac 52-\nu_2)(\frac 72-\nu_1)(\frac 72-\nu_2)}{(\frac 52+\nu_1)(\frac 52+\nu_2)(\frac 72+\nu_1)(\frac 72+\nu_2)}.$$
From these calculations, it follows that for $\nu_1\neq\nu_2,$ the
  only unitary parameters are in the region $\{\nu_1+\nu_2<
  3,\nu_1< \frac 52\}$ (and its closure). In this region, the
  decomposition of the standard module is
\begin{align}\notag
X(4A_1,(\nu_1,\nu_2,\frac 32))&=\overline X(4A_1,(\nu_1,\nu_2,\frac
32))+\overline X((A_3A_1)'',(\nu_1+\nu_2,\nu_1-\nu_2,1))\\\notag
&+X(D_4(a_1)A_1,(\nu_1,\nu_2)).
\end{align}
Moreover, $\overline X(4A_1,(\nu_1,\nu_2,\frac 32))$ is the $IM$-dual of
$X(D_4(a_1)+A_1),$ so it is indeed unitary.

When $\nu_1=\nu_2=\nu$, the decomposition is
\begin{align}\notag
X(4A_1,(\nu,\nu,\frac 32))&=\overline X(4A_1,(\nu,\nu,\frac
32))+\overline X((A_3A_1)'',(2\nu,1,0))\\\notag
&\overline X(A_2+3A_1,(1,-\frac 32+\nu))+X(D_4(a_1)A_1,(\nu,\nu)).
\end{align}
The factor $\overline X(4A_1,(\nu,\nu,\frac 32))$ is the $IM$-dual of
$X(D_4(a_1)A_1,(\nu,\nu),phi_1)$, so it is unitary if and only if
$0\le\nu\le\frac 32.$

$(44221):$ The parameters are $(\nu_1,1+\nu_2,-1+\nu_2).$ In
$\bH(C_3,2,3),$ they are unitary in the region $\{\nu_1<\frac
32,\nu_2<\frac 12\}.$ In $F_4$, the unitary set will be
different. Generically, here the standard module
decomposes as 
\begin{equation}\notag
X(4A_1,(\nu_1,1+\nu_2,-1+\nu_2))=\overline X(4A_1,(\nu_1,1+\nu_2,-1+\nu_2))+X(A_3+2A_1,(\nu_1,\nu_2)),
\end{equation}
and $\overline X(4A_1)$ is the $IM$-dual of $\overline
X((A_3+A_1)'',(\frac 12+\nu_1,-\frac 12+\nu_1,2\nu_2)),$ which is
unitary for $\{0\le\nu_1<\frac 12,0\le\nu_2<\frac 12\}.$

Similar calculations as in the previous case ($(52^4)$) show that the
parameters $(\nu_1,1+\nu_2,-1+\nu_2)$ can be unitary only for
$\{0\le\nu_1<\frac 12,0\le\nu_2<\frac 12\}$ (by the remark in the
previous paragraph, $\overline X(4A_1,(\nu_1,1+\nu_2,-1+\nu_2))$ has
to be unitary in this region), and on the segment
$\nu_1+2\nu_2=\frac 32,$ for $\frac 12\le\nu_1\le\frac 32.$ For this,
we use, in addition to the matched $W$-types, the operators on
$1_3,4_1,8_3.$

It remains to analyze the segment $\nu_1+2\nu_2=\frac 32.$ We rewrite
the parameters as $\nu_1=\frac 12+2\nu,\nu_2=\frac 12-\nu.$ Then the
factor $\overline X(4A_1,(\frac 12+2\nu,\frac 12+\nu,-\frac
32+\nu))$ is the $IM$-dual of $\overline X(A_3+2A_1,(-\frac 12+\nu,\frac
12+\nu))$, which is unitary for $0\le\nu<\frac 12$, which is precisely
the segment we were looking at. Therefore, this segment is also
unitary for $\overline X(4A_1).$

$(53221):$ The infinitesimal character is $(\nu,\frac 32,\frac 12).$
In $\bH(C_3,2,3)$, the corresponding spherical module is unitary for
$0\le\nu<\frac 32.$ 

In $F_4$, the only nonzero matched operator is on $9_1$, and it is
$\frac {\frac 72-\nu}{\frac 72+\nu}.$ The operator on $1_3$ is $\frac
{(\frac 52-\nu)(\frac 72-\nu)}{(\frac 52+\nu)(\frac 72+\nu)}.$ So it
remains to check the segment $0\le\nu\le\frac 52.$ For
$0\le\nu\le\frac 52,$ $\overline X(4A_1,(\nu,\frac 32,\frac 12))$ is
the $IM$-dual of $X(D_4(a_1)A_1,(\nu,\frac 12))$, so it is indeed unitary.

$(544):$ The infinitesimal character is $(1+\nu,-1+\nu,\frac 32).$ In
$\bH(C_3,2,3)$, the corresponding spherical module is unitary for
$0\le\nu<\frac 32.$ In $F_4$, the matched operator on $9_4$ gives $\frac {(\frac
  52-\nu)(\frac 92-\nu)}{(\frac
  52+\nu)(\frac 92+\nu)}.$ The operator on $1_2$ is $\frac {\frac
  32-\nu}{\frac 32+\nu},$ which implies the only unitary parameters
can be in $0\le\nu\le\frac 32.$ For $0\le\nu<\frac 32,$ $\overline
X(4A_1,(1+\nu,-1+\nu,\frac 32))$ is the $IM$-dual of
$X(D_4(a_1)A_1,(1+\nu,-1+\nu))$, so it is unitary.

$(661):$ The infinitesimal character is $(2+\nu,\nu,-2+\nu).$ In
$\bH(C_3,2,3)$, the corresponding spherical module is unitary for
$0\le\nu<\frac 12.$ In $F_4$, the matched operators give $4_5:\frac
{(\frac 12+\nu)(\frac 32-\nu)(\frac 72-\nu)}{(\frac 12-\nu)(\frac
  32+\nu)(\frac 72+\nu)},$ respectively $9_4:\frac {(\frac
  32-\nu)(\frac 72-\nu)}{ (\frac
  32+\nu)(\frac 72+\nu)}.$ For $0\le\nu<\frac 12,$ the factor
$\overline X(4A_1,(2+\nu,\nu,-2+\nu))$ is the $IM$-dual of
$X((A_5)'',(1,-\frac 12+\nu))$, therefore it is unitary.

$(751):$ The infinitesimal character is $(\frac 52,\frac 32,\frac
12)$, which is unitary in $\bH(C_3,2,3).$ In $F_4$, $\overline
X(4A_1,(\frac 52,\frac 32,\frac 12))$ is unitary as well, being the
$IM$-dual of $\overline X(D_6(a_1),\frac 12).$

$(922):$ The infinitesimal character is $(\nu,\frac 72,\frac 32),$
which in $\bH(C_3,2,3)$ is unitary for $0\le\nu<\frac 12.$ In $F_4$,
the matched operators are $0$, but the operators on $1_2$ and $8_4$ are $\frac
{(\frac 12+\nu)(\frac {13}2-\nu)}{(\frac 12-\nu)(\frac {13}2+\nu)}$,
respectively $\frac {\frac {13}2-\nu}{\frac {13}2+\nu}.$ 

For $0\le\nu<\frac 12$, the factor $\overline X(4A_1,(\nu,\frac
72,\frac 32))$ is the $IM$-dual of $X(D_6(a_1),\nu)$, and therefore
unitary. But also the point $\nu=\frac {13}2$ is unitary, since
$\overline X(4A_1,(\frac {13}2,\frac 72,\frac 32))$ is the $IM$-dual of
$X(E_7(a_1))$. (This point has no correspondent in $\bH(C_3,2,3)$.) 

$(931):$ The infinitesimal character is $(\frac 72,\frac 32,\frac
12)$, which is unitary in $\bH(C_3,2,3).$ It is also unitary in $F_4$,
as $\overline X(4A_1,(\frac 72,\frac 32,\frac 12))$ is the $IM$-dual of
$\overline X(D_6(a_1),\frac 12).$

$(13):$ The infinitesimal character is $(\frac {11}2,\frac 72,\frac
32).$ In $\bH(C_3,2,3)$ the corresponding spherical module is the
trivial representation. But $\overline X(4A_1,(\frac {11}2,\frac 72,\frac
32))$ is not unitary as seen from the operators for $(922)$ at
$\nu=\frac {11}2.$

\subsection{Spherical modules}\label{sec:5.2}

The spherical modules are parameterized by the nilpotent orbit
$(3A_1)''.$ If a spherical module does not contain the sign
$W$-representation, via $IM$ its unitarity was already determined in
the previous section. We record those results next.

\begin{small}
\begin{center}
\begin{longtable}{|c|c|c|}
\caption{Spherical unitary modules for $\bH(F_4,(1,2))$}
\label{table:5.2}\\
\hline
\multicolumn{1}{|c|}{\bf Nilpotent} &\multicolumn{1}{c|}{\bf Central
  character} &\multicolumn{1}{c|}{\bf Unitary}\\\hline
\endfirsthead
\multicolumn{3}{c}%
{{ \tablename\ \thetable{} -- continued from previous page}}
\\
\endhead
\hline\hline
\endlastfoot

$E_7$ &$(\frac {17}2,\frac 92,\frac 52,\frac 12)$ & \\
\hline
$E_7(a_2)$ &$(\frac {11}2,\frac 52,\frac 32,\frac 12)$ & \\
\hline
$D_6$ &$(\nu,\frac 92,\frac 52,\frac 12)$ &$\{0\le\nu\le\frac 12\}$\\
\hline
$D_5+A_1$ &$(2+\nu,-2+\nu,\frac 52,\frac 12)$ &$\{0\le\nu\le\frac 12\}$\\
\hline
$E_7(a_5)$ &$(\frac 52,\frac 32,\frac 12,\frac 12)$ &\\
\hline
$D_6(a_2)$ &$(\nu,\frac 52,\frac 32,\frac 12)$ &$\{0\le\nu<\frac 12\}$\\
\hline
$A_5+A_1$ &$(\frac 14,\frac 74,-\frac 14,-\frac 94)+\nu(\frac 32,\frac
12,\frac 12,\frac 12)$ &$\{0\le\nu<\frac 12\}$\\
\hline
$(A_5)"$ &$(\nu_2+\frac {3\nu_1}2+2+\frac {\nu_1}2,\frac
{\nu_1}2,-2+\frac {\nu_1}2)$
&$\{3\nu_1+2\nu_2<1\}$ \\
&&$\{2\nu_1+\nu_2<1<3\nu_1+\nu_2\}$ \\
\hline
$D_4+A_1$ &$(\nu_1,\nu_2,\frac 52,\frac 12)$
&$\{\nu_1+\nu_2<1\}$\\
&&$\{\nu_1-\nu_2>1,\nu_1<\frac 32\}$\\
    &$(\frac 12+\nu,-\frac 12+\nu,\frac 52,\frac 12)$ &$\{0\le\nu<\frac
    32\}$\\
\hline
$A_3A_2A_1$ &$(\frac 12,\frac 12,-\frac 32,\frac 12)+\nu(2,1,1,0)$
&$\{0\le\nu<1\}$\\
\hline
$A_3+2A_1$ &$(\nu_1,1+\nu_2,-1+\nu_2,\frac 12)$ &$\{0\le\nu_1<\frac
12,\nu_1+2\nu_2<\frac 32\}$\\
&&$\{0\le\nu_1<\frac 12,0\le\nu_2<\frac 32,2\nu_2-\nu_1>\frac 32\}$ \\
   &$(1+\nu,-1+\nu,\frac 12,\frac 12)$ &$\{0\le\nu_1<\frac 12\}\cup
   \{1<\nu<\frac 32\}$ \\
\hline
$(A_3+A_1)"$ &$(\frac {\nu_1+\nu_2}2,\frac {\nu_1-\nu_2}2,1+\frac
{\nu_3}2,-1+\frac{\nu_3}2)$
&$\{0\le\nu_3\le\nu_2\le\nu_1<1-\nu_2\}$\\
&&$\{0\le\nu_3\le\nu_2<1-\nu_3<\nu_1<1\}$ \\ 
&$(\nu_1,\nu_2,\frac 32,\frac 12)$ &$\{0\le\nu_2\le\nu_1<\frac 12\}$\\
\hline
$A_2+3A_1$ &$(\frac 12,-\frac 12,-\frac 12,\frac
12)+\nu_1(2,1,1,0)$ &$\{2\nu_1+\nu_2<1\}$\\ 
& $+\nu_2(1,1,0,0)$ &\\
\hline
$4A_1$ &$(\nu_1,\nu_2,\nu_3,\frac 12)$ &$\{\nu_1+\nu_2+\nu_3<\frac
32\}$\\
&&$\{\nu_1+\nu_2-\nu_3>\frac 32,\nu_1<\frac 32,\nu_1+\nu_2<2\}$\\
\hline
$(3A_1)''$ &$(\nu_1,\nu_2,\nu_3,\nu_4)$ &regions $\C F_1-\C F_5$ in
section \ref{sec:5.5}\\
\hline
\end{longtable}
\end{center}
\end{small}

\subsection{The $0$-complementary series}\label{sec:5.5} In this section we determine
the unitary irreducible spherical series $X(\nu).$ The parameter
$\nu=(\nu_1,\nu_2,\nu_3,\nu_4)$ is assumed in the dominant
Weyl chamber $\C C.$ This is partitioned by the hyperplanes 
\begin{equation}\label{eq:5.3.1}
\langle\al_l,\nu\rangle=1,\ \al_l \text{ long root and
}\langle\al_s,\nu\rangle=c, \al_s \text{ short root.}
\end{equation} 
We assume first that $c>1$ is arbitrary.  The principal series
$X(\nu)$ is reducible precisely when $\nu$ is on one of the
hyperplanes in (\ref{eq:5.3.1}).  If $\C F$ is an open connected component
of the complement of (\ref{eq:5.3.1}) in $\C C$ (we call $\C F$ a {\it
  region}), then all the intertwining operators $r_\sigma(\nu),$ $\sigma\in
\widehat W$, are invertible, and therefore have constant signature in
$\C F$. We
say that the region $\C F$ is unitary if $X(\nu)$ is unitary for all
(equivalently, any) $\nu\in \C F.$ The walls of any region $\C F$ are
of the form (\ref{eq:5.3.1}), or of the form
$\langle\al,\nu\rangle=0,$ for $\al\in\Pi.$ 

\begin{proposition}\label{p:5.3}
Consider the half-space $K=\{\nu:\langle\ep_1+\ep_2,\nu\rangle<c\}.$
The unitary regions $\C F$ in $\C C\cap K$ are:

\begin{enumerate}

\item[($\C F_1$)] $\{2\nu_1<1,\ \nu_1+\nu_2<c\};$

\item[($\C F_2$)] $\{\nu_1+\nu_2+\nu_3+\nu_4>1,\ \nu_1+\nu_2+\nu_3-\nu_4<1,\
  \nu_1+\nu_2<c\};$ 

\item[($\C F_3$)] $\{\nu_1+\nu_2-\nu_3+\nu_4>1,\ \nu_1-\nu_2+\nu_3+\nu_4<1,\
  \nu_1+\nu_2-\nu_3-\nu_4<1,\ \nu_1+\nu_2<c\};$

\item[($\C F_4$)] $\{\nu_1-\nu_2+\nu_3+\nu_4<1,\ 2\nu_2>1,\ \nu_1+\nu_2<c\};$

\item[($\C F_5$)] $\{2\nu_2>1,\ \nu_1-\nu_2+\nu_3-\nu_4>1,\ 2\nu_3<1,\
  \nu_1-\nu_2-\nu_3+\nu_4<1,\ \nu_1+\nu_2<c\};$

\item[($\C F_6$)]  $\{2\nu_2>1,\ \nu_1-\nu_2-\nu_3-\nu_4>1,\ 2\nu_3<1,\ \nu_1+\nu_2<c\};$

\item[($\C F_7$)] $\{\nu_1-\nu_2-\nu_3-\nu_4>1,\ 2\nu_4>1,\ \nu_1+\nu_2<c\}.$

\end{enumerate}

\end{proposition}

\begin{proof}

The proof is a case by case analysis, which we sketch here.

Note that $\ep_1+\ep_2$ is the highest short root of $F_4.$ The
condition that $\C F$ be in the half-space $K$ means that the walls of
$\C F$ can only be of the form $\langle\al_l,\nu\rangle=1,$
$\langle\al,\nu\rangle=0$ (or $\nu_1+\nu_2=c$). From the partial order
relation among the (long) roots, we see that there are $19$ such
regions. A case by case analysis gives that each of these regions has
a wall of the form $\langle\al,\nu\rangle=0.$  

If $\C F$ has a wall of
the form $\langle\al,\nu\rangle=0,$ then on this wall, $X(\nu)$ is
unitarily induced irreducibly from a principal series $X_M(\nu')$ for
a Hecke subalgebra $H_M,$ with $M$ Levi of type $B_3$ or $C_3.$ In
\cite{BC}, theorem 3.4 and 3.6, the unitary irreducible spherical
principal series for the Hecke algebras of type
$B_n/C_n$ with arbitrary unequal parameters are determined, so in
particular we know the unitarity of $X_M(\nu'),$ and therefore of
$X(\nu).$ 

\end{proof}

One of the results of \cite{BC}, when the Hecke algebra $\bH$ if of type $B_n$ with
arbitrary parameter $c$, is that any region $\C F$ on which the
highest short root is greater than $c$, is necessarily not
unitary. The similar statement is false in $F_4,$ as seen in the following example.

\begin{example}
The region $\C F_8=\{\nu_1-\nu_2-\nu_3-\nu_4>1,\ 2\nu_4>1,\
\nu_1+\nu_3>c,\ \nu_1+\nu_4<c\}$ is unitary.
\end{example}

\begin{proof}
The region $\C F_8$ has a wall given by
$\langle\ep_2-\ep_3,\nu\rangle=0.$ We deform the parameter $\nu$ to
this wall, \ie\ $\nu_2=\nu_3$. The corresponding module $X(\nu)$ is
unitarily induced irreducible from the principal series
$X_{B_3}(\nu'),$ where $\nu'=(\nu_1',\nu_2',\nu_3')$ satisfy
$\nu_1'-\nu_2'>1,$ $\nu_2'-\nu_3'>1,$ and $\nu_1'<c.$ These parameters
$\nu'$ are unitary in $B_3$, \cf \cite{BC}, theorem 3.6. 
\end{proof}

In order to show that  the regions $\C F_1-\C F_8$ are the {\it only}
unitary regions, we determined by brute force computer calculations
the signatures of $r_\sigma(\nu)$, with $\sigma\in \{1_1,4_2,9_1\}$,
and $\nu\in \C F$, for $\C F$ any region in $\C C.$

\begin{corollary}
In the case of the geometric Hecke algebra $\bH(F_4,(1,2))$, the
$0$-complementary series is given by the regions $\C F_1-\C F_5$ from
proposition \ref{p:5.3}.
\end{corollary}

(When $c=2,$ the regions $\C F_6-\C F_8$, in the notation as before,
become empty.) 

\subsection{$W$-structure}\label{green} We record the $W$-structure of standard
modules for $\bH(F_4,(1,2))$ as follows from the intertwining operator
calculations of previous sections. The correspondence is normalized so
that the nilpotent $E_7$ corresponds to the sign representation $1_4$,
and $(3A_1)''$ to the trivial $1_1$. For induced representations and
labeling of irreducible characters of
the Weyl group of $F_4$, see for example \cite{A}.

\begin{small}
\begin{center}
\begin{longtable}{|c|c|c|}
\caption{$W$-structure of standard modules of $\bH(F_4,(1,2))$}
\label{table:$W$-structure}\\
\hline
\multicolumn{1}{|c|}{$\mathbf\CO$} &\multicolumn{1}{c|}{$\mathbf
 {\mu(\CO,\phi)}$} &\multicolumn{1}{c|}{$W$-structure}\\
\hline 
\endfirsthead
\multicolumn{3}{c}%
{{ \tablename\ \thetable{} -- continued from previous page}}
\\
\endhead
\hline\hline
\endlastfoot

${E_7}$ &$1_4$ & $1_4$\\

\hline
${E_7(a_1)}$ &$2_4$ &$2_4$ \\

\hline
${E_7(a_2)}$ &$4_5$ &$4_5+1_4$ \\

\hline

${E_7(a_3)}$ &$8_4$ &$8_4+2_4$\\
${E_7(a_3)}$ &$1_2$ &$1_2$\\

\hline

$D_6$ &$9_4$ & $\Ind_{C_3}^{F_4}(0\times 111)$\\

\hline

$E_7(a_4)$ &$4_3$ &$4_3+1_2$\\
           &$2_2$ &$2_2$\\

\hline

$D_6(a_1)$ &$9_2$ &$\Ind_{C_3}^{F_4}(111\times 0)$\\

\hline

$D_5+A_1$ &$8_2$ &$\Ind_{B_3}^{F_4}(0\times 111)$\\

\hline

$E_7(a_5)$ &$12_1$ &$12_1+8_4+4_5+8_2+9_4+1_4$\\
           &$6_2$  &$6_2+4_5$\\

\hline

$D_6(a_2)$ &$16_1$ &$\Ind_{C_3}^{F_4}(1\times 11+0\times 111)$\\

\hline

$A_5+A_1$ &$6_1$ &$\Ind_{\wti A_2+A_1}^{F_4}((111)\otimes (11))$\\

\hline

$D_5(a_1)+A_1$ &$4_1$ &$\Ind_{B_3}^{F_4}(0\times 12)$\\

\hline

$(A_5)''$ &$8_1$ &$\Ind_{\wti A_2}^{F_4}((111))$\\

\hline

$D_4+A_1$ &$9_3$     &$\Ind_{C_2}^{F_4}(0\times 11)$\\

\hline

$A_3A_2A_1$ &$4_4$   &$\Ind_{A_2+\wti A_1}^{F_4}((111)\otimes (11))$ \\

\hline

$D_4(a_1)+A_1$ &$9_1$  &$\Ind_{B_3}^{F_4}(1\times 2+0\times 12)$\\

             &$2_1$   &$\Ind_{B_3}^{F_4}(0\times 3)$ \\

\hline

$A_3+2A_1$ &$8_3$     &$\Ind_{A_1+\wti A_1}^{F_4}((11)\otimes (11))$ \\

\hline

$(A_3+A_1)''$ &$4_2$   &$\Ind_{\wti A_1}^{F_4}((11))$ \\

\hline

$A_2+3A_1$    &$1_3$   &$\Ind_{A_2}^{F_4}((111))$ \\
\hline

$4A_1$       &$2_3$    &$\Ind_{A_1}^{F_4}((11))$ \\

\hline

$(3A_1)''$   &$1_1$    &$\Ind_{1}^{F_4}((1))$ \\

\hline

\end{longtable}
\end{center}
\end{small} 

As mentioned in the introduction, we verified that the dimensions of
the discrete series modules above agree with the dimensions obtained
from the weight spaces in \cite{Re1}. In \cite{Re1}, one uses the
realization of this algebra in $E_8$, and employs the notation of
\cite{Sp2} for nilpotent orbits, so to compare the two lists, one
needs to add an ``$A_1$'' to our notation of nilpotent orbits.



\end{document}